\documentclass[12pt,reqno]{amsart}
\usepackage[margin=1in]{geometry}
\usepackage{amsmath,amssymb,amsthm,graphicx,amsxtra, setspace}
\usepackage[utf8]{inputenc}
\usepackage{mathrsfs}
\usepackage{hyperref}
\usepackage{upgreek}
\usepackage{mathtools}
\allowdisplaybreaks

\usepackage[pagewise]{lineno}

\newtheorem{theorem}{Theorem}[section]
\newtheorem{lemma}[theorem]{Lemma}
\newtheorem{proposition}[theorem]{Proposition}
\newtheorem{assumption}[theorem]{Assumption}

\newtheorem{definition}[theorem]{Definition}

\newtheorem{remark}[theorem]{Remark}
\newtheorem{hypothesis}[theorem]{Hypothesis}

\let\originalleft\left
\let\originalright\right
\renewcommand{\left}{\mathopen{}\mathclose\bgroup\originalleft}
\renewcommand{\right}{\aftergroup\egroup\originalright}

\newcommand{\Tr}{\mathop{\mathrm{Tr}}}
\renewcommand{\d}{\/\mathrm{d}\/}

\def\w{\textbf{W}^{\varepsilon}_{{\theta}^{\varepsilon}}}

\def\e{\varepsilon}

\def\S{\mathrm{S}}
\def\T{T\wedge\tau_R^{\e}}
\def\L{\mathbb{L}}
\def\A{\mathrm{A}}
\def\U{\mathbf{U}}
\def\I{\mathrm{I}}
\def\F{\mathrm{F}}
\def\C{\mathrm{C}}
\def\f{\mathbf{f}}

\def\B{\mathrm{B}}
\def\D{\mathrm{D}}
\def\y{\mathbf{y}}

\def\Y{\mathbf{Y}}
\def\Z{\mathbf{Z}}
\def\E{\mathbb{E}}
\def\X{\mathbf{X}}
\def\x{\mathbf{x}}

\def\z{\mathbf{z}}
\def\v{\mathbf{v}}
\def\V{\mathbb{v}}
\def\w{\mathbf{w}}
\def\W{\mathrm{W}}
\def\G{\mathrm{G}}
\def\Q{\mathrm{Q}}

\def\N{\mathbb{N}}

\def\V{\mathbb{V}}
\def\wi{\widetilde}
\def\Q{\mathrm{Q}}
\def\u{\mathrm{U}}
\def\P{\mathrm{P}}
\def\u{\mathbf{u}}
\def\H{\mathbb{H}}

\newcommand{\R}{\mathbb{R}}

\renewcommand{\d}{\/\mathrm{d}\/}

\newcommand{\Addresses}{{
		\footnote{
			
			\noindent \textsuperscript{1}Department of Mathematics, Indian Institute of Technology Roorkee-IIT Roorkee,
			Haridwar Highway, Roorkee, Uttarakhand 247667, INDIA.\par\nopagebreak
			\noindent  \textit{e-mail:} \texttt{manilfma@iitr.ac.in, maniltmohan@gmail.com.}
			
			\noindent \textsuperscript{*}Corresponding author.

			\textit{Key words:} convective Brinkman-Forchheimer equations; averaging principle; invariant measure; large deviation principle; weak convergence; Khasminkii's time discretization.
			
			Mathematics Subject Classification (2010): Primary 60H15; Secondary 35R60, 35Q30, 70K70.

}}}

\begin{document}
	
	
	\title[Large deviations for the two-time-scale SCBF equations]{Large deviations  for the two-time-scale stochastic convective Brinkman-Forchheimer equations
			\Addresses}
	\author[M. T. Mohan ]{Manil T. Mohan\textsuperscript{1*}}

	\maketitle
	
	\begin{abstract}
The convective Brinkman-Forchheimer (CBF) equations characterize the motion of incompressible fluid flows in a saturated porous medium. The small noise asymptotic for the two-time-scale stochastic convective Brinkman-Forchheimer (SCBF) equations in two and three dimensional bounded domains is carried out in this work.  More precisely, we establish a Wentzell-Freidlin type large deviation principle for stochastic partial differential equations with slow and fast time-scales, where the slow component is the SCBF equations in two and three dimensions perturbed by small multiplicative Gaussian noise and the fast component is a stochastic reaction-diffusion equation with damping. The results are obtained by using a variational method (based on weak convergence approach) developed by Budhiraja and Dupuis, Khasminkii's time discretization approach and stopping time arguments. In particular, the results obtained from this work are true for two dimensional stochastic Navier-Stokes equations also. 
	\end{abstract}

	\section{Introduction}\label{sec1}\setcounter{equation}{0}
The  convective Brinkman-Forchheimer (CBF) equations  describe the motion of  incompressible viscous fluid  through a rigid, homogeneous, isotropic, porous medium. Let us first provide a mathematical formulation of the CBF equations.	Let $\mathcal{O}\subset\R^n$ ($n=2,3$) be a bounded domain with a smooth boundary $\partial\mathcal{O}$. Let  $\X_t(x) \in \R^n$ denotes the velocity field at time $t\in[0,T]$ and position $x\in\mathcal{O}$, $p_t(x)\in\R$ represents the pressure field, $\f_t(x)\in\R^n$ stands for an external forcing.   The	CBF  equations  are given by 
\begin{equation}\label{1}
\left\{
\begin{aligned}
\frac{\partial \X_t}{\partial t}-\mu \Delta\X_t+(\X_t\cdot\nabla)\X_t+\alpha\X_t+\beta|\X_t|^{r-1}\X_t+\nabla p_t&=\mathbf{f}_t, \ \text{ in } \ \mathcal{O}\times(0,T), \\ \nabla\cdot\X_t&=0, \ \text{ in } \ \mathcal{O}\times(0,T), \\
\X_t&=\mathbf{0},\ \text{ on } \ \partial\mathcal{O}\times(0,T), \\
\X_0&=\x, \ \text{ in } \ \mathcal{O},
\end{aligned}
\right.
\end{equation}
the constants $\mu>0$ represents the  Brinkman coefficient (effective viscosity),  $\alpha>0$ stands for the Darcy (permeability of porous medium) coefficient and $\beta>0$ denotes the Forchheimer (proportional to the porosity of the material) coefficient.  In order to obtain the uniqueness of the pressure $p$, we can impose the condition $ \int_{\mathcal{O}}p_t(x)\d x=0, $ for $t\in (0,T)$ also. The absorption exponent $r\in[1,\infty)$ and the case $r=3$ is known as the critical exponent.  Note that for the case $\alpha=\beta=0$, we get the classical 3D Navier-Stokes equations (see \cite{GGP,OAL,JCR3,Te,Te1}, etc). The works \cite{SNA,CLF,KT2,MTM7}, etc discuss the global solvability results (existence and uniqueness of weak as well as strong solutions) for the deterministic CBF equations in bounded domains. As in the case of classical 3D Navier-Stokes equations, the existence of a unique global strong solution for the CBF  equations \eqref{1} for $n=3$ and $r\in[1,3)$ is an open problem.

In the stochastic counterpart, the existence and uniqueness of strong solutions to the stochastic 3D tamed Navier-Stokes equations on bounded domains with Dirichlet boundary conditions is established in \cite{MRTZ}. The authors in  \cite{LHGH1} obtained the existence of martingale solutions for the stochastic 3D Navier-Stokes equations with nonlinear damping.  The existence of a pathwise unique strong solution  satisfying the energy equality (It\^o's formula) to the stochastic convective Brinkman-Forchheimer (SCBF) equations perturbed by multiplicative Gaussian noise is proved in \cite{MTM8}. The author exploited the monotonicity and hemicontinuity properties of the linear and nonlinear operators as well as a stochastic generalization of  the Minty-Browder technique in the proofs. The It\^o formula (energy equality) is established by using the fact that there are functions that can approximate functions defined on smooth bounded domains by elements of eigenspaces of linear operators (e.g., the Laplacian or the Stokes operator) in such a way that the approximations are bounded and converge in both Sobolev and Lebesgue spaces simultaneously (such a construction is available in \cite{CLF}). Making use of the exponential stability of strong solutions, the existence of a unique ergodic and strongly mixing invariant measure for the SCBF  equations subject to multiplicative Gaussian noise is also established in  \cite{MTM8}. The works \cite{HBAM,ZBGD,WLMR,WL,MRXZ1,MRTZ1}, etc discuss various results on the stochastic tamed 3D Navier-Stokes equations and related models on periodic domains as well as on the whole space. 

	The multiscale systems involve slow and fast components in mathematical models, and they are having  wide range of  applications in the areas like signal processing, climate dynamics, material science,   molecular dynamics, mathematical finance, fluid dynamics, etc. The theory of averaging principle for multiscale systems is well developed for the past several years and it has extensive  applications in science, technology and engineering (cf. \cite{RBJE,WEBE,FW1,EHVK,EAMEL,MMMC,FWTT}, etc and references therein).  Averaging principle  proposes that the slow component of a slow-fast system may be approximated by a simpler system obtained by averaging over the fast motion. For the deterministic systems, an averaging principle was first investigated by Bogoliubov and Mitropolsky in \cite{NNYA} and for the stochastic differential equations, an averaging principle was first studied by Khasminskii  in  \cite{RZK}. Several works are available in the literature for the theory of averaging principle for the stochastic partial differential equations (cf. \cite{CEB,CEB1,SC,SC1,SC2,SC3,ZDXS,HFJD,HFJL,HFLW,HFLW2,HFLW1,WWAJ,JXJL}, etc and references therein). An averaging principle  for the slow component as two dimensional stochastic Navier-Stokes equations  and the fast component as stochastic reaction-diffusion equations by using the classical Khasminskii approach based on time discretization is established  in \cite{SLXS}. The authors  in \cite{WLMRX} established a strong averaging principle for the slow-fast stochastic partial differential equations with locally monotone coefficients, which includes the  systems like stochastic porous medium equation, the stochastic Burgers type equation, the stochastic $p$-Laplace equation and the stochastic 2D Navier-Stokes equation, etc. A strong averaging principle for the stochastic convective Brinkman-Forchheimer (SCBF)  equations perturbed by Gaussian noise, in which the fast time scale component is governed by a stochastic reaction-diffusion equation with damping driven by multiplicative Gaussian noise,  has been obtained in \cite{MTM11}. 

The large deviations theory, which  concerns the asymptotic behavior of remote tails of sequences of probability distributions (cf. \cite{FW,Va}), is one of the classical areas in probability theory with numerous deep developments and variety of applications in the fields like queuing theory, statistics, finance, engineering, etc. The theory of large deviations explains the probabilities of rare events that are exponentially small as a function of some parameter. In the case of stochastic differential equations, this parameter can be regarded as  the amplitude of the noise perturbing a dynamical system.  The Wentzell-Freidlin type large deviation estimates for a class of infinite dimensional stochastic differential equations is developed in the works \cite{BD1,chow,DaZ,KX}, etc. Large deviation principles for the 2D stochastic Navier-Stokes equations driven by Gaussian noise are established in the works \cite{ICAM,MGo,SSSP}, etc. A Wentzell-Freidlin type large deviation principle  for the stochastic tamed 3D Navier-Stokes equations driven by multiplicative Gaussian noise  in the whole space or on a torus  is established in \cite{MRTZ1}. Small time large deviations principles  for the stochastic 3D tamed Navier-Stokes equations in bounded domains is established in the work \cite{RWW}. Large deviation principle for the 3D tamed Navier-Stokes equations driven by multiplicative L\'evy noise in periodic domains is established in \cite{HGHL}. The author in \cite{MTM10} obtained the Wentzell-Freidlin  large deviation principle for the two and three dimensional SCBF equations in bounded domains using a weak convergence approach developed by Budhiraja and Dupuis (see, \cite{BD1,BD2}). The large deviations for short time as well as the exponential estimates on certain exit times associated with the solution  trajectory of the  SCBF equations are also established in the same work.  It seems to the author that some of the LDP results available in the literature for the 3D stochastic tamed Navier-Stokes equations in bounded domains are not valid due to the technical difficulty described in the works \cite{CLF,KT2,MTM7}, etc.

The large deviation theory for multi-scale systems are also well studied in the literature (see for instance, \cite{PDKS,FW,HJK,RKLP,RLi,TJL,LPo,AAP,KSp,AYV1}, etc and references therein). A large deviation principle for a class of stochastic reaction-diffusion  partial  differential  equations  with  slow-fast components is derived in the work \cite{WWAJ1} .  The authors in \cite{WHMS} studied a large deviation principle for a system of stochastic reaction-diffusion equations  with a separation of fast and slow components, and small noise in the slow component, by using the weak convergence method in infinite dimensions. A Wentzell-Freidlin type large deviation principle for stochastic partial differential equations with slow and fast time-scales, where the slow component is a one-dimensional stochastic Burgers' equation with small noise and the fast component is a stochastic reaction-diffusion equation  is established in \cite{XSRW}. In this work, we establish a Wentzell-Freidlin type large deviation principle for the two-time-scale stochastic partial differential equations, where the slow component is the SCBF equations in two and three dimensional bounded domains ($r\in[1,\infty),$ for $n=2$ and $r\in[3,\infty),$ for $n=3$ with $2\beta\mu>1$ for $r=3$) perturbed by small multiplicative Gaussian noise and the fast component is a stochastic reaction-diffusion equation with damping. We use a variational method (based on weak convergence approach) developed by Budhiraja and Dupuis (see, \cite{BD1,BD2}), Khasminkii's time discretization approach and stopping time arguments to obtain the LPD for the two-time-scale SCBF equations.  Furthermore, we remark that the results obtained in this work are true for two dimensional stochastic Navier-Stokes equations also. 

Our main aim of this work is to study a Wentzell-Freidlin type large deviation principle (LDP) for the following two-time-scale stochastic  convective Brinkman-Forchheimer (SCBF) equations in two and three dimensional bounded domains. The  two-time-scale SCBF equations are given by
\begin{equation}\label{1p1}
\left\{\begin{aligned}
\d\X^{\e,\delta}_t&=[\mu\Delta\X^{\e,\delta}_t-(\X^{\e,\delta}_t\cdot\nabla)\X^{\e,\delta}_t-\alpha\X^{\e,\delta}_t-\beta|\X^{\e,\delta}_t|^{r-1}\X^{\e,\delta}_t+\F(\X^{\e,\delta}_t,\Y^{\e,\delta}_t)]\d t\\&\quad-\nabla p^{\e,\delta}_t\d t+{\sqrt{\e}}\sigma_1(\X^{\e,\delta}_t)\Q_1^{1/2}\d\W_t,\\
\d \Y^{\e,\delta}_t&=\frac{1}{\delta}[\mu\Delta \Y^{\e,\delta}_t-\alpha \Y^{\e,\delta}_t-\beta|\Y^{\e,\delta}_t|^{r-1}\Y^{\e,\delta}_t+\G(\X^{\e,\delta}_t,\Y^{\e,\delta}_t)]\d t\\&\quad+\frac{1}{\sqrt{\delta}}\sigma_2(\X^{\e,\delta}_t,\Y^{\e,\delta}_t)\Q_2^{1/2}\d\W_t,\\
\nabla\cdot\X^{\e,\delta}_t&=0,\ \nabla\cdot\Y^{\e,\delta}_t=0,\\
\X^{\e,\delta}_t\big|_{\partial\mathcal{O}}&=\Y^{\e,\delta}_t\big|_{\partial\mathcal{O}}=\mathbf{0},\\
\X^{\e,\delta}_0&=\x,\ \Y^{\e,\delta}_0=\y,
\end{aligned}
\right. 
\end{equation}
for $t\in(0,T)$, where $\e>0$, $\delta=\delta(\e)>0$ (with $\delta\to 0$ and $\frac{\delta}{\e}\to 0$ as $\e\to 0$) are small parameters describing the ratio of the time scales of the slow component $\X^{\e,\delta}_t$  and the fast component $\Y^{\e,\delta}_t$, $\F,\G,\sigma_1,\sigma_2$ are appropriate functions, and $\W_t$ is a Hilbert space valued standard cylindrical Wiener process on a complete probability space $(\Omega,\mathscr{F},\mathbb{P})$ with filtration $\{\mathscr{F}_t\}_{t\geq 0}$. The system \eqref{1p1} can be considered as stochastic  convective Brinkman-Forchheimer equations, whose drift coefficient is coupled with a stochastic perturbation $\Y^{\e,\delta}_t$, which can be considered as the dramatically varying temperature (with damping) in the system.

The rest of the paper is organized as follows. In the next section, we provide some functional spaces as well as the hypothesis satisfied by the functions $\F,\G,\sigma_1,\sigma_2$ needed to obtain the global solvability of the system \eqref{1p1}. An abstract formulation for the two-time-scale SCBF system \eqref{1p1}  is formulated in section \ref{sec4} and we discuss about the existence and uniqueness of a pathwise strong solution to the system \eqref{1p1} (Theorem \ref{exis}). A Wentzell-Freidlin type large deviation principle for  the two-time-scale SCBF system is established in  \ref{se4} (Theorem \ref{thm4.14}). The results are obtained by using a variational method (based on weak convergence approach) developed by Budhiraja and Dupuis, Khasminkii's time discretization approach and stopping time arguments (Theorems \ref{compact} and \ref{weak}). Moreover, we deduce that the results obtained this work are true for two dimensional stochastic Navier-Stokes equations also (Remark \ref{rem4.22}).

\section{Mathematical Formulation}\label{sec2}\setcounter{equation}{0}
In this section, we describe the necessary function spaces and the hypothesis satisfied by the functions $\F,\G$ and the noise coefficients $\sigma_1,\sigma_2$ needed to obtain the global solvability results for the coupled SCBF equations \eqref{1p1}.

\subsection{Function spaces}\label{sub2.1} Let $\C_0^{\infty}(\mathcal{O};\R^n)$ denotes the space of all infinitely differentiable functions  ($\R^n$-valued) with compact support in $\mathcal{O}\subset\R^n$.  We define 
\begin{align*} 
\mathcal{V}&:=\{\X\in\C_0^{\infty}(\mathcal{O},\R^n):\nabla\cdot\X=0\},\\
\mathbb{H}&:=\text{the closure of }\ \mathcal{V} \ \text{ in the Lebesgue space } \L^2(\mathcal{O})=\mathrm{L}^2(\mathcal{O};\R^n),\\
\mathbb{V}&:=\text{the closure of }\ \mathcal{V} \ \text{ in the Sobolev space } \H_0^1(\mathcal{O})=\mathrm{H}_0^1(\mathcal{O};\R^n),\\
\widetilde{\L}^{p}&:=\text{the closure of }\ \mathcal{V} \ \text{ in the Lebesgue space } \L^p(\mathcal{O})=\mathrm{L}^p(\mathcal{O};\R^n),
\end{align*}
for $p\in(2,\infty)$. Then under some smoothness assumptions on the boundary, we characterize the spaces $\H$, $\V$ and $\widetilde{\L}^p$ as 
$
\H=\{\X\in\L^2(\mathcal{O}):\nabla\cdot\X=0,\X\cdot\mathbf{n}\big|_{\partial\mathcal{O}}=0\}$,  with norm  $\|\X\|_{\H}^2:=\int_{\mathcal{O}}|\X(x)|^2\d x,
$
where $\mathbf{n}$ is the outward normal to $\partial\mathcal{O}$,
$
\V=\{\X\in\H_0^1(\mathcal{O}):\nabla\cdot\X=0\},$  with norm $ \|\X\|_{\V}^2:=\int_{\mathcal{O}}|\nabla\X(x)|^2\d x,
$ and $\widetilde{\L}^p=\{\X\in\L^p(\mathcal{O}):\nabla\cdot\X=0, \X\cdot\mathbf{n}\big|_{\partial\mathcal{O}}=0\},$ with norm $\|\X\|_{\widetilde{\L}^p}^p=\int_{\mathcal{O}}|\X(x)|^p\d x$, respectively.
Let $(\cdot,\cdot)$ denotes the inner product in the Hilbert space $\H$ and $\langle \cdot,\cdot\rangle $ denotes the induced duality between the spaces $\V$  and its dual $\V'$ as well as $\widetilde{\L}^p$ and its dual $\widetilde{\L}^{p'}$, where $\frac{1}{p}+\frac{1}{p'}=1$. Note that $\H$ can be identified with its dual $\H'$. We endow the space $\V\cap\widetilde{\L}^{p}$ with the norm $\|\X\|_{\V}+\|\X\|_{\widetilde{\L}^{p}},$ for $\X\in\V\cap\widetilde{\L}^p$ and its dual $\V'+\widetilde{\L}^{p'}$ with the norm $$\inf\left\{\max\left(\|\Y_1\|_{\V'},\|\Y_1\|_{\widetilde{\L}^{p'}}\right):\Y=\Y_1+\Y_2, \ \Y_1\in\V', \ \Y_2\in\widetilde{\L}^{p'}\right\}.$$ Furthermore, we have the continuous embedding $\V\cap\widetilde{\L}^p\hookrightarrow\H\hookrightarrow\V'+\widetilde{\L}^{p'}$.

\subsection{Linear operator}\label{sub2.2}
Let us define
\begin{equation*}
\left\{
\begin{aligned}
\A\X:&=-\mathrm{P}_{\H}\Delta\X,\;\X\in\D(\A),\\ \D(\A):&=\V\cap\H^{2}(\mathcal{O}).
\end{aligned}
\right.
\end{equation*}
It can be easily seen that the operator $\A$ is a non-negative self-adjoint operator in $\H$ with $\V=\D(\A^{1/2})$ and \begin{align}\label{2.7a}\langle \A\X,\X\rangle =\|\X\|_{\V}^2,\ \textrm{ for all }\ \X\in\V, \ \text{ so that }\ \|\A\X\|_{\V'}\leq \|\X\|_{\V}.\end{align}
For a bounded domain $\mathcal{O}$, the operator $\A$ is invertible and its inverse $\A^{-1}$ is bounded, self-adjoint and compact in $\H$. Thus, using the spectral theorem, the spectrum of $\A$ consists of an infinite sequence $0< \lambda_1\leq \lambda_2\leq\ldots\leq \lambda_k\leq \ldots,$ with $\lambda_k\to\infty$ as $k\to\infty$ of eigenvalues. 
Moreover, there exists an orthogonal basis $\{e_k\}_{k=1}^{\infty} $ of $\H$ consisting of eigenvectors of $\A$ such that $\A e_k =\lambda_ke_k$,  for all $ k\in\mathbb{N}$.  We know that $\X$ can be expressed as $\X=\sum_{k=1}^{\infty}\langle\X,e_k\rangle e_k$ and $\A\X=\sum_{k=1}^{\infty}\lambda_k\langle\X,e_k\rangle e_k$, for all $\X\in\D(\A)$. Thus, it is immediate that 
\begin{align}\label{poin}
\|\nabla\X\|_{\mathbb{H}}^2=\langle \A\X,\X\rangle =\sum_{k=1}^{\infty}\lambda_k|\langle \X,e_k\rangle|^2\geq \lambda_1\sum_{k=1}^{\infty}|\langle\X,e_k\rangle|^2=\lambda_1\|\X\|_{\mathbb{H}}^2,
\end{align}
which is the \emph{Poincar\'e inequality}. 

\subsection{Bilinear operator}
Let us define the \emph{trilinear form} $b(\cdot,\cdot,\cdot):\V\times\V\times\V\to\R$ by $$b(\X,\Y,\Z)=\int_{\mathcal{O}}(\X(x)\cdot\nabla)\Y(x)\cdot\Z(x)\d x=\sum_{i,j=1}^n\int_{\mathcal{O}}\X_i(x)\frac{\partial \Y_j(x)}{\partial x_i}\Z_j(x)\d x.$$ If $\X, \Y$ are such that the linear map $b(\X, \Y, \cdot) $ is continuous on $\V$, the corresponding element of $\V'$ is denoted by $\B(\X, \Y)$. We also denote $\B(\X) = \B(\X, \X)=\mathrm{P}_{\H}(\X\cdot\nabla)\X$.
An integration by parts yields  
\begin{equation}\label{b0}
\left\{
\begin{aligned}
b(\X,\Y,\Y) &= 0,\text{ for all }\X,\Y \in\V,\\
b(\X,\Y,\Z) &=  -b(\X,\Z,\Y),\text{ for all }\X,\Y,\Z\in \V.
\end{aligned}
\right.\end{equation}
In the trilinear form, an application of H\"older's inequality yields
\begin{align*}
|b(\X,\Y,\Z)|=|b(\X,\Z,\Y)|\leq \|\X\|_{\widetilde{\L}^{r+1}}\|\Y\|_{\widetilde{\L}^{\frac{2(r+1)}{r-1}}}\|\Z\|_{\V},
\end{align*}
for all $\X\in\V\cap\widetilde{\L}^{r+1}$, $\Y\in\V\cap\widetilde{\L}^{\frac{2(r+1)}{r-1}}$ and $\Z\in\V$, so that we get 
\begin{align}\label{2p9}
\|\B(\X,\Y)\|_{\V'}\leq \|\X\|_{\widetilde{\L}^{r+1}}\|\Y\|_{\widetilde{\L}^{\frac{2(r+1)}{r-1}}}.
\end{align}
Hence, the trilinear map $b : \V\times\V\times\V\to \R$ has a unique extension to a bounded trilinear map from $(\V\cap\widetilde{\L}^{r+1})\times(\V\cap\widetilde{\L}^{\frac{2(r+1)}{r-1}})\times\V$ to $\R$. It can also be seen that $\B$ maps $ \V\cap\widetilde{\L}^{r+1}$  into $\V'+\widetilde{\L}^{\frac{r+1}{r}}$ and using interpolation inequality, we get 
\begin{align}\label{212}
\left|\langle \B(\X,\X),\Y\rangle \right|=\left|b(\X,\Y,\X)\right|\leq \|\X\|_{\widetilde{\L}^{r+1}}\|\X\|_{\widetilde{\L}^{\frac{2(r+1)}{r-1}}}\|\Y\|_{\V}\leq\|\X\|_{\widetilde{\L}^{r+1}}^{\frac{r+1}{r-1}}\|\X\|_{\H}^{\frac{r-3}{r-1}}\|\Y\|_{\V},
\end{align}
for all $\Y\in\V\cap\widetilde{\L}^{r+1}$. Thus, we have 
\begin{align}\label{2.9a}
\|\B(\X)\|_{\V'+\widetilde{\L}^{\frac{r+1}{r}}}\leq\|\X\|_{\widetilde{\L}^{r+1}}^{\frac{r+1}{r-1}}\|\X\|_{\H}^{\frac{r-3}{r-1}},
\end{align}
for $r\geq 3$. 

For $n=2$ and $r\in[1,3]$, using H\"older's and Ladyzhenskaya's inequalities, we obtain 
\begin{align*}
|\langle\B(\X,\Y),\Z\rangle|=|\langle\B(\X,\Z),\Y\rangle|\leq\|\X\|_{\wi\L^4}\|\Y\|_{\wi\L^4}\|\Z\|_{\V},
\end{align*}
for all $\X,\Y\in\wi\L^4$ and $\Z\in\V$, so that we get $\|\B(\X,\Y)\|_{\V'}\leq\|\X\|_{\wi\L^4}\|\Y\|_{\wi\L^4}$. Furthermore, we have $$\|\B(\X,\X)\|_{\V'}\leq\|\X\|_{\wi\L^4}^2\leq\sqrt{2}\|\X\|_{\H}\|\X\|_{\V}\leq\sqrt{\frac{2}{\lambda_1}}\|\X\|_{\V}^2,$$ for all $\X\in\V$. 
\subsection{Nonlinear operator}
Let us now consider the operator $\mathcal{C}(\X):=\P_{\H}(|\X|^{r-1}\X)$. It is immediate that $\langle\mathcal{C}(\X),\X\rangle =\|\X\|_{\widetilde{\L}^{r+1}}^{r+1}$. 
For any $r\in[1,\infty)$, we have 
\begin{align}\label{2pp11}
&\langle \mathrm{P}_{\H}(\X|\X|^{r-1})-\mathrm{P}_{\H}(\Y|\Y|^{r-1}),\X-\Y\rangle\nonumber\\&=\int_{\mathcal{O}}\left(\X(x)|\X(x)|^{r-1}-\Y(x)|\Y(x)|^{r-1}\right)\cdot(\X(x)-\Y(x))\d x\nonumber\\&=\int_{\mathcal{O}}\left(|\X(x)|^{r+1}-|\X(x)|^{r-1}\X(x)\cdot\Y(x)-|\Y(x)|^{r-1}\X(x)\cdot\Y(x)+|\Y(x)|^{r+1}\right)\d x\nonumber\\&\geq\int_{\mathcal{O}}\left(|\X(x)|^{r+1}-|\X(x)|^{r}|\Y(x)|-|\Y(x)|^{r}|\X(x)|+|\Y(x)|^{r+1}\right)\d x\nonumber\\&=\int_{\mathcal{O}}\left(|\X(x)|^r-|\Y(x)|^r\right)(|\X(x)|-|\Y(x)|)\d x\geq 0. 
\end{align}
Furthermore, we find 
\begin{align}\label{224}
&\langle\mathrm{P}_{\H}(\X|\X|^{r-1})-\mathrm{P}_{\H}(\Y|\Y|^{r-1}),\X-\Y\rangle\nonumber\\&=\langle|\X|^{r-1},|\X-\Y|^2\rangle+\langle|\Y|^{r-1},|\X-\Y|^2\rangle+\langle\Y|\X|^{r-1}-\X|\Y|^{r-1},\X-\Y\rangle\nonumber\\&=\||\X|^{\frac{r-1}{2}}(\X-\Y)\|_{\H}^2+\||\Y|^{\frac{r-1}{2}}(\X-\Y)\|_{\H}^2\nonumber\\&\quad+\langle\X\cdot\Y,|\X|^{r-1}+|\Y|^{r-1}\rangle-\langle|\X|^2,|\Y|^{r-1}\rangle-\langle|\Y|^2,|\X|^{r-1}\rangle.
\end{align}
But, we know that 
\begin{align*}
&\langle\X\cdot\Y,|\X|^{r-1}+|\Y|^{r-1}\rangle-\langle|\X|^2,|\Y|^{r-1}\rangle-\langle|\Y|^2,|\X|^{r-1}\rangle\nonumber\\&=-\frac{1}{2}\||\X|^{\frac{r-1}{2}}(\X-\Y)\|_{\H}^2-\frac{1}{2}\||\Y|^{\frac{r-1}{2}}(\X-\Y)\|_{\H}^2+\frac{1}{2}\langle\left(|\X|^{r-1}-|\Y|^{r-1}\right),\left(|\X|^2-|\Y|^2\right)\rangle \nonumber\\&\geq -\frac{1}{2}\||\X|^{\frac{r-1}{2}}(\X-\Y)\|_{\H}^2-\frac{1}{2}\||\Y|^{\frac{r-1}{2}}(\X-\Y)\|_{\H}^2.
\end{align*}
From \eqref{224}, we finally have 
\begin{align}\label{2.23}
&\langle\mathrm{P}_{\H}(\X|\X|^{r-1})-\mathrm{P}_{\H}(\Y|\Y|^{r-1}),\X-\Y\rangle\geq \frac{1}{2}\||\X|^{\frac{r-1}{2}}(\X-\Y)\|_{\H}^2+\frac{1}{2}\||\Y|^{\frac{r-1}{2}}(\X-\Y)\|_{\H}^2\geq 0,
\end{align}
for $r\geq 1$.  	It is important to note that 
\begin{align}\label{a215}
\|\X-\Y\|_{\wi\L^{r+1}}^{r+1}&=\int_{\mathcal{O}}|\X(x)-\Y(x)|^{r-1}|\X(x)-\Y(x)|^{2}\d x\nonumber\\&\leq 2^{r-2}\int_{\mathcal{O}}(|\X(x)|^{r-1}+|\Y(x)|^{r-1})|\X(x)-\Y(x)|^{2}\d x\nonumber\\&\leq 2^{r-2}\||\X|^{\frac{r-1}{2}}(\X-\Y)\|_{\L^2}^2+2^{r-2}\||\Y|^{\frac{r-1}{2}}(\X-\Y)\|_{\L^2}^2. 
\end{align}
Combining \eqref{2.23} and \eqref{a215}, we obtain 
\begin{align}\label{214}
\langle\mathcal{C}(\X)-\mathcal{C}(\Y),\X-\Y\rangle\geq\frac{1}{2^{r-1}}\|\X-\Y\|_{\wi\L^{r+1}}^{r+1},
\end{align}
for $r\geq 1$. 
\subsection{Monotonicity}
In this subsection, we discuss about the monotonicity as well as the hemicontinuity properties of the linear and nonlinear operators.
\begin{definition}[\cite{VB}]
	Let $\mathbb{X}$ be a Banach space and let $\mathbb{X}^{'}$ be its topological dual.
	An operator $\G:\mathrm{D}\rightarrow
	\mathbb{X}^{'},$ $\mathrm{D}=\mathrm{D}(\G)\subset \mathbb{X}$ is said to be
	\emph{monotone} if
	$$\langle\G(x)-\G(y),x-y\rangle\geq
	0,\ \text{ for all } \ x,y\in \mathrm{D}.$$ 
	The operator $\G(\cdot)$ is said to be \emph{hemicontinuous}, if for all $x, y\in\mathbb{X}$ and $w\in\mathbb{X}',$ $$\lim_{\lambda\to 0}\langle\G(x+\lambda y),w\rangle=\langle\G(x),w\rangle.$$
	The operator $\G(\cdot)$ is called \emph{demicontinuous}, if for all $x\in\mathrm{D}$ and $y\in\mathbb{X}$, the functional $x \mapsto\langle \G(x), y\rangle$  is continuous, or in other words, $x_k\to x$ in $\mathbb{X}$ implies $\G(x_k)\xrightarrow{w}\G(x)$ in $\mathbb{X}'$. Clearly demicontinuity implies hemicontinuity. 
\end{definition}
\begin{lemma}[Theorem 2.2, \cite{MTM7}]\label{thm2.2}
	Let $\X,\Y\in\V\cap\widetilde{\L}^{r+1}$, for $r>3$. Then,	for the operator $\G(\X)=\mu \A\X+\B(\X)+\beta\mathcal{C}(\X)$, we  have 
	\begin{align}\label{fe}
	\langle(\G(\X)-\G(\Y),\X-\Y\rangle+\eta\|\X-\Y\|_{\H}^2\geq 0,
	\end{align}
	where $\eta=\frac{r-3}{2\mu(r-1)}\left(\frac{2}{\beta\mu (r-1)}\right)^{\frac{2}{r-3}}.$ That is, the operator $\G+\eta\mathrm{I}$ is a monotone operator from $\V\cap\widetilde{\L}^{r+1}$ to $\V'+\widetilde{\L}^{\frac{r+1}{r}}$. 
\end{lemma}

\begin{lemma}[Theorem 2.3, \cite{MTM7}]\label{thm2.3}
	For the critical case $r=3$ with $2\beta\mu \geq 1$, the operator $\G(\cdot):\V\cap\widetilde{\L}^{r+1}\to \V'+\widetilde{\L}^{\frac{r+1}{r}}$ is globally monotone, that is, for all $\X,\Y\in\V$, we have 
	\begin{align}\label{218}\langle\G(\X)-\G(\Y),\X-\Y\rangle\geq 0.\end{align}
\end{lemma}
\begin{lemma}[Remark 2.4, \cite{MTM7}]
	Let $n=2$, $r\in[1,3]$ and $\X,\Y\in\V$. Then,	for the operator $\G(\X)=\mu \A\X+\B(\X)+\beta\mathcal{C}(\X)$, we  have 
	\begin{align}\label{fe2}
	\langle(\G(\X)-\G(\Y),\X-\Y\rangle+ \frac{27}{32\mu ^3}N^4\|\X-\Y\|_{\H}^2\geq 0,
	\end{align}
	for all $\Y\in{\mathbb{B}}_N$, where ${\mathbb{B}}_N$ is an $\widetilde{\L}^4$-ball of radius $N$, that is,
	$
	{\mathbb{B}}_N:=\big\{\z\in\widetilde{\L}^4:\|\z\|_{\widetilde{\L}^4}\leq N\big\}.
	$
\end{lemma}

\begin{lemma}[Lemma 2.5, \cite{MTM7}]\label{lem2.8}
	The operator $\G:\V\cap\widetilde{\L}^{r+1}\to \V'+\widetilde{\L}^{\frac{r+1}{r}}$ is demicontinuous. 
\end{lemma}

\section{Stochastic  Coupled   Convective Brinkman-Forchheimer Equations}\label{sec4}\setcounter{equation}{0}
Let $(\Omega,\mathscr{F},\mathbb{P})$ be a complete probability space equipped with an increasing family of sub-sigma fields $\{\mathscr{F}_t\}_{0\leq t\leq T}$ of $\mathscr{F}$ satisfying:
\begin{enumerate}
	\item [(i)] $\mathscr{F}_0$ contains all elements $F\in\mathscr{F}$ with $\mathbb{P}(F)=0$,
	\item [(ii)] $\mathscr{F}_t=\mathscr{F}_{t+}=\bigcap\limits_{s>t}\mathscr{F}_s,$ for $0\leq t\leq T$.
\end{enumerate}  In this section, we consider the stochastic coupled   convective Brinkman-Forchheimer  equations perturbed by multiplicative Gaussian noise. On  taking orthogonal projection $\mathrm{P}_{\H}$ onto the first two equations in \eqref{1p1}, we obtain  
\begin{equation}\label{3.6}
\left\{
\begin{aligned}
\d \X^{\e,\delta}_t&=-[\mu\A \X^{\e,\delta}_t+\B(\X^{\e,\delta}_t)+\alpha\X^{\e,\delta}_t+\beta\mathcal{C}(\X^{\e,\delta}_t)-\mathrm{F}(\X^{\e,\delta}_t,\Y^{\e,\delta}_t)]\d t+\sqrt{\e}\sigma_1(\X^{\e,\delta}_t)\Q_1^{1/2}\d\W_t,\\
\d \Y^{\e,\delta}_t&=-\frac{1}{\delta}[\mu\A \Y^{\e,\delta}_t+\alpha\Y^{\e,\delta}_t+\beta\mathcal{C}(\Y_{t}^{\e,\delta})-\mathrm{G}(\X^{\e,\delta}_t,\Y^{\e,\delta}_t)]\d t+\frac{1}{\sqrt{\delta}}\sigma_2(\X^{\e,\delta}_t,\Y^{\e,\delta}_t)\Q_2^{1/2}\d\W_t,\\
\X^{\e,\delta}_0&=\x,\ \Y^{\e,\delta}_0=\y,
\end{aligned}
\right. 
\end{equation}
where $\Q_i$'s, for $i=1,2$  are positive symmetric trace class operators in $\H$ and $\W_t$ is an $\H$-valued standard cylindrical Wiener process. Strictly speaking, one has to use $\mathrm{P}_{\H}\mathrm{F}$, $\mathrm{P}_{\H}\mathrm{G}$, $\mathrm{P}_{\H}\sigma_1$ and $\mathrm{P}_{\H}\sigma_2$ instead of $\mathrm{F}$, $\mathrm{G}$, $ \sigma_1$ and $\sigma_2$ in \eqref{3.6}, and for simplicity of notations we are keeping them as it is. Since $\Q_1$ and $\Q_2$ are trace class operators, the embedding of $\Q_i^{1/2}\H$ in $\H$ is Hilbert-Schmidt, for $i=1,2$. Let $\mathcal{L}(\H;\H)$ denotes the space of all bounded linear operators on $\H$ and $\mathcal{L}_2(\H;\H)$ denotes the space of all Hilbert-Schmidt operators from $\H$ to $\H$.  The space $\mathcal{L}_2(\H;\H)$ is a Hilbert space equipped with the norm $ \left\|\Psi\right\|^2_{\mathcal{L}_2}=\Tr\left(\Psi\Psi^*\right)=\sum_{k=1}^{\infty}\|\Psi^*e_k\|_{\H}^2$ and inner product $(\Psi,\Phi)_{\mathcal{L}_2}=\Tr(\Psi\Phi^*)=\sum_{k=1}^{\infty}(\Phi^*e_k,\Psi^*e_k)$. For more details, the interested readers are referred to see \cite{DaZ}. 

We need the following Assumptions on $\F,\G,\sigma_1$ and $\sigma_2$ to obtain our main results (see \cite{MTM11,XSRW} also). 
\begin{assumption}\label{ass3.6}
	The functions $\F,\G:\H\times\H\to\H$, $\sigma_1\Q_1^{1/2}:\H\to\mathcal{L}_{2}(\H;\H)$ and $\sigma_2\Q_2^{1/2}:\H\times\H\to\mathcal{L}_{2}(\H;\H)$ satisfy the following Assumptions:
	\begin{itemize}
		\item [(A1)] The functions $\F,\G,\sigma_1,\sigma_2$ are Lipschitz continuous, that is, there exist positive constants $C,L_{\G}$ and $L_{\sigma_2}$ such that for any $\x_1,\x_2,\y_1,\y_2\in\H$, we have 
		\begin{align*}
		\|\F(\x_1,\y_1)-\F(\x_2,\y_2)\|_{\H}&\leq C(\|\x_1-\x_2\|_{\H}+\|\y_1-\y_2\|_{\H}),\\
		\|\G(\x_1,\y_1)-\G(\x_2,\y_2)\|_{\H}&\leq C\|\x_1-\x_2\|_{\H}+L_{\G}\|\y_1-\y_2\|_{\H},\\
		\|[\sigma_1(\x_1)-\sigma_1(\x_2)]\Q_1^{1/2}\|_{\mathcal{L}_2}&\leq C\|\x_1-\x_2\|_{\H},\\
		\|[\sigma_2(\x_1,\y_1)-\sigma_2(\x_2,\y_2)]\Q_2^{1/2}\|_{\mathcal{L}_2}&\leq C\|\x_1-\x_2\|_{\H}+L_{\sigma_2}\|\y_1-\y_2\|_{\H}. 
		\end{align*}
		\item [(A2)] The function $\sigma_2$ grows linearly in $\x$, but is bounded in $\y$, that is, there exists a constant $C>0$ such that  
		\begin{align*}
	\sup_{\y\in\H}	\|\sigma_2(\x,\y)\Q_2^{1/2}\|_{\mathcal{L}_2}\leq C(1+\|\x\|_{\H}), \ \text{ for all }\ \x\in\H. 
		\end{align*}
		\item [(A3)] The Brinkman coefficient $\mu>0$, Darcy coefficient $\alpha>0$,  the smallest eigenvalue $\lambda_1$  of the Stokes operator and the Lipschitz constants $L_{\G}$ and $L_{\sigma_2}$ satisfy $$\mu\lambda_1+2\alpha-2L_{\G}-2L_{\sigma_2}^2>0.$$ 
			\item [(A4)] $\lim\limits_{\e\downarrow 0}\delta(\e)=0$ and $\lim\limits_{\e\downarrow 0}\frac{\delta}{\e}=0$. 
	\end{itemize}
\end{assumption}

Let us now provide the definition of a unique global strong solution in the probabilistic sense to the system (\ref{3.6}).
\begin{definition}[Global strong solution]
	Let $(\x,\y)\in\H\times\H$ be given. An $\H\times\H$-valued $(\mathscr{F}_t)_{t\geq 0}$-adapted stochastic process $(\X_{t}^{\e,\delta},\Y_{t}^{\e,\delta})$ is called a \emph{strong solution} to the system (\ref{3.6}) if the following conditions are satisfied: 
	\begin{enumerate}
		\item [(i)] the process \begin{align*}(\X^{\e,\delta},\Y^{\e,\delta})&\in\mathrm{L}^2(\Omega;\mathrm{L}^{\infty}(0,T;\H)\cap\mathrm{L}^2(0,T;\V))\cap\mathrm{L}^{r+1}(\Omega;\mathrm{L}^{r+1}(0,T;\widetilde{\L}^{r+1}))\nonumber\\&\quad\times\mathrm{L}^2(\Omega;\mathrm{L}^{\infty}(0,T;\H)\cap\mathrm{L}^2(0,T;\V))\cap\mathrm{L}^{r+1}(\Omega;\mathrm{L}^{r+1}(0,T;\widetilde{\L}^{r+1}))\end{align*} and $(\X_t^{\e,\delta},\Y_t^{\e,\delta})$ has a $(\V\cap\widetilde{\L}^{r+1})\times(\V\cap\widetilde{\L}^{r+1})$-valued  modification, which is progressively measurable with continuous paths in $\H\times\H$ and \begin{align*}(\X^{\e,\delta},\Y^{\e,\delta})&\in\C([0,T];\H)\cap\mathrm{L}^2(0,T;\V)\cap\mathrm{L}^{r+1}(0,T;\widetilde{\L}^{r+1})\nonumber\\&\quad\times\C([0,T];\H)\cap\mathrm{L}^2(0,T;\V)\cap\mathrm{L}^{r+1}(0,T;\widetilde{\L}^{r+1}), \ \mathbb{P}\text{-a.s.}\end{align*} 
		\item [(ii)] the following equality holds for every $t\in [0, T ]$, as an element of $(\V'+\wi\L^{\frac{r+1}{r}})\times(\V'+\wi\L^{\frac{r+1}{r}}),$ $\mathbb{P}$-a.s.:
	\begin{equation}\label{4.4}
	\left\{
	\begin{aligned}
	\X^{\e,\delta}_t&=\x-\int_0^t[\mu\A \X^{\e,\delta}_s+\alpha\X^{\e,\delta}_s+\B(\X^{\e,\delta}_s)+\beta\mathcal{C}(\X^{\e,\delta}_s)-\F(\X^{\e,\delta}_s,\Y^{\e,\delta}_s)]\d s\\&\quad+\sqrt{\e}\int_0^t\sigma_1(\X^{\e,\delta}_s)\Q_1^{1/2}\d\W_s,\\
	 \Y^{\e,\delta}_t&=\y-\frac{1}{\delta}\int_0^t[\mu\A \Y^{\e,\delta}_s+\alpha\Y^{\e,\delta}_s+\beta\mathcal{C}(\Y_{s}^{\e,\delta})-\G(\X^{\e,\delta}_s,\Y^{\e,\delta}_s)]\d s\\&\quad+\frac{1}{\sqrt{\delta}}\int_0^t\sigma_2(\X^{\e,\delta}_s,\Y^{\e,\delta}_s)\Q_2^{1/2}\d\W_s,
	\end{aligned}
	\right. 
	\end{equation}
			\item [(iii)] the following It\^o formula (energy equality) holds true: 
				\begin{align}\label{3p3}
			&	\|	\X^{\e,\delta}_t\|_{\H}^2+\|\Y^{\e,\delta}_t\|_{\H}^2+2\mu \int_0^t\left(\|	\X^{\e,\delta}_s\|_{\V}^2+\frac{1}{\delta}\|\Y^{\e,\delta}_s\|_{\V}^2\right)\d s+2\alpha \int_0^t\left(\|	\X^{\e,\delta}_s\|_{\H}^2+\frac{1}{\delta}\|\Y^{\e,\delta}_s\|_{\H}^2\right)\d s\nonumber\\&\quad +2\beta\int_0^t\left(\|\X^{\e,\delta}_s\|_{\widetilde{\L}^{r+1}}^{r+1}+\frac{1}{\delta}\|\Y^{\e,\delta}_s\|_{\widetilde{\L}^{r+1}}^{r+1}\right)\d s\nonumber\\&=\|\x\|_{\H}^2+\|\y\|_{\H}^2+2\int_0^t(\F(\X^{\e,\delta}_s,\Y^{\e,\delta}_s),\X^{\e,\delta}_s)\d s+\frac{2}{\delta}\int_0^t(\G(\X^{\e,\delta}_s,\Y^{\e,\delta}_s),\X^{\e,\delta}_s)\d s\nonumber\\&\quad+\int_0^t\left(\e\|\sigma_1(\X^{\e,\delta}_s)\Q_1^{1/2}\|_{\mathcal{L}_2}^2+\frac{1}{\delta}\|\sigma_2(\X^{\e,\delta}_s,\Y^{\e,\delta}_s)\Q_2^{1/2}\|_{\mathcal{L}_2}^2\right)\d s\nonumber\\&\quad+2\int_0^t(\sqrt{\e}\sigma_1(\X^{\e,\delta}_s)\Q_1^{1/2}\d\W_s,\X^{\e,\delta}_s)+\frac{2}{\sqrt{\delta}}\int_0^t(\sigma_2(\X^{\e,\delta}_s,\Y^{\e,\delta}_s)\Q_2^{1/2}\d\W_s,\Y^{\e,\delta}_s),
			\end{align}
			for all $t\in(0,T)$, $\mathbb{P}$-a.s.
	\end{enumerate}
\end{definition}
An alternative version of condition (\ref{4.4}) is to require that for any  $(\mathrm{U},\mathrm{V})\in(\V\cap\widetilde{\L}^{r+1})\times(\V\cap\widetilde{\L}^{r+1})$, $\mathbb{P}$-a.s.:
	\begin{equation}\label{4.5}
\left\{
\begin{aligned}
(\X^{\e,\delta}_t,\mathrm{U})&=(\x,\mathrm{U})-\int_0^t\langle\mu\A \X^{\e,\delta}_s+\B(\X^{\e,\delta}_s)+\alpha\X^{\e,\delta}_s+\beta\mathcal{C}(\X^{\e,\delta}_s)-\F(\X^{\e,\delta}_s,\Y^{\e,\delta}_s),\mathrm{U}\rangle\d s\nonumber\\&\quad+\int_0^t(\sqrt{\e}\sigma_1(\X^{\e,\delta}_s)\Q_1^{1/2}\d\W_s,\mathrm{U}),\\
(\Y^{\e,\delta}_t,\mathrm{V})&=(\y,\mathrm{V})-\frac{1}{\delta}\int_0^t\langle\mu\A \Y^{\e,\delta}_s+\alpha\Y^{\e,\delta}_s+\beta\mathcal{C}(\Y_{s}^{\e,\delta})-\G(\X^{\e,\delta}_s,\Y^{\e,\delta}_s),\mathrm{V}\rangle\d s\nonumber\\&\quad+\frac{1}{\sqrt{\delta}}\int_0^t(\sigma_2(\X^{\e,\delta}_s,\Y^{\e,\delta}_s)\Q_2^{1/2}\d\W_s,\mathrm{V}),
\end{aligned}
\right. 
\end{equation}
for every $t\in[0,T]$. 
\begin{definition}
	A strong solution $(\X^{\e,\delta}_t,\Y^{\e,\delta}_t)$ to (\ref{3.6}) is called a
	\emph{pathwise  unique strong solution} if
	$(\wi\X^{\e,\delta}_t,\wi\Y^{\e,\delta}_t)$ is an another strong
	solution, then $$\mathbb{P}\Big\{\omega\in\Omega:  (\X^{\e,\delta}_t,\Y^{\e,\delta}_t)=(\wi\X^{\e,\delta}_t,\wi\Y^{\e,\delta}_t),\  \text{ for all }\ t\in[0,T]\Big\}=1.$$ 
\end{definition}
\begin{remark}\label{rem3.4}
For $n=2$ and $r\in[1,3]$, 	making use of the Gagliardo-Nirenberg interpolation inequality, we know that $\C([0,T];\H)\cap\mathrm{L}^2(0,T;\V)\subset\mathrm{L}^{r+1}(0,T;\wi\L^{r+1})$, and hence we get  $\C([0,T];\H)\cap\mathrm{L}^2(0,T;\V)\cap\mathrm{L}^{r+1}(0,T;\widetilde{\L}^{r+1})=\C([0,T];\H)\cap\mathrm{L}^2(0,T;\V)$. 
\end{remark}

\subsection{Global strong solution} In this section, we discuss the existence and uniqueness of strong solution to the system \eqref{3.6}. 
For convenience, we make use of  the following simplified notations. 	Let us define $\mathscr{H}:=\H\times\H$. For any $\mathrm{U}=(\x_1,\x_2),\mathrm{V}=(\y_1,\y_2)\in\mathscr{H},$ we denote the inner product and norm on this Hilbert space by 
\begin{align}
(\mathrm{U},\mathrm{V})=(\x_1,\y_1)+(\x_2,\y_2), \ \|\mathrm{U}\|_{\mathscr{H}}=\sqrt{(\mathrm{U},\mathrm{U})}=\sqrt{\|\x_1\|_{\H}^2+\|\x_2\|_{\H}^2}. 
\end{align}
In a similar way, we define $\mathscr{V}:=\V\times\V$. The  inner product and norm on this Hilbert space is defined by 
\begin{align}
(\mathrm{U},\mathrm{V})_{\mathscr{V}}=(\nabla\x_1,\nabla\y_1)+(\nabla\x_2,\nabla\y_2), \ \|\mathrm{U}\|_{\mathscr{V}}=\sqrt{(\mathrm{U},\mathrm{U})_{\mathscr{V}}}=\sqrt{\|\nabla\x_1\|_{\H}^2+\|\nabla\x_2\|_{\H}^2}, 
\end{align}
for all $\mathrm{U},\mathrm{V}\in\mathscr{V}$. We denote $\mathscr{V}'$ as the dual of $\mathscr{V}$. We define the space $\widetilde{\mathfrak{L}}^{r+1}:=\widetilde{\L}^{r+1}\times\widetilde{\L}^{r+1}$ with the norm given by 
$$\|\mathrm{U}\|_{\widetilde{\mathfrak{L}}^{r+1}}=\left\{\|\x_1\|_{\wi\L^{r+1}}^{r+1}+\|\x_2\|_{\wi\L^{r+1}}^{r+1}\right\}^{r+1},$$ 	for all $\mathrm{U}\in\widetilde{\mathfrak{L}}^{r+1}$. We represent the duality pairing between $\mathscr{V}$ and its dual $\mathscr{V}'$,  $\mathfrak{L}^{r+1}$ and its dual $\mathfrak{L}^{\frac{r+1}{r}}$, and $\mathscr{V}\cap\mathfrak{L}^{r+1}$ and its dual $\mathscr{V}'+\mathfrak{L}^{\frac{r+1}{r}}$ as $\langle\cdot,\cdot\rangle$. Note that we have the Gelfand triple $\mathscr{V}\cap\mathfrak{L}^{r+1} \subset\mathscr{H}\subset\mathscr{V}'+\mathfrak{L}^{\frac{r+1}{r}}$.  Let
$\Q=(\Q_1,\Q_2)$ be a positive symmetric trace class operator on $\mathscr{H}$. Let us now rewrite the system \eqref{3.6} for $\Z^{\e,\delta}_{t}=(\X^{\e,\delta}_{t},\Y^{\e,\delta}_{t})$ as 
\begin{equation}\label{3p6}
\left\{
\begin{aligned}
\d\Z^{\e,\delta}_{t}&=-\left[\mu\wi\A\Z^{\e,\delta}_{t}+\wi\F(\Z^{\e,\delta}_{t})\right]\d t+\wi\sigma(\Z^{\e,\delta}_{t})\Q^{1/2}\d\W_t, \\
\Z^{\e,\delta}_0&=(\x,\y)\in\mathscr{H}, 
\end{aligned}
\right. 
\end{equation}
where 
\begin{align*}
\wi\A\Z^{\e,\delta}&=\left(\A\X^{\e,\delta},\frac{1}{\delta}\A\Y^{\e,\delta}\right), \\
\wi\F(\Z^{\e,\delta})&=\left(\B(\X^{\e,\delta})+\alpha\X^{\e,\delta}+\beta\mathcal{C}(\X^{\e,\delta})-\F(\X^{\e,\delta},\Y^{\e,\delta}),\frac{\alpha}{\delta}\Y^{\e,\delta}+\frac{\beta}{\delta}\mathcal{C}(\Y^{\e,\delta})-\frac{1}{\delta}\G(\X^{\e,\delta},\Y^{\e,\delta})\right), \\
\wi\sigma(\Z^{\e,\delta})&=\left(\sqrt{\e}\sigma_1(\X^{\e,\delta}),\frac{1}{\sqrt{\delta}}\sigma_2(\X^{\e,\delta},\Y^{\e,\delta})\right). 
\end{align*}
Note that the mappings $\widetilde{\A}:\mathscr{V}\to\mathscr{V}'$ and $\widetilde{\F}:\mathscr{V}\cap\mathfrak{L}^{r+1}\to\mathscr{V}'+\mathfrak{L}^{\frac{r+1}{r}}$ are well defined.   It can be easily seen that the operator $\wi\sigma\Q^{1/2}:\mathscr{H}\to\mathcal{L}_{2}(\mathscr{H};\mathscr{H})$, where  $\mathcal{L}_{2}(\mathscr{H};\mathscr{H})$ is the space of all Hilbert-Schmidt operators from $\mathscr{H}$ to $\mathscr{H}$ with the norm 
\begin{align}
\|\wi\sigma(\z)\Q^{1/2}\|_{\mathcal{L}_2}=\sqrt{\|\sigma_1(\x)\Q_1^{1/2}\|_{\mathcal{L}_2}^2+\|\sigma_2(\x,\y)\Q_2^{1/2}\|_{\mathcal{L}_2}^2}, \ \text{ for }\  \z=(\x,\y)\in\mathscr{H}. 
\end{align}

The following Theorem on the existence and uniqueness of strong solution to the system \eqref{3p6} can be  proved in a similar way as in Theorem 3.4, \cite{MTM11}. 

\begin{theorem}[Theorem 3.4, \cite{MTM11}]\label{exis}
	Let $(\x,\y)\in \mathscr{H}$ be given.  Then for $n=2$, $r\in[1,\infty)$ and $n=3$, $r\in [3,\infty)$ ($2\beta\mu\geq 1$, for $r=3$), there exists a \emph{pathwise unique  strong solution}	$\Z^{\e,\delta}$ to the system (\ref{3p6}) such that \begin{align*}\Z^{\e,\delta}&\in\mathrm{L}^2(\Omega;\mathrm{L}^{\infty}(0,T;\mathscr{H})\cap\mathrm{L}^2(0,T;\mathscr{V}))\cap\mathrm{L}^{r+1}(\Omega;\mathrm{L}^{r+1}(0,T;\widetilde{\mathfrak{L}}^{r+1})),\end{align*} and a continuous modification with trajectories in $\mathscr{H}$ and $\Z^{\e,\delta}\in\C([0,T];\mathscr{H})\cap\mathrm{L}^2(0,T;\mathscr{V})\cap\mathrm{L}^{r+1}(0,T;\widetilde{\mathfrak{L}}^{r+1})$, $\mathbb{P}$-a.s.
\end{theorem}

\section{Large Deviation Principle}\label{se4}\setcounter{equation}{0}

In this section, we establish a Wentzell-Freidlin (see \cite{FW}) type large deviation principle for the two-time-scale SCBF equations \eqref{3.6} using the well known results of Varadhan as well as Bryc (see \cite{Va,DZ}) and Budhiraja-Dupuis (see \cite{BD1}). A Wentzell-Freidlin type large deviation principle for the two-time-scale one-dimensional stochastic Burgers equation is established in \cite{XSRW}. Interested readers are referred to see \cite{MTM10} (LDP for two and three dimensional SCBF equations),   \cite{SSSP} (LDP for the 2D stochastic Navier-Stokes equations), \cite{ICAM} (LDP for some 2D hydrodynamic systems),  \cite{MTM6} (LDP for the 2D Oldroyd fluids), \cite{MTM10} (LDP for the 2D and 3D SCBF equations)  for application of such methods to various hydrodynamic models. 

Let $(\Omega,\mathscr{F},\mathbb{P})$ be a probability space with an increasing family $\{\mathscr{F}_t\}_{t\geq 0}$ of the sub $\sigma$-fields of $\mathscr{F}$ satisfying the usual conditions.  We consider the following two-time-scale SCBF system:
\begin{equation}\label{2.18a}
\left\{
\begin{aligned}
\d \X^{\e,\delta}_t&=-[\mu\A \X^{\e,\delta}_t+\B(\X^{\e,\delta}_t)+\alpha\X^{\e,\delta}_t+\beta\mathcal{C}(\X^{\e,\delta}_t)-\mathrm{F}(\X^{\e,\delta}_t,\Y^{\e,\delta}_t)]\d t+\sqrt{\e}\sigma_1(\X^{\e,\delta}_t)\Q_1^{1/2}\d\W_t,\\
\d \Y^{\e,\delta}_t&=-\frac{1}{\delta}[\mu\A \Y^{\e,\delta}_t+\alpha\Y^{\e,\delta}_t+\beta\mathcal{C}(\Y_{t}^{\e,\delta})-\mathrm{G}(\X^{\e,\delta}_t,\Y^{\e,\delta}_t)]\d t+\frac{1}{\sqrt{\delta}}\sigma_2(\X^{\e,\delta}_t,\Y^{\e,\delta}_t)\Q_2^{1/2}\d\W_t,\\
\X^{\e,\delta}_0&=\x,\ \Y^{\e,\delta}_0=\y,
\end{aligned}
\right. 
\end{equation}
for some fixed point $(\x,\y)$ in $\H\times\H$. From Theorem \ref{exis} (see Theorem 3.4, \cite{MTM11} also), it is known that the system \eqref{2.18a} has a unique  pathwise strong solution $(\X^{\e,\delta}_t,\Y^{\e,\delta}_t)$ with $\mathscr{F}_t$-adapted paths (that is, for any $t\in[0,T]$ and $x\in\mathcal{O}$,  $(\X^{\e,\delta}_t(x),\Y^{\e,\delta}_t(x))$ is $\mathscr{F}_t$-measurable) in  $\C([0,T];\H)\cap \mathrm{L}^2(0,T;\V)\cap\mathrm{L}^{r+1}(0,T;\widetilde\L^{r+1})\times \C([0,T];\H)\cap \mathrm{L}^2(0,T;\V)\cap\mathrm{L}^{r+1}(0,T;\widetilde\L^{r+1}),\ \mathbb{P}\text{-a.s.}$ Moreover, such a strong solution satisfies the energy equality (It\^o's formula) given in \eqref{3p3}. 

As the parameter $\e\downarrow 0$, the slow component $\X^{\e,\delta}_t$ of (\ref{2.18a}) tends to the solution of the following deterministic averaged system:
\begin{equation}\label{2.19}
	\left\{
\begin{aligned}
\d\bar{\X}_t&=-[\mu\A\bar{\X}_t+\B(\bar{\X}_t)+\alpha\bar{\X}_t+\beta\mathcal{C}(\bar{\X}_t)]\d t+\bar{\F}(\bar{\X}_t)\d t, \\ \bar{\X}_0&=\x, 
\end{aligned}\right. 
\end{equation}
with the average 
$$
\bar{\F}(\x)=\int_{\H}\F(\x,\y)\nu^{\x}(\d\y), \ \x\in\H, 
$$ and $\nu^{\x}$ is the unique invariant distribution of the transition  semigroup for the frozen system: 
\begin{equation}\label{2p19} 
\left\{
\begin{aligned}
\d\Y_t&=-[\mu\A\Y_t+\alpha\Y_t+\beta\mathcal{C}(\Y_t)-\G(\x,\Y_t)]\d t+\sigma_2(\x,\Y_t)\Q_2^{1/2}\d\bar{\W}_t,\\
\Y_0&=\y,
\end{aligned}\right.
\end{equation}
where $\bar{\W}_t$ is a standard cylindrical Wiener process, which is independent of $\W_{t}$.  
Note that the system \eqref{2.19} is a Lipschitz perturbation of the CBF equations (see \eqref{3p93} below). Using similar techniques as in Theorem 3.4, \cite{MTM7} (see  \cite{CLF} also), one can show that the system (\ref{2.19}) has a unique weak solution in the Leray-Hopf sense, satisfying the energy equality: 
\begin{align*}
&\|\bar{\X}_t\|_{\H}^2+2\mu\int_0^t\|\bar{\X}_s\|_{\V}^2\d s+2\alpha\int_0^t\|\bar{\X}_s\|_{\H}^2\d s+2\beta\int_0^t\|\bar{\X}_s\|_{\wi\L^{r+1}}^{r+1}\d s\nonumber\\&=\|\x\|_{\H}^2+2\int_0^t(\bar{\F}(\bar{\X}_s),\bar{\X}_s)\d s,
\end{align*}
for all $t\in[0,T]$ in the Polish space $\C([0,T];\H)\cap \mathrm{L}^2(0,T;\V)\cap\mathrm{L}^{r+1}(0,T;\widetilde\L^{r+1})$. The strong averaging principle states that (Theorem 1.1, \cite{MTM11}) for any initial values $\x,\y\in\H$, $p\geq 1$ and $T>0$, we have 
\begin{align}\label{3.126}
\lim_{\e\to 0} \E\left(\sup_{t\in[0,T]}\|\X^{\e,\delta}_t-\bar\X_t\|_{\H}^{2p}\right)=0, 
\end{align}
 In this section, we  investigate the large deviations of $\X^{\e,\delta}_t$ from the deterministic solution $\bar{\X}_t$, as $\e\downarrow 0$.

\subsection{Frozen equation} The frozen equation associated with the fast motion for fixed slow component  $\x\in\H$ is given by 
\begin{equation}\label{3.74} 
\left\{
\begin{aligned}
\d\Y_t&=-[\mu\A\Y_t+\alpha\Y_t+\beta\mathcal{C}(\Y_t)-\G(\x,\Y_t)]\d t+\sigma_2(\x,\Y_t)\Q_2^{1/2}\d\bar{\W}_t,\\
\Y_0&=\y,
\end{aligned}\right.
\end{equation}
where $\bar{\W}_t$ is a cylindrical Wiener process in $\H$, which is independent of $\W_{t}$. From the Assumption \ref{ass3.6}, we know that $\G(\x,\cdot)$ and $\sigma_2(\x,\cdot)$ are Lipschitz continuous. Thus, one can show that for any fixed $\x\in\H$ and any initial data $\y\in\H$, there exists a unique strong solution $\Y_{t}^{\x,\y}\in\mathrm{L}^2(\Omega;\mathrm{L}^{\infty}(0,T;\H)\cap\mathrm{L}^2(0,T;\V))\cap\mathrm{L}^{r+1}(\Omega;\mathrm{L}^{r+1}(0,T;\wi\L^{r+1}))$ to the system \eqref{3.74} with a continuous modification with paths in $\C([0,T];\H)\cap\mathrm{L}^2(0,T;\V)\cap\mathrm{L}^{r+1}(0,T;\wi\L^{r+1})$, $\mathbb{P}$-a.s.  A proof of this result can be obtained in a same way as in Theorem 3.7, \cite{MTM8} by making use of the monotonicity property of the linear and nonlinear operators (see Lemmas \ref{thm2.2}-\ref{lem2.8}) as well as a stochastic generalization of the Minty-Browder technique (localized version for the case $n=2$ and $r\in[1,3]$). Furthermore, the strong solution satisfies the following  infinite dimensional It\^o formula (energy equality):
\begin{align}
&\|\Y_t^{\x,\y}\|_{\H}^2+2\mu\int_0^t\|\Y_s\|_{\V}^2\d s+2\alpha\int_0^t\|\Y_s\|_{\H}^2\d s+2\beta\int_0^t\|\Y_s\|_{\wi\L^{r+1}}^{r+1}\d s\nonumber\\&=\|\y\|_{\H}^2-2\int_0^t(\G(\x,\Y_s),\Y_s)\d s+\int_0^t\|\sigma_s(\x,\Y_s)\Q_2^{1/2}\|_{\mathcal{L}_2}^2\d s+2\int_0^t(\sigma_s(\x,\Y_s)\Q_2^{1/2}\d\bar\W_s,\Y_s),
\end{align}
$\mathbb{P}\text{-a.s.,}$ for all $t\in[0,T]$.  Let $\mathrm{P}_t^{\x}$ be the transition semigroup associated with the process $\Y_t^{\x,\y}$, that is, for any bounded measurable function $\varphi$ on $\H$, we have 
\begin{align}
\mathrm{P}_t^{\x}\varphi(\y)=\E\left[\varphi(\Y_t^{\x,\y})\right], \ \y\in\H \ \text{ and }\ t>0. 
\end{align} 
For the system \eqref{3.74}, the following result is available in the work \cite{MTM11}.

\begin{proposition}[Proposition 4.4, \cite{MTM11}]\label{prop3.12}
	For any given $\x,\y\in\H$, there exists a unique invariant measure for the system \eqref{3.74}. Furthermore, there exists a constant $C_{\mu,\alpha,\lambda_1,L_{\G}}>0$ and $\zeta>0$ such that for any Lipschitz function $\varphi:\H\to\R$, we have 
	\begin{align}\label{393}
	\left|\mathrm{P}_t^{\x}\varphi(\y)-\int_{\H}\varphi(\z)\nu^{\x}(\d\z)\right|\leq C_{\mu,\alpha,\lambda_1,L_{\G}}(1+\|\x\|_{\H}+\|\y\|_{\H})e^{-\frac{\zeta t}{2}}\|\varphi\|_{\mathrm{Lip}(\H)},
	\end{align}
	where $\zeta=2\mu\lambda_1+2\alpha-2L_{\G}-L_{\sigma_2}^2>0$ and $ \|\varphi\|_{\mathrm{Lip}(\H)}=\sup\limits_{\x,\y\in\H}\frac{|\varphi(\x)-\varphi(\y)|}{\|\x-\y\|_{\H}}$. 
\end{proposition}

The interested readers are referred to see \cite{GDJZ,ADe,FFBM,MHJC}, etc for more details on the invariant measures and ergodicity for the infinite dimensional dynamical systems and stochastic Navier-Stokes equations. We need the following lemma in the sequel. 
\begin{lemma}
	There exists a constant $C>0$ such that for any $\x_1,\x_2,\y\in\H$, we have
	\begin{align}\label{394}
	\sup_{t\geq 0}\E\left[\|\Y_t^{\x_1,\y}-\Y_t^{\x_2,\y}\|_{\H}^2\right]\leq C_{\mu,\alpha,\lambda_1,L_{\G},L_{\sigma_2}}\|\x_1-\x_2\|_{\H}^2. 
	\end{align}
\end{lemma}
\begin{proof}
	We know that $\mathbf{W}_{t}:=\Y_{t}^{\x_1,\y}-\Y_{t}^{\x_2,\y}$ satisfies the following system:
	\begin{equation}
	\left\{
	\begin{aligned}
	\d\mathbf{W}_t&=-[\mu\A\mathbf{W}_t+\alpha\mathbf{W}_t+\beta(\mathcal{C}(\Y_{t}^{\x_1,\y})-\mathcal{C}(\Y_{t}^{\x_2,\y}))]\d t+[\G(\x_1,\Y_t^{\x_1,\y})-\G(\x_2,\Y_t^{\x_2,\y})]\d t\\&\quad +[\sigma_2(\x_1,\Y_t^{\x_1,\y})-\sigma_2(\x_2,\Y_t^{\x_2,\y})]\Q_2^{1/2}\d\bar\W_t,\\
	\mathbf{W}_0&=\mathbf{0}. 
	\end{aligned}\right. 
	\end{equation}
	Applying the infinite dimensional It\^o formula to the process $\|\mathbf{W}_t\|_{\H}^2$, we find 
	\begin{align}\label{4p10}
	&	\|\mathbf{W}_t\|_{\H}^2+2\mu\int_0^t\|\mathbf{W}_s\|_{\V}^2\d s+2\alpha\int_0^t\|\mathbf{W}_s\|_{\H}^2\d s+2\beta\int_0^t\langle\mathcal{C}(\Y_{s}^{\x_1,\y})-\mathcal{C}(\Y_{s}^{\x_2,\y}),\mathbf{W}_s\rangle\d s\nonumber\\&\leq  2\int_0^t([\G(\x_1,\Y_s^{\x_1,\y})-\G(\x_2,\Y_s^{\x_2,\y})],\mathbf{W}_s)\d s\nonumber\\&\quad+\int_0^t\|[\sigma_2(\x_1,\Y_s^{\x_1,\y})-\sigma_2(\x_2,\Y_s^{\x_2,\y})]\Q_2^{1/2}\|_{\mathcal{L}_2}^2\d s\nonumber\\&\quad+2\int_0^t([\sigma_2(\x_1,\Y_s^{\x_1,\y})-\sigma_2(\x_2,\Y_s^{\x_2,\y})]\Q_2^{1/2},\mathbf{W}_s), \ \mathbb{P}\text{-a.s.},
	\end{align}
for all $t\in[0,T]$.	Taking expectation in \eqref{4p10} and using the fact the final term appearing the right hand side of the equality \eqref{4p10} is a martingale, we get 
	\begin{align}
	&	\E\left[\|\mathbf{W}_t\|_{\H}^2\right]+2\mu\E\left[\int_0^t\|\mathbf{W}_s\|_{\V}^2\d s\right]+2\alpha\E\left[\int_0^t\|\mathbf{W}_s\|_{\H}^2\d s\right]\nonumber\\&=-2\beta\E\left[\int_0^t\langle\mathcal{C}(\Y_{s}^{\x_1,\y})-\mathcal{C}(\Y_{s}^{\x_2,\y}),\mathbf{W}_s\rangle\d s\right]\nonumber\\&\quad+2\E\left[\int_0^t([\G(\x_1,\Y_s^{\x_1,\y})-\G(\x_2,\Y_s^{\x_2,\y})],\mathbf{W}_s)\d s\right]\nonumber\\&\quad+\E\left[\int_0^t\|[\sigma_2(\x_1,\Y_s^{\x_1,\y})-\sigma_2(\x_2,\Y_s^{\x_2,\y})]\Q_2^{1/2}\|_{\mathcal{L}_2}^2\d s\right],
	\end{align}
	so that using the Assumption \ref{ass3.6}  (A1) and \eqref{214},  we have 
	\begin{align}
	\frac{\d}{\d t}	\E\left[\|\mathbf{W}_t\|_{\H}^2\right]&=-2\mu\E\left[\|\mathbf{W}_t\|_{\V}^2\right]-2\alpha\E\left[\|\mathbf{W}_t\|_{\H}^2\right]-2\beta\E\left[\langle\mathcal{C}(\Y_{t}^{\x_1,\y})-\mathcal{C}(\Y_{t}^{\x_2,\y}),\mathbf{W}_t\rangle\right]\nonumber\\&\quad+2\E\left[([\G(\x_1,\Y_t^{\x_1,\y})-\G(\x_2,\Y_t^{\x_2,\y})],\mathbf{W}_t)\right]\nonumber\\&\quad+\E\left[\|[\sigma_2(\x_1,\Y_t^{\x_1,\y})-\sigma_2(\x_2,\Y_t^{\x_2,\y})]\Q_2^{1/2}\|_{\mathcal{L}_2}^2\right]\nonumber\\&\leq-(2\mu\lambda_1+2\alpha)\E\left[\|\mathbf{W}_t\|_{\H}^2\right]-\frac{\beta}{2^{r-2}}\E\left[\|\mathbf{W}_t\|_{\wi\L^{r+1}}^{r+1}\right]\nonumber\\&\quad+2\E\left[C\|\x_1-\x_2\|_{\H}\|\mathbf{W}_t\|_{\H}+L_{\G}\|\mathbf{W}_t\|_{\H}^2\right]\nonumber\\&\quad+\E\left[\left(C\|\x_1-\x_2\|_{\H}+L_{\sigma_2}\|\mathbf{W}_t\|_{\H}\right)^2\right]\nonumber\\&\leq -\left(\mu\lambda_1+2\alpha-2L_{\G}-2L_{\sigma_2}^2\right)\E\left[\|\mathbf{W}_t\|_{\H}^2\right]+C\left(1+\frac{1}{\mu\lambda_1}\right)\|\x_1-\x_2\|_{\H}^2,
	\end{align}
	for a.e. $t\in[0,T]$. Using the variation of constants formula, we further have 
	\begin{align}
	\E\left[\|\mathbf{W}_t\|_{\H}^2\right]\leq C\left(1+\frac{1}{\mu\lambda_1}\right)\|\x_1-\x_2\|_{\H}^2\int_0^te^{-\xi(t-s)}\d s\leq \frac{C}{\xi}\left(1+\frac{1}{\mu\lambda_1}\right)\|\x_1-\x_2\|_{\H}^2,
	\end{align}
	for any $t>0$, where $\xi=\mu\lambda_1+2\alpha-2L_{\G}-2L_{\sigma_2}^2>0$ and the estimate \eqref{394} follows. 
\end{proof}

\subsection{Preliminaries}
In this subsection, we provide some preliminaries regarding the Large deviation principle (LDP). Let us denote by $\mathscr{E}$, a complete separable metric space  (Polish space)  with the Borel $\sigma$-field $\mathscr{B}(\mathscr{E})$.

\begin{definition}
	A function $\mathrm{I} : \mathscr{E}\rightarrow [0, \infty]$ is called a \emph{rate	function} if $\I$ is lower semicontinuous. A rate function $\I$ is	called a \emph{good rate function,} if for arbitrary $M \in [0,	\infty)$, the level set $\mathcal{K}_M = \big\{x\in\mathscr{E}: \I(x)\leq M\big\}$ is compact in $\mathscr{E}$.
\end{definition}
\begin{definition}[Large deviation principle]\label{LDP} Let $\I$ be a rate function defined on $\mathscr{E}$. A family $\big\{\mathrm{X}^{\varepsilon}: \varepsilon
	> 0\big\}$ of $\mathscr{E}$-valued random elements is said to satisfy \emph{the large deviation principle} on $\mathscr{E}$ with rate function $\I$, if the following two conditions hold:
	\begin{enumerate}
		\item[(i)] (Large deviation upper bound) For each closed set $\F\subset \mathscr{E}$:
		$$ \limsup_{\varepsilon\rightarrow 0} \varepsilon\log \mathbb{P}\left(\mathrm{X}^{\e}\in\F\right) \leq -\inf_{x\in \F} \I(x).$$
		\item[(ii)] (Large deviation lower bound) For each open set $\G\subset \mathscr{E}$:
		$$ \liminf_{\varepsilon\rightarrow 0}\varepsilon \log \mathbb{P}(\mathrm{X}^{\e}\in\G) \geq -\inf_{x\in \G} \I(x).$$
	\end{enumerate}
\end{definition}

\begin{definition}
	Let $\I$ be a rate function on $\mathscr{E}$. A family $\big\{\mathrm{X}^{\e} :\e > 0\big\}$ of $\mathscr{E}$-valued random elements is said to satisfy the \emph{Laplace principle} on $\mathscr{E}$ with rate function $\I,$ if for each real-valued, bounded and continuous function $h$ defined on $\mathscr{E}$, that is, for $h\in\C_b(\mathscr{E})$,
	\begin{equation}\label{LP}
	\lim_{\varepsilon \rightarrow 0} {\varepsilon }\log
	\mathbb{E}\left\{\exp\left[-
	\frac{1}{\varepsilon}h(\mathrm{X}^{\varepsilon})\right]\right\} = -\inf_{x
		\in \mathscr{E}} \big\{h(x) + \I(x)\big\}. 
	\end{equation}
\end{definition}

\begin{lemma}[Varadhan's Lemma, \cite{Va}]\label{VL}
	Let $\mathscr{E}$ be a Polish space and $\{\mathrm{X}^{\varepsilon}: \varepsilon > 0\}$ be a family of $\mathscr{E}$-valued random elements satisfying LDP with rate function $\I$. Then $\{\mathrm{X}^{\varepsilon}: \varepsilon > 0\}$ satisfies the Laplace principle on $\mathscr{E}$ with the same rate function $\I$.
\end{lemma}
\begin{lemma}[Bryc's Lemma, \cite{DZ}]\label{BL}
	The Laplace principle implies the LDP with the same rate function.
\end{lemma}
It should be noted that Varadhan's Lemma together with Bryc's converse of Varadhan's Lemma state that for Polish space valued random elements, the Laplace principle and the large deviation principle are equivalent. 

\subsection{Functional setting and Budhiraja-Dupuis LDP}\label{sec4.2}
In this subsection, the notation and terminology are built in order to state the large deviations result of Budhiraja and Dupuis \cite{BD1} for Polish space valued random elements. Let us define $$\mathcal{A}:=\left\{\H\text{-valued }\{\mathscr{F}_t\}\text{-predictable processes }h\text{ such that }\int_0^T\|h(s)\|_{\H}^2\d s<+\infty,\ \mathbb{P}\text{-a.s.}\right\}, $$  and $$\mathcal{S}_M: = \left\{h\in\mathrm{L}^2(0,T;\H): \int_0^T \|h(s)\|_{\H}^2\d s \leq M\right\}.$$ It is known from \cite{BD2} that the space $\mathcal{S}_M$ is a compact metric space under the metric $\widetilde{d}(\u,\v)=\sum\limits_{j=1}^{\infty}\frac{1}{2^j}\left|\int_0^T(\u(s)-\v(s),\widetilde{e}_j(s))\d s\right|$, where $\{\widetilde{e}_j\}_{j=1}^{\infty}$ are orthonormal basis of $\mathrm{L}^2(0,T;\H)$. Since every compact metric space is complete, the set $\mathcal{S}_M$ endowed with the weak topology obtained from the metric $\widetilde{d}$ is a Polish space. Let us now define $$\mathcal{A}_M=\big\{h\in\mathcal{A}: h(\omega)\in \mathcal{S}_M, \ \mathbb{P}\text{-a.s.}\big\}.$$ Next, we state an important lemma regarding the convergence of the sequence  $\int_0^{\cdot}h_n(s)\d s$, which is useful in proving compactness as well as weak convergence results.

\begin{lemma}[Lemma 3.2, \cite{BD1}]
	Let $\{h_n\}$ be a sequence of elements from $\mathcal{A}_M,$ for some $0<M <+\infty$. Let the sequence $\{h_n\}$ converges in distribution to $h$ with respect to the weak topology on 	$\mathrm{L}^2(0,T;\H)$. Then $\int_0^{\cdot}h_n(s)\d s$ converges in distribution as $\C([0,T];\H)$-valued processes to $\int_0^{\cdot}h(s)\d s$ as $n\to\infty$.
\end{lemma} Let $\mathscr{E}$ denotes a Polish space, and for $\e>0$, let $\mathcal{G}^{\e} :\C([0,T];\H)\to\mathscr{E}$ be a measurable map. Let us define $$\mathrm{X}^{\e} =\mathcal{G}^{\e}(\W(\cdot)).$$ We are interested in the large deviation principle for $\mathrm{X}^{\e}$ as $\e\to 0$.
\begin{hypothesis}\label{hyp1}
	There exists a measurable map $\mathcal{G}^0 : \C ([0, T ] ;\H )\to\mathscr{ E}$ such that the following
	hold:
	\begin{enumerate}
		\item [(i)] Let $\{h^{\e} :\e>0\}\subset \mathcal{A}_{M},$  for some $M<+\infty$. Let $h^{\e}$ converge in distribution as an $\mathcal{S}_M$-valued random element to $h$ as $\e\to0$. Then $\mathcal{G}^{\e}(\W(\cdot)+\frac{1}{\sqrt{\e}}\int_0^{\cdot}h^{\e}(s)\d s)$ converges in distribution to $\mathcal{G}^0(\int_0^{\cdot}h(s)\d s)$ as $\e\to0$. 
		\item [(ii)] For every $M<+\infty$, the set $$\mathcal{K}_M=\left\{\mathcal{G}^0\left(\int_0^{\cdot}h(s)\d s\right):h\in \mathcal{S}_M\right\}$$ is a compact subset of $\mathscr{E}$.
	\end{enumerate}
\end{hypothesis}
For each $f\in\mathscr{E}$, we define
\begin{align}\label{rate}
\I(f):=\inf_{\left\{h\in\mathrm{L}^2(0,T;\H):f=\mathcal{G}^0\left(\int_0^{\cdot}h(s)\d s\right)\right\}}\left\{\frac{1}{2}\int_0^T\|h(s)\|_{\H}^2\d s\right\},
\end{align}
where infimum over an empty set is taken as $+\infty$. Next, we state an important result due to Budhiraja and Dupuis \cite{BD1}.
\begin{theorem}[Budhiraja-Dupuis principle, Theorem 4.4, \cite{BD1}]\label{BD}
	Let $\mathrm{X}^{\e}= \mathcal{G}^{\e}(\W(\cdot))$. If $\{\mathcal{G}^{\e}\}$ satisfies the Hypothesis \ref{hyp1}, then the family $\{\mathrm{X}^{\e}:\e>0\}$ satisfies the Laplace principle in $\mathscr{E}$ with rate function $\I$ given by (\ref{rate}).
\end{theorem}
It should be noted that Hypothesis \ref{hyp1} (i) is a statement on the weak convergence of a certain family of random variables and is at the core of weak convergence approach to the study of large deviations. Hypothesis \ref{hyp1} (ii) says that the level sets of the rate function are compact.

\subsection{LDP for SCBF equations} Let us recall that the system (\ref{2.18a})
has an $\mathscr{F}_t$-adapted pathwise unique strong solution $(\X^{\e,\delta}_t,\Y^{\e,\delta}_t)$  in the Polish space $$\C([0,T];\H)\cap\mathrm{L}^2(0,T;\V)\cap\mathrm{L}^{r+1}(0,T;\widetilde\L^{r+1})\times \C([0,T];\H)\cap\mathrm{L}^2(0,T;\V)\cap\mathrm{L}^{r+1}(0,T;\widetilde\L^{r+1}), \ \mathbb{P}\text{-a.s.}$$ The solution to the first equation in (\ref{2.18a}) denoted by $\X^{\e,\delta}_{\cdot}$ can be written as $\mathcal{G}^{\e}(\W (\cdot))$, for a Borel measurable function $\mathcal{G}^{\e} : \C([0, T ]; \H)\to \mathscr{E}, $ where $\mathscr{E}=\C([0,T];\H)\cap\mathrm{L}^2(0,T;\V)\cap\mathrm{L}^{r+1}(0,T;\widetilde\L^{r+1})$ (see Corollary 4.2, Chapter X, \cite{VF} for 2D Navier-Stokes equations, see \cite{BD1} also). Our main goal is to verify that such a $\mathcal{G}^{\e}$ satisfies the Hypothesis \ref{hyp1}. Then, applying the Theorem \ref{BD}, the LDP for $\big\{\X^{\e,\delta} : \e > 0\big\}$ in $\mathscr{E}$ can be established. Let us now state  our main result on the Wentzell-Freidlin type large deviation principle for the system \eqref{3.6} ($r\in[1,\infty),$ for $n=2$ and $r\in[3,\infty),$ for $n=3$ with $2\beta\mu>1$ for $r=3$) .  
\begin{theorem}\label{thm4.14}
	Under the Assumption \ref{ass3.6}, $\{\X^{\e,\delta}:\e>0\}$ obeys an LDP on $\C([0, T ]; \H ) \cap \mathrm{L}^2(0, T ; \V )\cap\mathrm{L}^{r+1}(0,T;\wi\L^{r+1})$ with the rate function $\I$ defined in \eqref{rate}. 
\end{theorem}
The LDP for $\big\{\X^{\e,\delta} : \e> 0\big\}$ in $\mathscr{E}$ (Theorem \ref{thm4.14}) is proved in the following way.  We show the well-posedness of certain controlled deterministic and controlled stochastic  equations in $\mathscr{E}$. These results help us to prove the two main results on the compactness of the level sets and weak convergence of the stochastic controlled equation, which verifies the Hypothesis \ref{hyp1}. We exploit the  classical Khasminskii approach based on time discretization in the  weak convergence part of the Hypothesis \ref{hyp1}.

\subsection{Compactness} Let us first verify the Hypothesis \ref{hyp1} (ii) on compactness. 

\begin{theorem}\label{thm5.9}
	Let $h\in\mathrm{L}^2(0,T;\H)$ and the Assumption \ref{ass3.6} be satisfied. Then the following deterministic control system:
	\begin{equation}\label{5.4y}
	\left\{
\begin{aligned}
\d\bar{\X}^{h}_t&=-[\mu\A\bar{\X}^{h}_t+\B(\bar{\X}^{h}_t)+\alpha\bar{\X}^{h}_t+\beta\mathcal{C}(\bar{\X}^{h}_t)]\d t+\bar{\F}(\bar{\X}^{h}_t)\d t+\sigma_1(\bar{\X}^{h}_t)\Q_1^{1/2}h_t\d t, \\ \bar{\X}^{h}_0&=\x, 
\end{aligned}\right. 
	\end{equation}
	has a \emph{unique weak solution}  in $\C([0,T];\H)\cap\mathrm{L}^2(0,T;\V)\cap\mathrm{L}^{r+1}(0,T;\wi\L^{r+1})$, and 
	\begin{align}\label{5.5y}
	&	\sup_{h\in \mathcal{S}_M}\left\{\sup_{t\in[0,T]}\|\bar{\X}^{h}_t\|_{\H}^2+\mu\int_0^T\|\bar{\X}^{h}_t\|_{\V}^2\d t+\alpha\int_0^T\|\bar{\X}^{h}_t\|_{\H}^2\d t+\beta\int_0^T\|\bar{\X}^{h}_t\|_{\wi\L^{r+1}}^{r+1}\d t\right\}\nonumber\\&\leq C_{\mu,\alpha,\lambda_1,L_{\G},L_{\sigma_2},M,T}\left(1+\|\x\|_{\H}^2\right). 
	\end{align}
\end{theorem}
\begin{proof}
From the Assumption \ref{ass3.6}, we know that the operator $\sigma_1(\cdot)\Q_1^{1/2}$ is Lipschitz. 	Since $\F(\cdot,\cdot)$ is Lipschitz, one can show that $\bar{\F}(\cdot)$ is Lipschitz in the following way. Using the Assumption \ref{ass3.6} (A1), \eqref{393} and \eqref{394}, we have 
	\begin{align*}
	&\|\bar{\F}(\x_1)-\bar{\F}(\x_2)\|_{\H}\nonumber\\&=\left\|\int_{\H}\F(\x_1,\z)\mu^{\x_1}(\d\z)-\int_{\H}\F(\x_1,\z)\mu^{\x_1}(\d\z)\right\|_{\H}\nonumber\\&\leq\left\|\int_{\H}\F(\x_1,\z)\mu^{\x_1}(\d\z)-\E\left[\F(\x_1,\Y_t^{\x_1,\y})\right]\right\|_{\H}+\left\|\E\left[\F(\x_2,\Y_t^{\x_2,\y})\right]-\int_{\H}\F(\x_1,\z)\mu^{\x_1}(\d\z)\right\|_{\H}\nonumber\\&\quad+\|\E\left[\F(\x_1,\Y_t^{\x_1,\y})\right]-\E\left[\F(\x_2,\Y_t^{\x_2,\y})\right]\|_{\H}\nonumber\\&\leq C_{\mu,\alpha,\lambda_1,L_{\G}}(1+\|\x_1\|_{\H}+\|\x_2\|_{\H}+\|\y\|_{\H})e^{-\frac{\zeta t}{2}}+C\left(\|\x_1-\x_2\|_{\H}+\E\left[\|\Y_t^{\x_1,\y}-\Y_t^{\x_2,\y}\|_{\H}\right]\right)\nonumber\\&\leq C_{\mu,\alpha,\lambda_1,L_{\G}}(1+\|\x_1\|_{\H}+\|\x_2\|_{\H}+\|\y\|_{\H})e^{-\frac{\zeta t}{2}}+C_{\mu,\alpha,\lambda_1,L_{\G},L_{\sigma_2}}\|\x_1-\x_2\|_{\H}. 
	\end{align*}
	Taking $t\to\infty$, we get 
	\begin{align}\label{3p93}
	&\|\bar{\F}(\x_1)-\bar{\F}(\x_2)\|_{\H}\leq C_{\mu,\alpha,\lambda_1,L_{\G},L_{\sigma_2}}\|\x_1-\x_2\|_{\H}. 
	\end{align}
	Thus, it is immediate that the system \eqref{5.4y} is a Lipschitz perturbation of the CBF equations. The existence and uniqueness of weak solution in the Leray-Hopf sense (satisfying the energy equality) of the system (\ref{5.4y}) can be proved using the  monotonicty as well as demicontinuous properties of the linear and nonlinear operators and  the Minty-Browder technique as in Theorem 3.4, \cite{MTM7}. Thus, we need to show \eqref{5.5y} only. Taking inner product with $\bar{\X}^{h}_t$ to the first equation in \eqref{5.4y}, we find 
	\begin{align}\label{6.7}
	&\frac{1}{2}\frac{\d}{\d t}\|\bar{\X}^{h}_t\|_{\H}^2+\mu\|\bar{\X}^{h}_t\|_{\V}^2+\alpha\|\bar{\X}^{h}_t\|_{\V}^2+\beta\|\bar{\X}^{h}_t\|_{\wi\L^{r+1}}^{r+1}=-(\bar{\F}(\bar{\X}^{h}_t)-\sigma_1(\bar{\X}^{h}_t)\Q_1^{1/2}h_t,\bar{\X}^{h}_t). 
	\end{align}
	since $\langle\B(\bar{\X}^{h}_t),\bar{\X}^{h}_t\rangle=0$. Using the Cauchy-Schwarz inequality, \eqref{3p93} and Young's inequality, we estimate $(\bar{\F}(\bar{\X}^{h}),\bar{\X}^{h})$ as 
	\begin{align}\label{410}
	|(\bar{\F}(\bar{\X}^{h}),\bar{\X}^{h})|&\leq\|\bar{\F}(\bar{\X}^{h})\|_{\H}\|\bar{\X}^{h}\|_{\H}\leq C_{\mu,\alpha,\lambda_1,L_{\G},L_{\sigma_2}}(1+\|\bar{\X}^{h}\|_{\H})\|\bar{\X}^{h}\|_{\H}\nonumber\\&\leq C_{\mu,\alpha,\lambda_1,L_{\G},L_{\sigma_2}}(1+\|\bar{\X}^{h}\|_{\H}^2). 
	\end{align}
	Using the Cauchy-Schwarz and H\"older inequalities, and Assumption \ref{ass3.6} (A.1), we estimate $(\sigma_1(\bar{\X}^{h})\Q_1^{1/2}h,\bar{\X}^{h})$ as 
	\begin{align}\label{6.8}
	|(\sigma_1(\bar{\X}^{h})\Q_1^{1/2}h,\bar{\X}^{h})|&\leq\|\sigma_1(\bar{\X}^{h})\Q_1^{1/2}h\|_{\H}\|\bar{\X}^{h}\|_{\H}\leq\|\sigma_1(\bar{\X}^{h})\Q_1^{1/2}\|_{\mathcal{L}_2}\|h\|_{\H}\|\bar{\X}^{h}\|_{\H}\nonumber\\&\leq\frac{1}{2}\|\bar{\X}^{h}\|_{\H}^2+\frac{1}{2}\|\sigma_1(\bar{\X}^{h})\Q_1^{1/2}\|_{\mathcal{L}_2}^2\|h\|_{\H}^2\nonumber\\&\leq\frac{1}{2}\|\bar{\X}^{h}\|_{\H}^2+C\|h\|_{\H}^2+C\|\bar{\X}^{h}\|^2_{\H}\|h\|_{\H}^2. 
	\end{align}
	Substituting \eqref{410} and \eqref{6.8} in \eqref{6.7}, we obtain 
	\begin{align}\label{6.9}
	&\|\bar{\X}^{h}_t\|_{\H}^2+2\mu\int_0^t\|\bar{\X}^{h}_s\|_{\V}^2\d s+2\alpha\int_0^t\|\bar{\X}^{h}_s\|_{\H}^2\d s+2\beta\int_0^t\|\bar{\X}^{h}_s\|_{\wi\L^{r+1}}^{r+1}\d s\nonumber\\&\leq\|\x\|_{\H}^2+C_{\mu,\alpha,\lambda_1,L_{\G},L_{\sigma_2}}t+C\int_0^t\|h_s\|_{\H}^2\d s+C_{\mu,\alpha,\lambda_1,L_{\G},L_{\sigma_2}}\int_0^t\|\bar{\X}^{h}_s\|_{\H}^2\d s\nonumber\\&\quad+C\int_0^t\|\bar{\X}^{h}_s\|^2_{\H}\|h_s\|_{\H}^2\d s.
	\end{align}
	Applying Gronwall's inequality in \eqref{6.9}, we get  
	\begin{align}
	\|\bar{\X}^{h}_t\|_{\H}^2\leq\left(\|\x\|_{\H}^2+C_{\mu,\alpha,\lambda_1,L_{\G},L_{\sigma_2}}T+C\int_0^T\|h_t\|_{\H}^2\d t\right)\exp\left(C_{\mu,\alpha,\lambda_1,L_{\G},L_{\sigma_2}}T+C\int_0^T\|h_t\|_{\H}^2\d t\right),
	\end{align}
	for all $t\in[0,T]$. Thus, taking $h\in \mathcal{S}_M$, we finally obtain \eqref{5.5y}. 
\end{proof}

\begin{lemma}\label{lem3.13}
For any $\x,\y\in\H$, $T>0$, $\e,\Delta>0$  small enough, there exists a constant $C_{\mu,\alpha,\lambda_1,L_{\G},T}>0$ such that
	\begin{align}\label{3.87}
\sup_{h\in\mathcal{S}_M}\int_0^T\|\bar\X^h_t-\bar\X^h_{t(\Delta)}\|_{\H}^2\d t \leq C_{\mu,\alpha,\beta,\lambda_1,L_{\G},L_{\sigma_2},M,T}\Delta\left(1+\|\x\|_{\H}^{\ell}\right),
	\end{align}
 where $\ell=3$, for $n=2$,  $r\in[1,3)$, and $\ell=2$,  for $n=2,3$, $r\in[3,\infty)$  ($2\beta\mu> 1,$ for $n=r=3$). Here, $t(\Delta):=\left[\frac{t}{\Delta}\right]\Delta$ and $[s]$ stands for the largest integer which is less than or equal $s$. 
\end{lemma}

\begin{proof}
Using \eqref{5.5y},	it can be easily seen that 
	\begin{align}\label{388}
	\int_0^{T}\|\bar\X^h_t-\bar\X^h_{t(\Delta)}\|_{\H}^2\d t&\leq 	\int_0^{\Delta}\|\bar\X^h_t-\x\|_{\H}^2\d t+	\int_{\Delta}^T\|\bar\X^h_t-\bar\X^h_{t(\Delta)}\|_{\H}^2\d t\nonumber\\&\leq C_{\mu,\alpha,\lambda_1,L_{\G},L_{\sigma_2},M,T}\left(1+\|\x\|_{\H}^2\right)\Delta+	2\int_{\Delta}^T\|\bar\X^h_t-\bar\X^h_{t-\Delta}\|_{\H}^2\d t\nonumber\\&\quad+2	\int_{\Delta}^T\|\bar\X^h_{t(\Delta)}-\bar\X^h_{t-\Delta}\|_{\H}^2\d t.
	\end{align}
	Let us first estimate the second term from the right hand side of the equality \eqref{388}. Note that $\bar\X^h_r-\bar\X^h_{t-\Delta}$, for $r\in[t-\Delta,t]$ satisfies the following energy equality: 
	\begin{align}\label{4p26}
	\|\bar\X^h_t-\bar\X^h_{t-\Delta}\|_{\H}^2&=-2\mu\int_{t-\Delta}^t\langle\A\bar\X^h_s,\bar\X^h_s-\bar\X^h_{t-\Delta}\rangle\d s-2\alpha\int_{t-\Delta}^t(\bar\X^h_s,\bar\X^h_s-\bar\X^h_{t-\Delta})\d s\nonumber\\&\quad-2\int_{t-\Delta}^t\langle\B(\bar\X^h_s),\bar\X^h_s-\bar\X^h_{t-\Delta}\rangle\d s-2\beta\int_{t-\Delta}^t\langle\mathcal{C}(\bar\X^h_s),\bar\X^h_s-\bar\X^h_{t-\Delta}\rangle\d s\nonumber\\&\quad-2\int_{t-\Delta}^t(\bar\F(\bar\X^h_s),\bar\X^h_s-\bar\X^h_{t-\Delta})\d s-2\int_{t-\Delta}^t(\sigma_1(\bar\X^h_s)\Q_1^{1/2}h_s,\bar\X^h_s-\bar\X^h_{t-\Delta})\d s\nonumber\\&=:\sum_{i=1}^6I_i(t). 
	\end{align}
	Making use of an integration by parts, H\"older's inequality,  Fubini's Theorem and \eqref{5.5y}, we estimate $\int_{\Delta}^T|I_1(t)|\d t$ as 
	\begin{align}\label{4p27}
	\int_{\Delta}^T|I_1(t)|\d t&\leq 2\mu\int_{\Delta}^T\int_{t-\Delta}^t\|\bar\X^h_s\|_{\V}\|\bar\X^h_s-\bar\X^h_{t-\Delta}\|_{\V}\d s\d t\nonumber\\&\leq 2\mu\left(\int_{\Delta}^T\int_{t-\Delta}^t\|\bar\X^h_s\|_{\V}^2\d s\d t\right)^{1/2} \left(\int_{\Delta}^T\int_{t-\Delta}^t\|\bar\X^h_s-\bar\X^h_{t-\Delta}\|_{\V}^2\d s\d t\right)^{1/2}\nonumber\\&\leq 2\mu\left(\Delta\int_{0}^T\|\bar\X^h_t\|_{\V}^2\d t\right)^{1/2}\left(2\Delta\int_{0}^T\|\bar\X^h_t\|_{\V}^2\d t\right)^{1/2}\nonumber\\&\leq C_{\mu,\alpha,\lambda_1,L_{\G},L_{\sigma_2},M,T}\Delta\left(1+\|\x\|_{\H}^2\right). 
	\end{align}
	Similarly, we estimate the term $\int_{\Delta}^T|I_2(t)|\d t$ as 
	\begin{align}
	\int_{\Delta}^T|I_2(t)|\d t&\leq C\alpha T\Delta \sup_{t\in[0,T]}\|\X^{\e}_t\|_{\H}^2\leq C_{\mu,\alpha,\lambda_1,L_{\G},L_{\sigma_2},M,T}\Delta\left(1+\|\x\|_{\H}^2+\|\y\|_{\H}^2\right).
	\end{align}
		For $n=2$ and $r\in[1,3)$,	using  H\"older's and Ladyzhenskaya's inequalities,  Fubini's Theorem and \eqref{5.5y}, we estimate $\int_{\Delta}^T|I_3(t)|\d t$ as 
	\begin{align}\label{4p29}
	\int_{\Delta}^T|I_3(t)|\d t&\leq 2\int_{\Delta}^T\int_{t-\Delta}^t\|\bar\X^h_s\|_{\wi\L^4}^2\|\bar\X^h_s-\bar\X^h_{t-\Delta}\|_{\V}\d s\d t\nonumber\\&\leq 2\sqrt{2}\left(\int_{\Delta}^T\int_{t-\Delta}^t\|\bar\X^h_s\|_{\H}^2\|\bar\X^h_s\|_{\V}^2\d s\d t\right)^{1/2}\left(\int_{\Delta}^T\int_{t-\Delta}^t\|\bar\X^h_s-\bar\X^h_{t-\Delta}\|_{\V}^2\d s\d t\right)^{1/2}\nonumber\\&\leq 2\sqrt{2}\left(\Delta\sup_{t\in[0,T]}\|\bar\X^h_t\|_{\H}^2\int_0^T\|\bar\X^h_t\|_{\V}^2\d t\right)^{1/2} \left(2\Delta\int_{0}^T\|\bar\X^h_t\|_{\V}^2\d t\right)^{1/2}\nonumber\\&\leq  C_{\mu,\alpha,\lambda_1,L_{\G},L_{\sigma_2},M,T}\Delta\left(1+\|\x\|_{\H}^3\right). 
	\end{align}
	For $n=2,3$ and $r\geq 3$ (take $2\beta\mu>1$, for $r=3$), we estimate $\int_{\Delta}^T|I_3(t)|\d t$  using H\"older's inequality, interpolation inequality and \eqref{5.5y} as 
	\begin{align}\label{4p30}
		\int_{\Delta}^T|I_3(t)|\d t&\leq 2\int_{\Delta}^T\int_{t-\Delta}^t\|\bar\X^h_s\|_{\wi\L^{r+1}}\|\bar\X^h_s\|_{\wi\L^{\frac{2(r+1)}{r-1}}}\|\bar\X^h_s-\bar\X^h_{t-\Delta}\|_{\V}\d s\d t\nonumber\\&\leq 2\int_{\Delta}^T\int_{t-\Delta}^t\|\bar\X^h_s\|_{\H}^{\frac{r-3}{r-1}}\|\bar\X^h_s\|_{\wi\L^{r+1}}^{\frac{r+1}{r-1}}\|\bar\X^h_s-\bar\X^h_{t-\Delta}\|_{\V}\d s\d t \nonumber\\&\leq 2\left(\int_{\Delta}^T\int_{t-\Delta}^t\|\bar\X^h_s\|_{\H}^{\frac{2(r-3)}{r-1}}\|\bar\X^h_s\|_{\wi\L^{r+1}}^{\frac{2(r+1)}{r-1}}\d s\d t\right)^{1/2}\left(\int_{\Delta}^T\int_{t-\Delta}^t\|\bar\X^h_s-\bar\X^h_{t-\Delta}\|_{\V}^2\d s\d t\right)^{1/2}\nonumber\\&\leq 2\left(\Delta\int_{0}^T\|\bar\X^h_t\|_{\H}^{\frac{2(r-3)}{r-1}}\|\bar\X^h_t\|_{\wi\L^{r+1}}^{\frac{2(r+1)}{r-1}}\d t\right)^{1/2}\left(2\Delta\int_{0}^T\|\bar\X^h_t\|_{\V}^2\d t\right)^{1/2}\nonumber\\&\leq 2\Delta T^{\frac{r-3}{2(r-1)}}\left(\sup_{t\in[0,T]}\|\bar\X^h_t\|_{\H}^{r-3}\int_{0}^T\|\bar\X^h_t\|_{\wi\L^{r+1}}^{r+1}\d t\right)^{\frac{1}{r-1}}\left[2\E\left(\int_{0}^T\|\bar\X^h_t\|_{\V}^2\d t\right)\right]^{1/2}\nonumber\\&\leq C_{\mu,\alpha,\beta,\lambda_1,L_{\G},L_{\sigma_2},M,T}\Delta\left(1+\|\x\|_{\H}^2\right).
	\end{align}
	Once again using H\"older's inequality,  Fubini's Theorem and \eqref{5.5y}, we estimate the term  $\int_{\Delta}^T|I_4(t)|\d t$ as 
	\begin{align}\label{4p31}
		\int_{\Delta}^T|I_4(t)|\d t&\leq 2\beta\int_{\Delta}^T\int_{t-\Delta}^t\|\bar\X^h_s\|_{\wi\L^{r+1}}^r\|\bar\X^h_s-\bar\X^h_{t-\Delta}\|_{\wi\L^{r+1}}\d s\d t\nonumber\\&\leq 2\beta\left(\Delta\int_{0}^T\|\bar\X^h_t\|_{\wi\L^{r+1}}^{r+1}\d t\right)^{\frac{r}{r+1}}\left(2^{r}\Delta\int_0^T\|\bar\X^h_t\|_{\wi\L^{r+1}}^{r+1}\d t\right)^{\frac{1}{r+1}}\nonumber\\&\leq C_{\mu,\alpha,\beta,\lambda_1,L_{\G},L_{\sigma_2},M,T}\Delta\left(1+\|\x\|_{\H}^2\right).
	\end{align}
		We estimate $\int_{\Delta}^T|I_5(t)|\d t$ using the Assumption \ref{ass3.6} (A1) and \eqref{5.5y} as 
	\begin{align}\label{4p32}
		\int_{\Delta}^T|I_5(t)|\d t&\leq 2\int_{\Delta}^T\int_{t-\Delta}^t\|\bar\F(\bar\X^h_s)\|_{\H}\|\bar\X^h_s-\bar\X^h_{t-\Delta}\|_{\H}\d s\d t\nonumber\\&\leq 2\left(\int_{\Delta}^T\int_{t-\Delta}^t\|\bar\F(\bar\X^h_s)\|_{\H}^2\d s\d t\right)^{1/2}\left(\int_{\Delta}^T\int_{t-\Delta}^t\|\bar\X^h_s-\bar\X^h_{t-\Delta}\|_{\H}^2\d s\d t\right)^{1/2}\nonumber\\&\leq C\left(\Delta\int_{0}^T(1+\|\bar\X^h_t\|_{\H}^2)\d t\right)^{1/2}\left(2\Delta\int_{0}^T\|\bar\X^h_t\|_{\H}^2\d t\right)^{1/2}\nonumber\\&\leq C_{\mu,\alpha,\lambda_1,L_{\G},L_{\sigma_2},M,T}\Delta(1+\|\x\|_{\H}^2). 
	\end{align}
	Similarly, we estimate the term $\int_{\Delta}^T|I_6(t)|\d t$  as 
	\begin{align}\label{433}
	\int_{\Delta}^T|I_6(t)|\d t&\leq 2\int_{\Delta}^T\int_{t-\Delta}^t\|\sigma_1(\bar\X^h_s)\Q_1^{1/2}h_s\|_{\H}\|\bar\X^h_s-\bar\X^h_{t-\Delta}\|_{\H}\d s\d t\nonumber\\&\leq 2\left(\int_{\Delta}^T\int_{t-\Delta}^t\|\sigma_1(\bar\X^h_s)\Q_1^{1/2}\|_{\mathcal{L}_2}^2\|h_s\|_{\H}^2\d s\d t\right)^{1/2}\left(\int_{\Delta}^T\int_{t-\Delta}^t\|\bar\X^h_s-\bar\X^h_{t-\Delta}\|_{\H}^2\d s\d t\right)^{1/2}\nonumber\\&\leq C\left(\Delta\sup_{t\in[0,T]}(1+\|\bar\X^h_t\|_{\H}^2)\int_0^T\|h_t\|_{\H}^2\d t\right)^{1/2} \left(2\Delta\int_{0}^T\|\bar\X^h_t\|_{\H}^2\d t\right)^{1/2}\nonumber\\&\leq C_{\mu,\alpha,\lambda_1,L_{\G},L_{\sigma_2},M,T}\Delta(1+\|\x\|_{\H}^2), 
	\end{align}
	since $h\in\mathcal{S}_M$. Combing \eqref{4p26}-\eqref{433}, we obtain the required result \eqref{3.87}. 
\end{proof}

We are now in a position to verify the Hypothesis \ref{hyp1} (ii). 
\begin{theorem}[Compactness]\label{compact}
	Let $M <+\infty$ be a fixed positive number. Let $$\mathcal{K}_M :=\big\{ \bar{\X}^h \in\C([0,T];\H)\cap \mathrm{L}^2(0,T;\V )\cap\mathrm{L}^{r+1}(0,T;\wi\L^{r+1}):h\in \mathcal{S}_M\big\},$$
	where $\bar{\X}_t^h$ is the unique Leray-Hopf weak solution of the deterministic controlled equation (\ref{5.4y}), and $\bar{\X}_0^h= \x \in\H$ in $\mathscr{E} = \C([0,T];\H)\cap\mathrm{L}^2(0,T;\V)\cap\mathrm{L}^{r+1}(0,T;\wi\L^{r+1})$. Then $\mathcal{K}_M$ is compact in $\mathscr{E}$.
\end{theorem}
\begin{proof}
	Let us consider a sequence $\{\bar{\X}^{h_n}\}$ in $\mathcal{K}_M$, where $\bar{\X}^{h_n}$ corresponds to the solution of (\ref{5.4y}) with control $h_n \in \mathcal{S}_M$ in place of $h$, that is, 
	\begin{equation}\label{5.20z}
	\left\{
	\begin{aligned}
\d\bar{\X}_t^{h_n}&=-[\mu\A\bar{\X}_t^{h_n}+\B(\bar{\X}_t^{h_n})+\alpha\bar{\X}_t^{h_n}+\beta\mathcal{C}(\bar{\X}_t^{h_n})]\d t+\bar{\F}(\bar{\X}_t^{h_n})\d t+\sigma_1(\bar{\X}_t^{h_n})\Q_1^{1/2}h_t^n\d t, \\ \bar{\X}_0^{h_n}&=\x\in\H. 
	\end{aligned}
	\right.
	\end{equation}
	Then, by using the weak compactness of $\mathcal{S}_M$, there exists a subsequence of $\{h_n\}$, (still denoted by $\{h_n\}$), which converges weakly to $ h\in \mathcal{S}_M$ in $\mathrm{L}^2(0, T ; \H)$. Using the estimates (\ref{5.5y}), we obtain 
	\begin{equation}\label{5.29}
	\left\{
	\begin{aligned}
\bar{\X}_t^{h_n}&\xrightarrow{w^*}\bar{\X}_t^{h}\ \text{ in }\ \mathrm{L}^{\infty}(0,T;\H),\\ 
\bar{\X}_t^{h_n}&\xrightarrow{w}\bar{\X}_t^{h}\ \text{ in }\ \mathrm{L}^2(0,T;\V),\\
\bar{\X}_t^{h_n}&\xrightarrow{w}\bar{\X}_t^{h}\ \text{ in }\ \mathrm{L}^{r+1}(0,T;\wi\L^{r+1}).
	\end{aligned}
	\right.
	\end{equation}
Using the monotonicity property of linear and nonlinear operators, and Minty-Browder technique (see \cite{MTM7}), one can establish that $\bar{\X}_t^{h}$ is the unique weak solution of the system (\ref{5.4y}).	In order to prove $\mathcal{K}_M$ is compact, we need to show that $\bar{\X}_t^{h_n}\to\bar{\X}_t^{h}$ in $\mathscr{E}$ as $n\to\infty$. In other words, it is required to show that 
	\begin{align}\label{5.30z}
	\sup_{t\in[0,T]}\|\bar{\X}_t^{h_n}-\bar{\X}_t^{h}\|_{\H}^2+\int_0^T\|\bar{\X}_t^{h_n}-\bar{\X}_t^{h}\|_{\V}^2\d t+\int_0^T\|\bar{\X}_t^{h_n}-\bar{\X}_t^{h}\|_{\wi\L^{r+1}}^{r+1}\d t\to 0,
	\end{align}
	as $n\to\infty$. Recall that for the system (\ref{5.4y}), the energy estimate given in (\ref{5.5y}) holds true. Let us now define $\mathbf{W}_t^{h_n,h}:=\bar{\X}_t^{h_n}-\bar{\X}_t^{h}$, so that $\mathbf{W}_t^{h_n,h}$ satisfies: 
	\begin{equation}\label{5.22z}
	\left\{
	\begin{aligned}
	\d\mathbf{W}_t^{h_n,h}&=-\left[\mu\A\mathbf{W}_t^{h_n,h}+(\B(\bar{\X}_t^{h_n})-\B(\bar{\X}_t^{h}))+\alpha\mathbf{W}_t^{h_n,h}+\beta(\mathcal{C}(\bar{\X}_t^{h_n})-\mathcal{C}(\bar{\X}_t^{h}))\right]\d t\\&\quad +\left[\bar{\F}(\bar{\X}_t^{h_n})-\bar{\F}(\bar{\X}_t^{h})\right]\d t +\left[\sigma_1(\bar{\X}_t^{h_n})\Q_1^{1/2}h_t^n-\sigma_1(\bar{\X}_t^{h})\Q_1^{1/2}h_t\right]\d t,\\
	\mathbf{W}_0^{h_n,h}&=\mathbf{0}.
	\end{aligned}
	\right.\end{equation}
	Taking inner product with $\mathbf{W}_t^{h_n,h}$ to the first equation in (\ref{5.22z}), we get 
	\begin{align}\label{5.23z}
	&\|\mathbf{W}_t^{h_n,h}\|_{\H}^2+2\mu\int_0^t\|\mathbf{W}_s^{h_n,h}\|_{\V}^2\d s+2\alpha\int_0^t\|\mathbf{W}_s^{h_n,h}\|_{\H}^2\d s\nonumber\\&\quad+2\beta\int_0^t\langle\mathcal{C}(\bar{\X}_s^{h_n})-\mathcal{C}(\bar{\X}_s^{h}),\mathbf{W}_s^{h_n,h}\rangle\d s\nonumber\\&=-2\int_0^t\langle\B(\bar{\X}_s^{h_n})-\B(\bar{\X}_s^{h}),\mathbf{W}_s^{h_n,h}\rangle\d s+2\int_0^t(\bar{\F}(\bar{\X}_s^{h_n})-\bar{\F}(\bar{\X}_s^{h}),\mathbf{W}_s^{h_n,h})\d s\nonumber\\&\quad+2\int_0^t(\sigma_1(\bar{\X}_s^{h_n})\Q_1^{1/2}h_s^n-\sigma_1(\bar{\X}_s^{h})\Q_1^{1/2}h_s,\mathbf{W}_s^{h_n,h})\d s.
	\end{align}
	Note that $\langle\B(\bar{\X}^{h_n},\bar{\X}^{h_n}-\bar{\X}^{h}),\bar{\X}^{h_n}-\bar{\X}^{h}\rangle=0$ and it implies that
	\begin{align}\label{6.28}
&	\langle \B(\bar{\X}^{h_n})-\B(\bar{\X}^{h}),\bar{\X}^{h_n}-\bar{\X}^{h}\rangle \nonumber\\&=\langle\B(\bar{\X}^{h_n},\bar{\X}^{h_n}-\bar{\X}^{h}),\bar{\X}^{h_n}-\bar{\X}^{h}\rangle +\langle \B(\bar{\X}^{h_n}-\bar{\X}^{h},\bar{\X}^{h}),\bar{\X}^{h_n}-\bar{\X}^{h}\rangle \nonumber\\&=\langle \B(\bar{\X}^{h_n}-\bar{\X}^{h},\bar{\X}^{h}),\bar{\X}^{h_n}-\bar{\X}^{h}\rangle=-\langle \B(\bar{\X}^{h_n}-\bar{\X}^{h},\bar{\X}^{h_n}-\bar{\X}^{h}),\bar{\X}^{h}\rangle.
	\end{align} 
For $n=2$ and $r\in[1,3]$,	 making use of H\"older's, Ladyzhenskaya's and Young's inequalities, we estimate $2|\langle\B(\bar{\X}^{h_n})-\B(\bar{\X}^{h}),\mathbf{W}^{h_n,h}\rangle|$ as 
	\begin{align}\label{419}
	2|\langle\B(\bar{\X}^{h_n})-\B(\bar{\X}^{h}),\mathbf{W}^{h_n,h}\rangle|&= 2|\langle\B(\mathbf{W}^{h_n,h},\bar{\X}^{h}),\mathbf{W}^{h_n,h}\rangle |\leq 2\|\bar{\X}^{h}\|_{\V}\|\mathbf{W}^{h_n,h}\|_{\wi\L^4}^2\nonumber\\&\leq 2\sqrt{2}\|\bar{\X}^{h}\|_{\V}\|\mathbf{W}^{h_n,h}\|_{\H}\|\mathbf{W}^{h_n,h}\|_{\V}\nonumber\\&\leq\mu\|\mathbf{W}^{h_n,h}\|_{\V}^2+\frac{2}{\mu}\|\bar{\X}^{h}\|_{\V}^2\|\mathbf{W}^{h_n,h}\|_{\H}^2. 
	\end{align}
	Using \eqref{2.23} and \eqref{a215}, we know that 
	\begin{align}
	-2\beta\langle\mathcal{C}(\bar{\X}^{h_n})-\mathcal{C}(\bar{\X}^{h}),\mathbf{W}^{h_n,h}\rangle\leq-\frac{\beta}{2^{r-2}}\|\mathbf{W}^{h_n,h}\|_{\wi\L^{r+1}}^{r+1},
	\end{align}
	for $r\in[1,\infty)$. Using \eqref{3p93}, H\"older's and Young's inequalities, we estimate $2(\bar{\F}(\bar{\X}^{h_n})-\bar{\F}(\bar{\X}^{h}),\mathbf{W}^{h_n,h})$ as 
	\begin{align}\label{4.21}
	2(\bar{\F}(\bar{\X}^{h_n})-\bar{\F}(\bar{\X}^{h}),\mathbf{W}^{h_n,h})&\leq 2\|\bar{\F}(\bar{\X}^{h_n})-\bar{\F}(\bar{\X}^{h})\|_{\H}\|\mathbf{W}^{h_n,h}\|_{\H}\leq C_{\mu,\alpha,\lambda_1,L_{\G},L_{\sigma_2}}\|\mathbf{W}^{h_n,h}\|_{\H}^2. 
	\end{align}
	Combining \eqref{419}-\eqref{4.21} and substituting it in \eqref{5.23z}, we obtain 
	\begin{align}\label{4.22}
	&\|\mathbf{W}_t^{h_n,h}\|_{\H}^2+\mu\int_0^t\|\mathbf{W}_s^{h_n,h}\|_{\V}^2\d s+2\alpha\int_0^t\|\mathbf{W}_s^{h_n,h}\|_{\H}^2\d s+\frac{\beta}{2^{r-2}}\int_0^t\|\mathbf{W}^{h_n,h}_s\|_{\wi\L^{r+1}}^{r+1}\d s\nonumber\\&\leq \frac{2}{\mu}\int_0^t\|\bar{\X}^{h}_s\|_{\V}^2\|\mathbf{W}^{h_n,h}_s\|_{\H}^2\d s+C_{\mu,\alpha,\lambda_1,L_{\G},L_{\sigma_2}}\int_0^t\|\mathbf{W}^{h_n,h}_s\|_{\H}^2\d s+I_1+I_2,
	\end{align}
	where 
	\begin{align*}
	I_1&=2\int_0^t((\sigma_1(\bar{\X}_s^{h_n})-\sigma_1(\bar{\X}_s^{h}))\Q_1^{1/2}h_s^n,\mathbf{W}_s^{h_n,h})\d s,\\
	I_2&=2\int_0^t(\sigma_1(\bar{\X}_s^{h})\Q_1^{1/2}(h_s^n-h_s),\mathbf{W}_s^{h_n,h})\d s.
	\end{align*}
	Using the Cauchy-Schwarz inequality, H\"older's and Young's inequalities, and the Assumption \ref{ass3.6} (A1), we estimate $I_1$ as  
	\begin{align}\label{634}
	I_1	&\leq 2\int_0^t|((\sigma_1(\bar{\X}_s^{h_n})-\sigma_1(\bar{\X}_s^{h}))\Q_1^{1/2}h_s^n,\mathbf{W}_s^{h_n,h})|\d s\nonumber\\&\leq 2\int_0^t\|(\sigma_1(\bar{\X}_s^{h_n})-\sigma_1(\bar{\X}_s^{h}))\Q_1^{1/2}\|_{\mathcal{L}_2}\|h_s^n\|_{\H}\|\mathbf{W}_s^{h_n,h}\|_{\H}\d s\leq C\int_0^t\|h_s^n\|_{\H}\|\mathbf{W}_s^{h_n,h}\|_{\H}^2\d s\nonumber\\&\leq C\int_0^t\|\mathbf{W}_s^{h_n,h}\|_{\H}^2\d s+C\int_0^t\|h_s^n\|_{\H}^2\|\mathbf{W}_s^{h_n,h}\|_{\H}^2\d s.
	\end{align}
	Making use of the Cauchy-Schwarz inequality and Young's inequality, we estimate $I_2$ as 
	\begin{align}\label{635}
	I_2&\leq 2\int_0^t\|\sigma_1(\bar{\X}_s^{h})\Q_1^{1/2}(h_s^n-h_s)\|_{\H}\|\mathbf{W}_s^{h_n,h}\|_{\H}\d s\nonumber\\&\leq \int_0^t\|\mathbf{W}_s^{h_n,h}\|_{\H}^2\d s+\int_0^t\|\sigma_1(\bar{\X}_s^{h})\Q_1^{1/2}(h_s^n-h_s)\|_{\H}^2\d s.
	\end{align}
	Thus, it is immediate from \eqref{4.22} that 
	\begin{align}\label{425}
	&\|\mathbf{W}_t^{h_n,h}\|_{\H}^2+\mu\int_0^t\|\mathbf{W}_s^{h_n,h}\|_{\V}^2\d s+2\alpha\int_0^t\|\mathbf{W}_s^{h_n,h}\|_{\H}^2\d s+\frac{\beta}{2^{r-2}}\int_0^t\|\mathbf{W}^{h_n,h}_s\|_{\wi\L^{r+1}}^{r+1}\d s\nonumber\\&\leq \frac{2}{\mu}\int_0^t\|\bar{\X}^{h}_s\|_{\V}^2\|\mathbf{W}^{h_n,h}_s\|_{\H}^2\d s+C_{\mu,\alpha,\lambda_1,L_{\G},L_{\sigma_2}}\int_0^t\|\mathbf{W}^{h_n,h}_s\|_{\H}^2\d s+C\int_0^t\|h_s^n\|_{\H}^2\|\mathbf{W}_s^{h_n,h}\|_{\H}^2\d s\nonumber\\&\quad+\int_0^t\|\sigma_1(\bar{\X}_s^{h})\Q_1^{1/2}(h_s^n-h_s)\|_{\H}^2\d s,
	\end{align}
for all $t\in[0,T]$. An application of Gronwall's inequality in \eqref{425} gives 
\begin{align}\label{426}
&\sup_{t\in[0,T]}\|\mathbf{W}_t^{h_n,h}\|_{\H}^2+\mu\int_0^T\|\mathbf{W}_t^{h_n,h}\|_{\V}^2\d t+2\alpha\int_0^T\|\mathbf{W}_t^{h_n,h}\|_{\H}^2\d t+\frac{\beta}{2^{r-2}}\int_0^T\|\mathbf{W}^{h_n,h}_t\|_{\wi\L^{r+1}}^{r+1}\d t\nonumber\\&\leq \left(\int_0^T\|\sigma_1(\bar{\X}_t^{h})\Q_1^{1/2}(h_t^n-h_t)\|_{\H}^2\d t\right)\exp\left\{\frac{2}{\mu}\int_0^T\|\bar{\X}^{h}_t\|_{\V}^2\d t+\int_0^T\|h_t^n\|_{\H}^2\d t\right\}e^{C_{\mu,\alpha,\lambda_1,L_{\G},L_{\sigma_2}}T}\nonumber\\&\leq C_{\mu,\alpha,\lambda_1,L_{\G},L_{\sigma_2},M,T,\|\x\|_{\H}}\left(\int_0^T\|\sigma_1(\bar{\X}_t^{h})\Q_1^{1/2}(h_t^n-h_t)\|_{\H}^2\d t\right),
\end{align}
since $\{h_n\}\in\mathcal{S}_M$ and the process $\bar{\X}_t^h$ satisfies the energy estimate \eqref{5.5y}. 	It should be noted that the operator $\sigma_1(\cdot)\Q^{\frac{1}{2}}$ is Hilbert-Schmidt in $\H$, and hence it is a compact operator on $\H$. Furthermore, we know that compact operator maps weakly convergent sequences into strongly convergent sequences. Since $\{h_n\}$ converges weakly to $ h\in \mathcal{S}_M$ in $\mathrm{L}^2(0, T ; \H)$, we infer that $$\int_0^T\|\sigma_1(\bar{\X}_t^{h})\Q_1^{1/2}(h_t^n-h_t)\|_{\H}^2\d t\to 0 \ \text{ as } \ n\to\infty.$$ Thus, from \eqref{426}, we obtain 
\begin{align}
&\sup_{t\in[0,T]}\|\mathbf{W}_t^{h_n,h}\|_{\H}^2+\mu\int_0^T\|\mathbf{W}_t^{h_n,h}\|_{\V}^2\d t+\frac{\beta}{2^{r-2}}\int_0^T\|\mathbf{W}^{h_n,h}_t\|_{\wi\L^{r+1}}^{r+1}\d t\to 0\  \text{ as }\  n\ \to \infty,
\end{align}
which concludes the proof for $n=2$ and $r\in[1,3]$. 

For $n=2,3$ and $r\in(3,\infty)$,	from \eqref{2.23}, we easily have 
	\begin{align}\label{6.27}
	\beta	\langle\mathcal{C}(\bar{\X}^{h_n})-\mathcal{C}(\bar{\X}^{h}),\bar{\X}^{h_n}-\bar{\X}^{h}\rangle \geq \frac{\beta}{2}\||\bar{\X}^{h_n}|^{\frac{r-1}{2}}(\bar{\X}^{h_n}-\bar{\X}^{h})\|_{\H}^2+\frac{\beta}{2}\||\bar{\X}^{h}|^{\frac{r-1}{2}}(\bar{\X}^{h_n}-\bar{\X}^{h})\|_{\H}^2. 
	\end{align}
	Using \eqref{6.28}, H\"older's and Young's inequalities, we estimate $|\langle\B(\bar{\X}^{h_n})-\B(\bar{\X}^{h}),\bar{\X}^{h_n}-\bar{\X}^{h}\rangle|$ as  
	\begin{align}\label{6.29}
|\langle\B(\bar{\X}^{h_n})-\B(\bar{\X}^{h}),\bar{\X}^{h_n}-\bar{\X}^{h}\rangle|&=	|\langle\B(\bar{\X}^{h_n}-\bar{\X}^{h},\bar{\X}^{h_n}-\bar{\X}^{h}),\bar{\X}^{h}\rangle|\nonumber\\&\leq\|\bar{\X}^{h_n}-\bar{\X}^{h}\|_{\V}\|\bar{\X}^{h}(\bar{\X}^{h_n}-\bar{\X}^{h})\|_{\H}\nonumber\\&\leq\frac{\mu }{2}\|\bar{\X}^{h_n}-\bar{\X}^{h}\|_{\V}^2+\frac{1}{2\mu }\|\bar{\X}^{h}(\bar{\X}^{h_n}-\bar{\X}^{h})\|_{\H}^2.
	\end{align}
	We take the term $\|\bar{\X}^{h}(\bar{\X}^{h_n}-\bar{\X}^{h})\|_{\H}^2$ from \eqref{6.29} and use H\"older's and Young's inequalities to estimate it as (see \cite{KWH} also)
	\begin{align}\label{6.30}
	&\int_{\mathcal{O}}|\bar{\X}^{h}(x)|^2|\bar{\X}^{h_n}(x)-\bar{\X}^{h}(x)|^2\d x\nonumber\\&=\int_{\mathcal{O}}|\bar{\X}^{h}(x)|^2|\bar{\X}^{h_n}(x)-\bar{\X}^{h}(x)|^{\frac{4}{r-1}}|\bar{\X}^{h_n}(x)-\bar{\X}^{h}(x)|^{\frac{2(r-3)}{r-1}}\d x\nonumber\\&\leq\left(\int_{\mathcal{O}}|\bar{\X}^{h}(x)|^{r-1}|\bar{\X}^{h_n}(x)-\bar{\X}^{h}(x)|^2\d x\right)^{\frac{2}{r-1}}\left(\int_{\mathcal{O}}|\bar{\X}^{h_n}(x)-\bar{\X}^{h}(x)|^2\d x\right)^{\frac{r-3}{r-1}}\nonumber\\&\leq\frac{\beta\mu }{2}\left(\int_{\mathcal{O}}|\bar{\X}^{h}(x)|^{r-1}|\bar{\X}^{h_n}(x)-\bar{\X}^{h}(x)|^2\d x\right)\nonumber\\&\quad+\frac{r-3}{r-1}\left(\frac{4}{\beta\mu (r-1)}\right)^{\frac{2}{r-3}}\left(\int_{\mathcal{O}}|\bar{\X}^{h_n}(x)-\bar{\X}^{h}(x)|^2\d x\right),
	\end{align}
	for $r>3$. Combining \eqref{6.27} and \eqref{6.30}, we find 
	\begin{align}\label{630}
	&\beta\langle\mathcal{C}(\bar{\X}^{h_n})-\mathcal{C}(\bar{\X}^{h}),\bar{\X}^{h_n}-\bar{\X}^{h}\rangle+\langle\B(\bar{\X}^{h_n})-\B(\bar{\X}^{h}),\bar{\X}^{h_n}-\bar{\X}^{h}\rangle\nonumber\\&\geq\frac{\beta}{4}\||\bar{\X}^{h_n}|^{\frac{r-1}{2}}(\bar{\X}^{h}-\bar{\X}^{h})\|_{\H}^2+\frac{\beta}{2}\||\bar{\X}^{h}|^{\frac{r-1}{2}}(\bar{\X}^{h_n}-\bar{\X}^{h})\|_{\H}^2\nonumber\\&\quad-\frac{r-3}{2\mu(r-1)}\left(\frac{4}{\beta\mu (r-1)}\right)^{\frac{2}{r-3}}\left(\int_{\mathcal{O}}|\bar{\X}^{h_n}(x)-\bar{\X}^{h}(x)|^2\d x\right)-\frac{\mu}{2} \|\bar{\X}^{h_n}-\bar{\X}^{h}\|_{\V}^2.
	\end{align}
	Using \eqref{a215},  we obtain 
	\begin{align}
	\frac{2^{2-r}\beta}{4}\|\bar{\X}^{h_n}-\bar{\X}^{h}\|_{\wi\L^{r+1}}^{r+1}\leq\frac{\beta}{4}\||\bar{\X}^{h_n}|^{\frac{r-1}{2}}(\bar{\X}^{h_n}-\bar{\X}^{h})\|_{\L^2}^2+\frac{\beta}{4}\||\bar{\X}^{h}|^{\frac{r-1}{2}}(\bar{\X}^{h_n}-\bar{\X}^{h})\|_{\L^2}^2.
	\end{align}
	Thus, from \eqref{630}, it is immediate that 
	\begin{align}\label{622}
	&\beta\langle\mathcal{C}(\bar{\X}^{h_n})-\mathcal{C}(\bar{\X}^{h}),\bar{\X}^{h_n}-\bar{\X}^{h}\rangle+\langle\B(\bar{\X}^{h_n})-\B(\bar{\X}^{h}),\bar{\X}^{h_n}-\bar{\X}^{h}\rangle\nonumber\\&\geq \frac{\beta}{2^r}\|\bar{\X}^{h_n}-\bar{\X}^{h}\|_{\wi\L^{r+1}}^{r+1}-\frac{\wi\eta}{2}\|\bar{\X}^{h_n}-\bar{\X}^{h}\|_{\H}^2-\frac{\mu}{2} \|\bar{\X}^{h_n}-\bar{\X}^{h}\|_{\V}^2,
	\end{align}
	where \begin{align}\label{623}\wi\eta=\frac{r-3}{\mu(r-1)}\left(\frac{4}{\beta\mu (r-1)}\right)^{\frac{2}{r-3}}.\end{align}
	Using \eqref{4.21} and \eqref{622} in \eqref{5.23z}, we get 
	\begin{align}
	&\|\mathbf{W}_t^{h_n,h}\|_{\H}^2+\mu\int_0^t\|\mathbf{W}_s^{h_n,h}\|_{\V}^2\d s+2\alpha\int_0^t\|\mathbf{W}_s^{h_n,h}\|_{\H}^2\d s+\frac{\beta}{2^{r-1}}\int_0^t\|\mathbf{W}^{h_n,h}_s\|_{\wi\L^{r+1}}^{r+1}\d s\nonumber\\&\leq \wi\eta\int_0^t\|\mathbf{W}^{h_n,h}_s\|_{\H}^2\d s+C_{\mu,\alpha,\lambda_1,L_{\G},L_{\sigma_2}}\int_0^t\|\mathbf{W}^{h_n,h}_s\|_{\H}^2\d s+C\int_0^t\|h_s^n\|_{\H}^2\|\mathbf{W}_s^{h_n,h}\|_{\H}^2\d s\nonumber\\&\quad+\int_0^t\|\sigma_1(\bar{\X}_s^{h})\Q_1^{1/2}(h_s^n-h_s)\|_{\H}^2\d s,
	\end{align}
for all $t\in[0,T]$. Then, proceeding similarly as in the previous case, we arrive at the required result \eqref{5.30z}.

	For the case $n=r=3$, from \eqref{2.23}, we find 
	\begin{align}\label{6.33}
	\beta\langle\mathcal{C}(\bar{\X}^{h_n})-\mathcal{C}(\bar{\X}^{h}),\bar{\X}^{h_n}-\bar{\X}^{h}\rangle\geq\frac{\beta}{2}\|\bar{\X}^{h_n}(\bar{\X}^{h_n}-\bar{\X}^{h})\|_{\H}^2+\frac{\beta}{2}\|\bar{\X}^{h}(\bar{\X}^{h_n}-\bar{\X}^{h})\|_{\H}^2. 
	\end{align}
	We estimate $|\langle\B(\bar{\X}^{h_n})-\B(\bar{\X}^{h}),\bar{\X}^{h_n}-\bar{\X}^{h}\rangle|$  using H\"older's and Young's inequalities as 
	\begin{align}\label{6.34}
|\langle\B(\bar{\X}^{h_n})-\B(\bar{\X}^{h}),\bar{\X}^{h_n}-\bar{\X}^{h}\rangle|&=	|\langle\B(\bar{\X}^{h_n}-\bar{\X}^{h},\bar{\X}^{h_n}-\bar{\X}^{h}),\bar{\X}^{h}\rangle|\nonumber\\&\leq\|\bar{\X}^{h}(\bar{\X}^{h_n}-\bar{\X}^{h})\|_{\H}\|\bar{\X}^{h_n}-\bar{\X}^{h}\|_{\V} \nonumber\\&\leq\frac{\theta\mu}{2} \|\bar{\X}^{h_n}-\bar{\X}^{h}\|_{\V}^2+\frac{1}{2\theta\mu }\|\bar{\X}^{h}(\bar{\X}^{h_n}-\bar{\X}^{h})\|_{\H}^2.
	\end{align}
	Combining \eqref{6.33} and \eqref{6.34}, we obtain 
	\begin{align}\label{632}
	&\mu\langle\A(\bar{\X}^{h_n}-\bar{\X}^{h}),\bar{\X}^{h_n}-\bar{\X}^{h}\rangle+ \beta\langle\mathcal{C}(\bar{\X}^{h_n})-\mathcal{C}(\bar{\X}^{h}),\bar{\X}^{h_n}-\bar{\X}^{h}\rangle\nonumber\\&\quad+\langle\B(\bar{\X}^{h_n})-\B(\bar{\X}^{h}),\bar{\X}^{h_n}-\bar{\X}^{h}\rangle\nonumber\\&\geq\mu\left(1-\frac{\theta}{2}\right) \|\bar{\X}^{h_n}-\bar{\X}^{h}\|_{\V}^2+ \frac{\beta}{2}\|\bar{\X}^{h_n}(\bar{\X}^{h_n}-\bar{\X}^{h})\|_{\H}^2+\frac{1}{2}\left(\beta-\frac{1}{\theta\mu}\right)\|\bar{\X}^{h}(\bar{\X}^{h_n}-\bar{\X}^{h})\|_{\H}^2\nonumber\\&\geq\mu\left(1-\frac{\theta}{2}\right) \|\bar{\X}^{h_n}-\bar{\X}^{h}\|_{\V}^2+ \frac{1}{2}\left(\beta-\frac{1}{\theta\mu}\right)\|\bar{\X}^{h_n}-\bar{\X}^{h}\|_{\wi\L^4}^4,
	\end{align}
for $\frac{1}{\beta\mu}<\theta<2$.	Thus,  we infer that 
	\begin{align}\label{642}
	&\|\mathbf{W}_t^{h_n,h}\|_{\H}^2+\mu(2-\theta)\int_0^t\|\mathbf{W}_s^{h_n,h}\|_{\V}^2\d s+\left(\beta-\frac{1}{\theta\mu}\right)\int_0^t\|\mathbf{W}_s^{h_n,h}\|_{\wi\L^{4}}^{4}\d s\nonumber\\&\leq C_{\mu,\alpha\lambda_1,L_{\G},L_{\sigma_2}}\int_0^t\|\mathbf{W}^{h_n,h}_s\|_{\H}^2\d s+C\int_0^t\|h_s^n\|_{\H}^2\|\mathbf{W}_s^{h_n,h}\|_{\H}^2\d s+\int_0^t\|\sigma_1(\bar{\X}_s^{h})\Q_1^{1/2}(h_s^n-h_s)\|_{\H}^2\d s. 
	\end{align}
	Hence, for $2\beta\mu >1$, arguing similarly as in the case of $n=2,3$ and $r\in[1,3]$, we finally obtain the required result  \eqref{5.30z}. 
\end{proof}

\subsection{Weak convergence} 
Let us now verify the Hypothesis \ref{hyp1} (i) on weak convergence.  We first establish the existence and uniqueness result of the following stochastic controlled SCBF equations.

\begin{theorem}\label{thm5.10}
	For any $h\in\mathcal{A}_M$, $0<M<+\infty$, under the Assumption \ref{ass3.6}, the stochastic control problem:
	\begin{equation}\label{5.4z}
	\left\{
	\begin{aligned}
	\d \X^{\e,\delta,h}_t&=-[\mu\A \X^{\e,\delta,h}_t+\B(\X^{\e,\delta,h}_t)+\alpha\X^{\e,\delta,h}_t+\beta\mathcal{C}(\X^{\e,\delta,h}_t)-\mathrm{F}(\X^{\e,\delta,h}_t,\Y^{\e,\delta,h}_t)]\d t\\&\quad+\sigma_1(\X^{\e,\delta,h}_t)\Q_1^{1/2}h_t\d t+\sqrt{\e}\sigma_1(\X^{\e,\delta,h}_t)\Q_1^{1/2}\d\W_t,\\
	\d \Y^{\e,\delta,h}_t&=-\frac{1}{\delta}[\mu\A \Y^{\e,\delta,h}_t+\alpha\Y^{\e,\delta,h}_t+\beta\mathcal{C}(\Y_{t}^{\e,\delta})-\mathrm{G}(\X^{\e,\delta,h}_t,\Y^{\e,\delta,h}_t)]\d t\\&\quad+\frac{1}{\sqrt{\delta\e}}\sigma_2(\X^{\e,\delta,h}_t,\Y^{\e,\delta,h}_t)\Q_2^{1/2}h_t\d t+\frac{1}{\sqrt{\delta}}\sigma_2(\X^{\e,\delta,h}_t,\Y^{\e,\delta,h}_t)\Q_2^{1/2}\d\W_t,\\
	\X^{\e,\delta,h}_0&=\x,\ \Y^{\e,\delta,h}_0=\y,
	\end{aligned}
	\right. 
	\end{equation}
	has a \emph{pathwise unique  strong solution} $(\X^{\e,\delta,h},\Y^{\e,\delta,h})\in\mathrm{L}^{2}(\Omega;\mathscr{E})\times\mathrm{L}^2(\Omega;\mathscr{E})$, where $\mathscr{E}=\C([0,T];\H)\cap\mathrm{L}^2(0,T;\V)\cap\mathrm{L}^{r+1}(0,T;\wi\L^{r+1})$ with $\mathscr{F}_t$-adapted paths in $\mathscr{E}\times\mathscr{E}$, $\mathbb{P}$-a.s. Furthermore, for all $(\e,\delta)\in(0,1)$ with $\frac{\delta}{\e}\leq\frac{1}{C_{\mu,\alpha,\lambda_1,L_{\G},M}},$ we have 
	\begin{align}\label{5.5z}
&\E\left[\sup_{t\in[0,T]}\|\X^{\e,\delta,h}_t\|_{\H}^{2}+\mu \int_0^T\|\X^{\e,\delta,h}_t\|_{\V}^{2}\d t+\beta \int_0^T\|\X^{\e,\delta,h}_t\|_{\wi\L^{r+1}}^{r+1}\d t\right] \nonumber\\&\leq C_{\mu,\alpha,\lambda_1,L_{\G},M,T}(1+\|\x\|_{\H}^2+\|\y\|_{\H}^2),
	\end{align}
 and 
	\begin{align}\label{443}
	\E\left[\int_0^T\|\Y^{\e,\delta,h}_t\|_{\H}^{2}\d t\right]\leq C_{\mu,\alpha,\lambda_1,L_{\G},M,T}(1+\|\x\|_{\H}^2+\|\y\|_{\H}^2).
	\end{align}
\end{theorem}
\begin{proof}
	The existence and uniqueness of pathwise strong solution satisfying the energy equality to the system (\ref{5.4z}) can be obtained similarly as in Theorem 3.4, \cite{MTM11} (see, Theorem 3.7, \cite{MTM8} also), by using the  monotonicty as well as demicontinuous properties of the linear and nonlinear operators and a stochastic generalization of the Minty-Browder technique.

Let us now prove the uniform energy estimates  \eqref{5.5z} and \eqref{443}.	An application of the infinite dimensional It\^o formula  to the process $\|\Y^{\e,\delta,h}_{t}\|_{\H}^2$  yields (cf. \cite{MTM8})
\begin{align}\label{3.9}
\|\Y^{\e,\delta,h}_t\|_{\H}^2&=\|\y\|_{\H}^2-\frac{2\mu}{\delta}\int_0^t\|\Y^{\e,\delta,h}_s\|_{\V}^2\d s-\frac{2\alpha}{\delta}\int_0^t\|\Y^{\e,\delta,h}_s\|_{\H}^2\d s-\frac{2\beta}{\delta}\int_0^t\|\Y^{\e,\delta,h}_s\|_{\wi\L^{r+1}}^{r+1}\d s\nonumber\\&\quad+\frac{2}{\delta}\int_0^t(\G(\X^{\e,\delta,h}_s,\Y^{\e,\delta,h}_s),\Y^{\e,\delta,h}_s)\d s+\frac{2}{\sqrt{\delta\e}}\int_0^t(\sigma_2(\X^{\e,\delta,h}_s,\Y^{\e,\delta,h}_s)\Q_2^{1/2}h_s,\Y^{\e,\delta,h}_s)\d s\nonumber\\&\quad+\frac{1}{\delta}\int_0^t\|\sigma_2(\X^{\e,\delta,h}_s,\Y^{\e,\delta,h}_s)\Q_2^{1/2}\|_{\mathcal{L}_2}^2\d s+\frac{2}{\sqrt{\delta}}\int_0^t(\sigma_2(\X^{\e,\delta,h}_s,\Y^{\e,\delta,h}_s)\Q_2^{1/2}\d\W_s,\Y^{\e,\delta,h}_s), 
\end{align}
for all $t\in[0,T]$, $\mathbb{P}$-a.s. Taking expectation in \eqref{3.9} and using the fact that the final term appearing in \eqref{3.9} is a martingale, we obtain 
\begin{align}
&\E\left[\|\Y^{\e,\delta,h}_t\|_{\H}^{2}\right]\nonumber\\&=\|\y\|_{\H}^{2}-\frac{2\mu}{\delta}\E\left[\int_0^t	\|\Y^{\e,\delta,h}_s\|_{\V}^2\d s\right]-\frac{2\alpha}{\delta}\E\left[\int_0^t	\|\Y^{\e,\delta,h}_s\|_{\H}^2\d s\right]-\frac{2\beta}{\delta}\E\left[\int_0^t	\|\Y^{\e,\delta,h}_s\|_{\wi\L^{r+1}}^{r+1}\d s\right]\nonumber\\&\quad+\frac{2}{\delta}\E\left[\int_0^t(\G(\X^{\e,\delta,h}_s,\Y^{\e,\delta,h}_s),\Y^{\e,\delta,h}_s)\d s\right]+\frac{2}{\sqrt{\delta\e}}\E\left[\int_0^t(\sigma_2(\X^{\e,\delta,h}_s,\Y^{\e,\delta,h}_s)\Q_2^{1/2}h_s,\Y^{\e,\delta,h}_s)\d s\right]\nonumber\\&\quad+\frac{1}{\delta}\E\left[\int_0^t\|\sigma_2(\X^{\e,\delta,h}_s,\Y^{\e,\delta,h}_s)\Q_2^{1/2}\|_{\mathcal{L}_2}^2\d s\right], 
\end{align}
for all $t\in[0,T]$. Thus, it is immediate that 
\begin{align}\label{312}
\frac{\d}{\d t}	\E\left[\|\Y^{\e,\delta,h}_t\|_{\H}^{2}\right]&=-\frac{2\mu}{\delta}\E\left[	\|\Y^{\e,\delta,h}_t\|_{\V}^2\right]-\frac{2\alpha}{\delta}\E\left[	\|\Y^{\e,\delta,h}_t\|_{\H}^2\right]-\frac{2\beta}{\delta}\E\left[	\|\Y^{\e,\delta,h}_t\|_{\wi\L^{r+1}}^{r+1}\right]\nonumber\\&\quad+\frac{2}{\delta}\E\left[(\G(\X^{\e,\delta,h}_t,\Y^{\e,\delta,h}_t),\Y^{\e,\delta,h}_t)\right]+\frac{2}{\sqrt{\delta\e}}\E\left[(\sigma_2(\X^{\e,\delta,h}_t,\Y^{\e,\delta,h}_t)\Q_2^{1/2}h_t,\Y^{\e,\delta,h}_t)\right]\nonumber\\&\quad+\frac{1}{\delta}\E\left[\|\sigma_2(\X^{\e,\delta,h}_t,\Y^{\e,\delta,h}_t)\Q_2^{1/2}\|_{\mathcal{L}_2}^2\right], 
\end{align}
for a.e. $t\in[0,T]$. 	Using the Assumption \ref{ass3.6} (A1), the Cauchy-Schwarz inequality and Young's inequality, we get 
\begin{align}\label{313}
\frac{2}{\delta}(\G(\X^{\e,\delta,h},\Y^{\e,\delta,h}),\Y^{\e,\delta,h})&\leq \frac{2}{\delta}\|\G(\X^{\e,\delta,h},\Y^{\e,\delta,h})\|_{\H}\|\Y^{\e,\delta,h}\|_{\H}\nonumber\\&\leq \frac{2}{\delta} \left(\|\G(\mathbf{0},\mathbf{0})\|_{\H}+C\|\X^{\e,\delta,h}\|_{\H}+L_{\G}\|\Y^{\e,\delta,h}\|_{\H}\right)\|\Y^{\e,\delta,h}\|_{\H}\nonumber\\&\leq\left(\frac{\mu\lambda_1}{2\delta}+\frac{2L_{\G}}{\delta}\right)\|\Y^{\e,\delta,h}\|_{\H}^{2}+\frac{C}{\mu\lambda_1\delta}\|\X^{\e,\delta,h}\|_{\H}^{2}+\frac{C}{\mu\lambda_1\delta}.
\end{align}
Using the Cauchy-Schwarz inequality,  Assumption \ref{ass3.6} (A2) and Young's inequality, we estimate the term  $\frac{2}{\sqrt{\delta\e}}(\sigma_2(\X^{\e,\delta,h},\Y^{\e,\delta,h})h,\Y^{\e,\delta,h})$ as 
\begin{align}
\frac{2}{\sqrt{\delta\e}}(\sigma_2(\X^{\e,\delta,h},\Y^{\e,\delta,h})\Q_2^{1/2}h,\Y^{\e,\delta,h})&\leq \frac{2}{\sqrt{\delta\e}}\|\sigma_2(\X^{\e,\delta,h},\Y^{\e,\delta,h})\Q_2^{1/2}\|_{\mathcal{L}_2}\|h\|_{\H}\|\Y^{\e,\delta,h}\|_{\H}\nonumber\\&\leq \frac{C}{\sqrt{\delta\e}}(1+\|\X^{\e,\delta,h}\|_{\H})\|h\|_{\H}\|\Y^{\e,\delta,h}\|_{\H}\nonumber\\&\leq\frac{\mu\lambda_1}{4\delta}\|\Y^{\e,\delta,h}\|_{\H}^{2} +\frac{C}{\mu\lambda_1\e}\left(1+\|\X^{\e,\delta,h}\|_{\H}^{2}\right)\|h\|_{\H}^2.
\end{align}
Once again, using the Assumption \ref{ass3.6} (A2) and Young's inequality, we obtain 
\begin{align}\label{315}
\frac{1}{\delta}\|\sigma_2(\X^{\e,\delta,h},\Y^{\e,\delta,h})\Q_2^{1/2}\|_{\mathcal{L}_2}^2&\leq 	\frac{C}{\delta}(1+\|\X^{\e,\delta,h}\|_{\H}^2)\leq \frac{\mu\lambda_1}{4\delta}\|\Y^{\e,\delta,h}\|_{\H}^{2}+\frac{C}{\mu\lambda_1\delta}\|\X^{\e,\delta,h}\|_{\H}^{2p}+\frac{C}{\mu\lambda_1\delta}.
\end{align}
Combining \eqref{313}-\eqref{315} and using it in \eqref{312}, we find 
\begin{align}\label{3.16}
\frac{\d}{\d t}	\E\left[\|\Y^{\e,\delta,h}_t\|_{\H}^{2}\right]&=-\frac{1}{\delta}\left(\mu{\lambda_1}+2\alpha-2L_{\G}\right)\E\left[	\|\Y^{\e,\delta,h}_t\|_{\H}^{2}\right]+\frac{C}{\mu\lambda_1\delta}\E\left[\|\X^{\e,\delta,h}_t\|_{\H}^{2}\right]+\frac{C}{\mu\lambda_1\delta}\nonumber\\&\quad+\frac{C}{\mu\lambda_1\e}\E\left[\left(1+\|\X^{\e,\delta,h}_t\|_{\H}^{2}\right)\|h_t\|_{\H}^2\right],
\end{align}
for a.e. $t\in[0,T]$. By the Assumption \ref{ass3.6} (A3), we know that $\mu\lambda_1+2\alpha>2L_{\G}$ and an application of variation of constants formula gives 
\begin{align}\label{3.17}
\E\left[\|\Y^{\e,\delta,h}_t\|_{\H}^{2}\right]&\leq\|\y\|_{\H}^{2}e^{-\frac{\gamma t}{\delta}}+\frac{C}{\mu\lambda_1\delta}\int_0^te^{-\frac{\gamma(t-s)}{\delta}}\left(1+\E\left[\|\X^{\e,\delta,h}_s\|_{\H}^{2}\right]\right)\d s\nonumber\\&\quad+\frac{C}{\mu\lambda_1\e}\int_0^te^{-\frac{\gamma(t-s)}{\delta}}\E\left[\left(1+\|\X^{\e,\delta,h}_s\|_{\H}^{2}\right)\|h_s\|_{\H}^2\right]\d s,
\end{align}
for all $t\in[0,T]$, where $\gamma=\left(\mu{\lambda_1}+2\alpha-2L_{\G}\right)$. Integrating the above inequality from $0$ to $T$, performing a change of order of integration and using Fubini's theorem, we also get 
\begin{align}\label{452}
\E\left[\int_0^T\|\Y^{\e,\delta,h}_t\|_{\H}^{2}\d t\right]&\leq\|\y\|_{\H}^{2}\int_0^Te^{-\frac{\gamma t}{\delta}}\d t+\frac{C}{\mu\lambda_1\delta}\int_0^T\int_0^te^{-\frac{\gamma(t-s)}{\delta}}\left(1+\E\left[\|\X^{\e,\delta,h}_s\|_{\H}^{2}\right]\right)\d s\d t\nonumber\\&\quad+\frac{C}{\mu\lambda_1\e}\E\left[\sup_{t\in[0,T]}\left(1+\|\X^{\e,\delta,h}_t\|_{\H}^{2}\right)\int_0^T\int_0^te^{-\frac{\gamma(t-s)}{\delta}}\|h_s\|_{\H}^2\d s\d t\right]\nonumber\\&\leq\frac{\delta}{\gamma}\|\y\|_{\H}^{2}+\frac{C}{\mu\lambda_1\gamma}\E\left[\int_0^T\left(1+\|\X^{\e,\delta,h}_t\|_{\H}^{2}\right)\d t\right]\nonumber\\&\quad+\frac{C\delta}{\mu\lambda_1\gamma\e}\E\left[\sup_{t\in[0,T]}\left(1+\|\X^{\e,\delta,h}_t\|_{\H}^{2}\right)\int_0^T\|h_t\|_{\H}^2\d t\right]\nonumber\\&\leq C_{\mu,\alpha,\lambda_1,L_{\G},T}(1+\|\y\|_{\H}^2)+C_{\mu,\alpha,\lambda_1,L_{\G},T}\E\left[\int_0^T\|\X^{\e,\delta,h}_t\|_{\H}^{2}\d t\right]\nonumber\\&\quad+C_{\mu,\alpha,\lambda_1,L_{\G},M}\left(\frac{\delta}{\e}\right)\left[T+\E\left(\sup_{t\in[0,T]}\|\X^{\e,\delta,h}_t\|_{\H}^{2}\right)\right],
\end{align}
since $\delta\in(0,1)$.

Let us now obtain the energy  estimates satisfied by the process $\X^{\e,\delta,h}_{t}$. 	Applying the infinite dimensional It\^o formula to the process $\|\X^{\e,\delta,h}_{t}\|_{\H}^2$ (see \cite{MTM8}), we find 
\begin{align}\label{3.19}
\|\X^{\e,\delta,h}_t\|_{\H}^2&=\|\x\|_{\H}^2-2\mu\int_0^t\|\X^{\e,\delta,h}_s\|_{\V}^2\d s-2\alpha\int_0^t\|\X^{\e,\delta,h}_s\|_{\H}^2\d s-2\beta\int_0^t\|\X^{\e,\delta,h}_s\|_{\wi\L^{r+1}}^{r+1}\d s\nonumber\\&\quad+2\int_0^t(\F(\X^{\e,\delta,h}_s,\Y^{\e,\delta,h}_s),\X^{\e,\delta,h}_s)\d s+2\int_0^t(\sigma_1(\X^{\e,\delta,h}_s)\Q_1^{1/2}h_s,\X^{\e,\delta,h}_s)\d s\nonumber\\&\quad+\e\int_0^t\|\sigma_1(\X^{\e,\delta,h}_s)\Q_1^{1/2}\|_{\mathcal{L}_2}^2\d s+2\sqrt{\e}\int_0^t(\sigma_1(\X^{\e,\delta,h}_s)\Q_1^{1/2}\d\W_s,\X^{\e,\delta,h}_s),	
\end{align}
for all $t\in[0,T]$, $\mathbb{P}$-a.s.
Taking supremum over $[0,T]$ and then taking expectation in \eqref{3.19}, we obtain 
\begin{align}\label{320}
&\E\left[\sup_{t\in[0,T]}\|\X^{\e,\delta,h}_t\|_{\H}^{2}+2\mu \int_0^T\|\X^{\e,\delta,h}_t\|_{\V}^{2}\d t+2\alpha \int_0^T\|\X^{\e,\delta,h}_t\|_{\H}^{2}\d t+2\beta \int_0^T\|\X^{\e,\delta,h}_t\|_{\wi\L^{r+1}}^{r+1}\d t\right]\nonumber\\&\leq \|\x\|_{\H}^{2}+2\E\left[\int_0^T|(\F(\X^{\e,\delta,h}_t,\Y^{\e,\delta,h}_t),\X^{\e,\delta,h}_t)|\d t\right]+2\E\left[\int_0^T|(\sigma_1(\X^{\e,\delta,h}_t)\Q_1^{1/2}h_t,\X^{\e,\delta,h}_t)|\d t\right]\nonumber\\&\quad  +\e\E\left[\int_0^T\|\sigma_1(\X^{\e,\delta,h}_t)\Q_1^{1/2}\|_{\mathcal{L}_2}^2\d t\right]+2\sqrt{\e}\E\left[\sup_{t\in[0,T]}\left|\int_0^t(\sigma_1(\X^{\e,\delta,h}_s)\Q_1^{1/2}\d\W_s,\X^{\e,\delta,h}_s)\right|\right]\nonumber\\&\leq \|\x\|_{\H}^{2}+CT+C\E\left[\int_0^T\|\X^{\e,\delta,h}_t\|_{\H}^{2}\d t\right]+C\E\left[\int_0^T\|\Y^{\e,\delta,h}_t\|_{\H}^{2}\d t\right]\nonumber\\&\quad+2\E\left[\int_0^T|(\sigma_1(\X^{\e,\delta,h}_t)\Q_1^{1/2}h_t,\X^{\e,\delta,h}_t)|\d t\right]+2\sqrt{\e}\E\left[\sup_{t\in[0,T]}\left|\int_0^t(\sigma_1(\X^{\e,\delta,h}_s)\Q_1^{1/2}\d\W_s,\X^{\e,\delta,h}_s)\right|\right], 
\end{align}
where we used  calculations similar to \eqref{313}-\eqref{315}. We estimate the penultimate term from the right hand side of the inequality \eqref{320} using the Cauchy-Schwarz inequality, H\"older's and Young's inequalities as 
\begin{align}\label{5.9z}
&2\E\left[\int_0^T|(\sigma_1(\X^{\e,\delta,h}_t)\Q_1^{1/2}h_t,\X^{\e,\delta,h}_t)|\d t\right]\nonumber\\&\leq 2\E\left[\int_0^{T}\|\sigma_1(\X^{\e,\delta,h}_t)\Q_1^{1/2}\|_{\mathcal{L}_2}\|h_t\|_{\H}\|\X^{\e,\delta,h}_t\|_{\H}\d t\right]\nonumber\\&\leq 2\E\left[\sup_{t\in[0,T]}\|\X^{\e,\delta,h}_t\|_{\H}\int_0^{T}\|\sigma_1(\X^{\e,\delta,h}_t)\Q_1^{1/2}\|_{\mathcal{L}_2}\|h_t\|_{\H}\d t\right]\nonumber\\&\leq  \frac{1}{4}\E\left[\sup_{t\in[0,T]}\|\X^{\e,\delta,h}_t\|_{\H}^2\right]+4\E\left(\int_0^{T}\|\sigma_1(\X^{\e,\delta,h}_t)\Q_1^{1/2}\|_{\mathcal{L}_2}\|h_t\|_{\H}\d t\right)^2\nonumber\\&\leq\frac{1}{4}\E\left[\sup_{t\in[0,T]}\|\X^{\e,\delta,h}_t\|_{\H}^2\right]+4 \E\left[\left(\int_0^{T}\|\sigma_1(\X^{\e,\delta,h}_t)\Q_1^{1/2}\|_{\mathcal{L}_2}^2\d t\right)\left(\int_0^{T}\|h_t\|_{\H}^2\d t\right)\right]\nonumber\\&\leq \frac{1}{4}\E\left[\sup_{t\in[0,T]}\|\X^{\e,\delta,h}_t\|_{\H}^2\right]+4M \E\left[\int_0^{T}\|\sigma_1(\X^{\e,\delta,h}_t)\Q_1^{1/2}\|_{\mathcal{L}_2}^2\d t\right]\nonumber\\&\leq \frac{1}{4}\E\left[\sup_{t\in[0,T]}\|\X^{\e,\delta,h}_t\|_{\H}^2\right]+CMT+CM\E\left[\int_0^T\|\X^{\e,\delta,h}_t\|_{\H}^2\d t\right],
\end{align}
where we used the fact that $h\in\mathcal{A}_M$. Using the Burkholder-Davis-Gundy inequality (see Theorem 1, \cite{BD} for the Burkholder-Davis-Gundy inequality for the case $p=1$ and Theorem 1.1, \cite{DLB} for the best constant, \cite{CMMR} for BDG inequality in infinite dimensions), we estimate the final term from the right hand side of the inequality \eqref{320} as 
\begin{align}\label{3.21}
&2\sqrt{\e}\E\left[\sup_{t\in[0,T]}\left|\int_0^t(\sigma_1(\X^{\e,\delta,h}_s)\Q_1^{1/2}\d\W_s,\X^{\e,\delta,h}_s)\right|\right]\nonumber\\&\leq C\sqrt{\e}\E\left[\int_0^T\|\sigma_1(\X^{\e,\delta,h}_t)\Q_1^{1/2}\|_{\mathcal{L}_2}^2\|\X^{\e,\delta,h}_t\|_{\H}^2\d t\right]^{1/2}\nonumber\\&\leq C\sqrt{\e}\E\left[\sup_{t\in[0,T]}\|\X^{\e,\delta,h}_t\|_{\H}^{2}\left(\int_0^T\|\sigma_1(\X^{\e,\delta,h}_t)\Q_1^{1/2}\|_{\mathcal{L}_2}^2\d t\right)^{1/2}\right]\nonumber\\&\leq\frac{1}{4}\E\left[\sup_{t\in[0,T]}\|\X^{\e,\delta,h}_t\|_{\H}^{2}\right]+C\e\E\left[\int_0^T\|\sigma_1(\X^{\e,\delta,h}_t)\Q_1^{1/2}\|_{\mathcal{L}_2}^2\d t\right]\nonumber\\&\leq\frac{1}{4}\E\left[\sup_{t\in[0,T]}\|\X^{\e,\delta,h}_t\|_{\H}^{2}\right]+C\e\E\left[\int_0^T\|\X^{\e,\delta,h}_t\|_{\H}^{2p}\d t\right]+C\e T,
\end{align}
where we used the Assumption \ref{ass3.6} (A1). Using \eqref{3.17}, \eqref{5.9z} and \eqref{3.21} in \eqref{320}, we deduce that 
\begin{align}\label{322}
&\E\left[\sup_{t\in[0,T]}\|\X^{\e,\delta,h}_t\|_{\H}^{2}+4\mu \int_0^T\|\X^{\e,\delta,h}_t\|_{\V}^{2}\d t+4\alpha \int_0^T\|\X^{\e,\delta,h}_t\|_{\V}^{2}\d t+4\beta \int_0^T\|\X^{\e,\delta,h}_t\|_{\wi\L^{r+1}}^{r+1}\d t\right] \nonumber\\&\leq  2\|\x\|_{\H}^{2}+C(1+M+\e)T+C(1+M+\e)\E\left[\int_0^T\|\X^{\e,\delta,h}_t\|_{\H}^{2}\d t\right]+C\E\left[\int_0^T\|\Y^{\e,\delta,h}_t\|_{\H}^{2}\d t\right]\nonumber\\&\leq  2\|\x\|_{\H}^{2}+C_{\mu,\alpha,\lambda_1,L_{\G},T}(1+\|\y\|_{\H}^2)+C_MT+C_M\E\left[\int_0^T\|\X^{\e,\delta,h}_t\|_{\H}^{2}\d t\right]\nonumber\\&\quad+C_{\mu,\alpha,\lambda_1,L_{\G},T}\E\left[\int_0^T\|\X^{\e,\delta,h}_t\|_{\H}^{2}\d t\right]+C_{\mu,\alpha,\lambda_1,L_{\G},M}\left(\frac{\delta}{\e}\right)\left[T+\E\left(\sup_{t\in[0,T]}\|\X^{\e,\delta,h}_t\|_{\H}^{2}\right)\right],
\end{align}
since $\e,\delta\in(0,1)$. By the Assumption \ref{ass3.6} (A4), one can choose $\frac{\delta}{\e}<\frac{1}{C_{\mu,\alpha,\lambda_1,L_{\G},M}}$, so that
\begin{align}\label{3.22}
&\E\left[\sup_{t\in[0,T]}\|\X^{\e,\delta,h}_t\|_{\H}^{2}\right]\leq C_{\mu,\alpha,\lambda_1,L_{\G},M,T}\left\{1+\|\x\|_{\H}^2+\|\y\|_{\H}^2+\E\left[\int_0^T\|\X^{\e,\delta,h}_t\|_{\H}^{2}\d t\right]\right\}.
\end{align}
 An application of Gronwall's inequality in \eqref{3.22} implies 
\begin{align}\label{3.23}
\E\left[\sup_{t\in[0,T]}\|\X^{\e,\delta,h}_t\|_{\H}^{2}\right]\leq C_{\mu,\alpha,\lambda_1,L_{\G},M,T}\left(1+\|\x\|_{\H}^2+\|\y\|_{\H}^2\right). 
\end{align}
Substitution of \eqref{3.23} in \eqref{322} yields the estimate \eqref{5.5z}. Using \eqref{3.23} in \eqref{452}, we finally obtain \eqref{443}. 
\end{proof}

Let us now take $(\X^{\e,\delta,h^{\e}},\Y^{\e,\delta,h^{\e}})$ solves the following stochastic control system:
\begin{equation}\label{442}
\left\{
\begin{aligned}
\d \X^{\e,\delta,h^{\e}}_t&=-[\mu\A \X^{\e,\delta,h^{\e}}_t+\B(\X^{\e,\delta,h^{\e}}_t)+\alpha\X^{\e,\delta,h^{\e}}_t+\beta\mathcal{C}(\X^{\e,\delta,h^{\e}}_t)-\mathrm{F}(\X^{\e,\delta,h^{\e}}_t,\Y^{\e,\delta,h^{\e}}_t)]\d t\\&\quad+\sigma_1(\X^{\e,\delta,h^{\e}}_t)\Q_1^{1/2}h_t^{\e}\d t+\sqrt{\e}\sigma_1(\X^{\e,\delta,h^{\e}}_t)\Q_1^{1/2}\d\W_t,\\
\d \Y^{\e,\delta,h^{\e}}_t&=-\frac{1}{\delta}[\mu\A \Y^{\e,\delta,h^{\e}}_t+\alpha\Y^{\e,\delta,h^{\e}}_t+\beta\mathcal{C}(\Y_{t}^{\e,\delta})-\mathrm{G}(\X^{\e,\delta,h^{\e}}_t,\Y^{\e,\delta,h^{\e}}_t)]\d t\\&\quad+\frac{1}{\sqrt{\delta\e}}\sigma_2(\X^{\e,\delta,h^{\e}}_t,\Y^{\e,\delta,h^{\e}}_t)\Q_2^{1/2}h_t^{\e}\d t+\frac{1}{\sqrt{\delta}}\sigma_2(\X^{\e,\delta,h^{\e}}_t,\Y^{\e,\delta,h^{\e}}_t)\Q_2^{1/2}\d\W_t,\\
\X^{\e,\delta,h^{\e}}_0&=\x,\ \Y^{\e,\delta,h^{\e}}_0=\y.
\end{aligned}
\right. 
\end{equation}
Using Theorem \ref{thm5.10}, we know that the system \eqref{442} has a pathwise unique strong solution $(\X^{\e,\delta,h^{\e}},\Y^{\e,\delta,h^{\e}})$ with paths in $\mathscr{E}\times\mathscr{E},$ $\mathbb{P}\text{-a.s.}$, where  $\mathscr{E}=\C([0,T];\H)\cap\mathrm{L}^2(0,T;\V)\cap\mathrm{L}^{r+1}(0,T;\wi\L^{r+1}).$ Since, 
\begin{align*}
\mathbb{E}\left(\exp\left\{-\frac{1}{\sqrt{\e}}\int_0^T(h^{\e}_t,\d\W_t)-\frac{1}{2\e}\int_0^T\|h^{\e}_t\|_{\H}^2\d t \right\}\right)=1,
\end{align*}
the measure $\widehat{\mathbb{P}}$ defined by 
\begin{align*}
\d\widehat{\mathbb{P}}(\omega)=\exp\left\{-\frac{1}{\sqrt{\e}}\int_0^T(h^{\e}_t,\d\W_t)-\frac{1}{2\e}\int_0^T\|h^{\e}_t\|_{\H}^2\d t\right\}\d{\mathbb{P}}(\omega)
\end{align*}
is a probability measure on $(\Omega,\mathscr{F},\mathbb{P})$. Moreover, $\widehat{\mathbb{P}}(\omega)$ is mutually absolutely continuous with respect to $\mathbb{P}(\omega)$ and by using Girsanov's Theorem (Theorem 10.14, \cite{DZ}, Appendix, \cite{GDFF}),  we have the process 
\begin{align*}
\widehat{\W}_t:=\W_t+\frac{1}{\sqrt{\e}}\int_0^th^{\e}_s\d s, \ t\in[0,T],
\end{align*}
is a cylindrical Wiener process with respect to $\{\mathscr{F}_t\}_{t\geq 0}$ on the probability space $(\Omega,\mathscr{F},\widehat{\mathbb{P}})$. Thus, we know that (\cite{BD,MRBS}) $$(\X^{\e,\delta,h^{\e}}_{\cdot},\Y^{\e,\delta,h^{\e}}_{\cdot})=\left(\mathcal{G}^{\e}\left(\W_{\cdot} +\frac{1}{\sqrt{\e}}\int_0^{\cdot}h^{\e}_s\d s\right),\mathcal{G}^{\delta}\left(\W_{\cdot} +\frac{1}{\sqrt{\e}}\int_0^{\cdot}h^{\e}_s\d s\right)\right)$$ is the unique strong solution of \eqref{2.18a} with $\W_{\cdot}$ replaced by $\widehat{\W}_{\cdot}$, on $(\Omega,\mathscr{F},\{\mathscr{F}_t\}_{t\geq 0},\widehat{\mathbb{P}})$. Moreover, the system \eqref{2.18a} with $\widehat{\W}_{\cdot}$ is same as the system \eqref{442}, and since $\widehat{\mathbb{P}}$ and $\mathbb{P}$ are mutually absolutely continuous, we further find that $(\X^{\e,\delta,h^{\e}}_{\cdot},\Y^{\e,\delta,h^{\e}}_{\cdot})$ is the unique strong solution of \eqref{442} on $(\Omega,\mathscr{F},\{\mathscr{F}_t\}_{t\geq 0},{\mathbb{P}})$.  Thus, the solution of the first equation in \eqref{442} is represented as 
$$\X^{\e,\delta,h^{\e}}_{\cdot}=\mathcal{G}^{\e}\left(\W_{\cdot} +\frac{1}{\sqrt{\e}}\int_0^{\cdot}h^{\e}_s\d s\right).$$

Since our approach is based on the Khasminskii time discretization, we need the following lemma.  Similar result for the stochastic 2D Navier-Stokes equation is obtained in \cite{SLXS}, for the stochastic Burgers equation is established in \cite{XSRW} and for SCBF equations is proved in \cite{MTM11}. Since the system \eqref{442} is a controlled SCBF  equations, we provide a proof here. 

Let us first define a sequence of stopping times $\{\tau_R^{\e}\}$ as 
\begin{align}\label{stop}
\tau_R^{\e}:=\inf_{t\geq 0}\left\{t:\|\X_t^{\delta,\e,h^{\e}}\|_{\H}> R\right\}, 
\end{align}
for any $\e,R>0$.

\begin{lemma}\label{lem3.8}
	For any $\x,\y\in\H$, $T>0$, $\e,\Delta>0$  small enough, there exists a constant $C_{\mu,\alpha,\beta,\lambda_1,L_{\G},M,T,R}>0$ such that
	\begin{align}\label{3.24}
	\E\left[\int_0^{\T}\|\X_{t}^{\e,\delta,h^{\e}}-\X^{\e,\delta,h^{\e}}_{t(\Delta)}\|_{\H}^2\d t \right]\leq C_{\mu,\alpha,\beta,\lambda_1,L_{\G},M,T,R}\Delta^{1/2}(1+\|\x\|_{\H}^2+\|\y\|_{\H}^2). 
	\end{align}
 where $t(\Delta):=\left[\frac{t}{\Delta}\right]\Delta$ and $[s]$ stands for the largest integer which is less than or equal $s$. 
\end{lemma}

\begin{proof}
	A straightforward calculation gives 
	\begin{align}\label{3.25}
	&	\E\left[\int_0^{\T}\|\X_{t}^{\e,\delta,h^{\e}}-\X^{\e,\delta,h^{\e}}_{t(\Delta)}\|_{\H}^2\d t\right] \nonumber\\&\leq 	\E\left[\int_0^{\Delta}\|\X_{t}^{\e,\delta,h^{\e}}-\x\|_{\H}^2\chi_{\{t\leq\tau^{\e}_R\}}\d t\right]+	\E\left[\int_{\Delta}^T\|\X_{t}^{\e,\delta,h^{\e}}-\X^{\e,\delta,h^{\e}}_{t(\Delta)}\|_{\H}^2\chi_{\{t\leq\tau^{\e}_R\}}\d t\right]\nonumber\\&\leq C_{R}\left(1+\|\x\|_{\H}^{2}\right)\Delta+	2\E\left[\int_{\Delta}^T\|\X_{t}^{\e,\delta,h^{\e}}-\X^{\e,\delta,h^{\e}}_{t-\Delta}\|_{\H}^2\chi_{\{t\leq\tau^{\e}_R\}}\d t\right]\nonumber\\&\quad+2	\E\left[\int_{\Delta}^T\|\X^{\e,\delta,h^{\e}}_{t(\Delta)}-\X^{\e,\delta,h^{\e}}_{t-\Delta}\|_{\H}^2\chi_{\{t\leq\tau^{\e}_R\}}\d t\right].
	\end{align}
	Let us first estimate the second term from the right hand side of the inequality \eqref{3.25}. Using the infinite dimensional It\^o formula applied to the process $\mathrm{Z}_r=\|\X_{r}^{\e,\delta,h^{\e}}-\X^{\e,\delta,h^{\e}}_{t-\Delta}\|_{\H}^2$ over the interval $[t-\Delta,t]$, we find 
	\begin{align}
&\|\X_{t}^{\e,\delta,h^{\e}}-\X^{\e,\delta,h^{\e}}_{t-\Delta}\|_{\H}^2\nonumber\\&=-2\mu\int_{t-\Delta}^t\langle\A\X_{s}^{\e,\delta,h^{\e}},\X_{s}^{\e,\delta,h^{\e}}-\X^{\e,\delta,h^{\e}}_{t-\Delta}\rangle\d s-2\alpha\int_{t-\Delta}^t(\X_{s}^{\e,\delta,h^{\e}},\X_{s}^{\e,\delta,h^{\e}}-\X^{\e,\delta,h^{\e}}_{t-\Delta})\d s\nonumber\\&\quad-2\int_{t-\Delta}^t\langle\B(\X_{s}^{\e,\delta,h^{\e}}),\X_{s}^{\e,\delta,h^{\e}}-\X^{\e,\delta,h^{\e}}_{t-\Delta}\rangle\d s-2\beta\int_{t-\Delta}^t\langle\mathcal{C}(\X_{s}^{\e,\delta,h^{\e}}),\X_{s}^{\e,\delta,h^{\e}}-\X_{t-\Delta}^{\e,\delta,h^{\e}}\rangle\d s\nonumber\\&\quad+2\int_{t-\Delta}^t(\F(\X_{s}^{\e,\delta,h^{\e}},\Y^{\e,\delta,h^{\e}}_s),\X_{s}^{\e,\delta,h^{\e}}-\X_{t-\Delta}^{\e,\delta,h^{\e}})\d s+2\int_{t-\Delta}^t(\sigma_1(\X^{\e,\delta,h^{\e}}_s)\Q_1^{1/2}h_s^{\e},\X_{s}^{\e,\delta,h^{\e}}
-\X_{t-\Delta}^{\e,\delta,h^{\e}})\d s\nonumber\\&\quad+\e\int_{t-\Delta}^t\|\sigma_1(\X_{s}^{\e,\delta,h^{\e}})\Q_1^{1/2}\|_{\mathcal{L}_2}^2\d s+2\sqrt{\e}\int_{t-\Delta}^t(\sigma_1(\X_{s}^{\e,\delta,h^{\e}})\Q_1^{1/2}\d\W_s,\X_{s}^{\e,\delta,h^{\e}}-\X_{t-\Delta}^{\e,\delta,h^{\e}})\d s\nonumber\\&=:\sum_{k=1}^8I_k(t).
	\end{align} 
	Using an integration by parts, H\"older's inequality,  Fubini's Theorem and \eqref{5.5z}, we estimate $\E\left(\int_{\Delta}^T|I_1(t)|\chi_{\{t\leq\tau^{\e}_R\}}\d t\right)$ as (see \eqref{4p27} also)
	\begin{align}\label{327}
&	\E\left(	\int_{\Delta}^T|I_1(t)|\chi_{\{t\leq\tau^{\e}_R\}}\d t\right)\nonumber\\&\leq 2\mu\left[\E\left(\int_{\Delta}^T\int_{t-\Delta}^t\|\X_{s}^{\e,\delta,h^{\e}}\|_{\V}^2\chi_{\{t\leq\tau^{\e}_R\}}\d s\d t\right)\right]^{1/2} \left[\E\left(\int_{\Delta}^T\int_{t-\Delta}^t\|\X_{s}^{\e,\delta,h^{\e}}-\X_{t-\Delta}^{\e,\delta,h^{\e}}\|_{\V}^2\chi_{\{t\leq\tau^{\e}_R\}}\d s\d t\right)\right]^{1/2}\nonumber\\&\leq 2\mu\left[\Delta\E\left(\int_{0}^T\|\X_{t}^{\e,\delta,h^{\e}}\|_{\V}^2\d t\right)\right]^{1/2}\left[2\Delta\E\left(\int_{0}^T\|\X_{t}^{\e,\delta,h^{\e}}\|_{\V}^2\d t\right)\right]^{1/2}\nonumber\\&\leq C_{\mu,\alpha,\lambda_1,L_{\G},T}\Delta\left(1+\|\x\|_{\H}^2+\|\y\|_{\H}^2\right). 
	\end{align}
		Similarly, we estimate $\E\left(\int_{\Delta}^T|I_2(t)|\chi_{\{t\leq\tau^{\e}_R\}}\d t\right)$ as 
	\begin{align}
	\E\left(\int_{\Delta}^T|I_2(t)|\chi_{\{t\leq\tau^{\e}_R\}}\d t\right)&\leq C\alpha\Delta T\E\left(\sup_{t\in[0,T]}\|\X_{t}^{\e,\delta,h^{\e}}\|_{\H}^2\right)\leq C_{\mu,\alpha,\uplambda_1,L_{\mathrm{G}},T}\Delta\left(1+\|\x\|_{\H}^2+\|\y\|_{\H}^2\right).
	\end{align}
	For $n=2$ and $r\in[1,3)$,	using  H\"older's and Ladyzhenskaya's inequalities,  Fubini's Theorem and \eqref{5.5z}, we estimate $\E\left(\int_{\Delta}^T|I_3(t)|\chi_{\{t\leq\tau^{\e}_R\}}\d t\right)$ as (see \eqref{4p29} also)
	\begin{align}
	&	\E\left(\int_{\Delta}^T|I_3(t)|\chi_{\{t\leq\tau^{\e}_R\}}\d t\right)\nonumber\\&\leq 2\sqrt{2}\left[\E\left(\int_{\Delta}^T\int_{t-\Delta}^t\|\X_{s}^{\e,\delta,h^{\e}}\|_{\H}^2\|\X_{s}^{\e,\delta,h^{\e}}\|_{\V}^2\chi_{\{t\leq\tau^{\e}_R\}}\d s\d t\right)\right]^{1/2}\nonumber\\&\quad\times\left[\E\left(\int_{\Delta}^T\int_{t-\Delta}^t\|\X_{s}^{\e,\delta,h^{\e}}-\X_{t-\Delta}^{\e,\delta,h^{\e}}\|_{\V}^2\chi_{\{t\leq\tau^{\e}_R\}}\d s\d t\right)\right]^{1/2}\nonumber\\&\leq 2\sqrt{2}\left[\Delta\E\left(\int_0^T\|\X_{t}^{\e,\delta,h^{\e}}\|_{\H}^2\|\X_{t}^{\e,\delta,h^{\e}}\|_{\V}^2\chi_{\{t\leq\tau^{\e}_R\}}\d t\right)\right]^{1/2} \left[2\Delta\E\left(\int_{0}^T\|\X_{t}^{\e,\delta,h^{\e}}\|_{\V}^2\d t\right)\right]^{1/2}\nonumber\\&\leq  C_{\mu,\alpha,\lambda_1,L_{\G},T,R}\Delta\left(1+\|\x\|_{\H}^2+\|\y\|_{\H}^2\right). 
	\end{align}
	For $n=2,3$ and $r\geq 3$ (take $2\beta\mu>1$, for $r=3$), we estimate $\E\left(\int_{\Delta}^T|I_3(t)|\chi_{\{t\leq\tau^{\e}_R\}}\d t\right)$  using H\"older's inequality, interpolation inequality and \eqref{5.5z} as (see \eqref{4p30} also)
	\begin{align}
	&	\E\left(\int_{\Delta}^T|I_3(t)|\chi_{\{t\leq\tau^{\e}_R\}}\d t\right) \nonumber\\&\leq 2\left[\E\left(\int_{\Delta}^T\int_{t-\Delta}^t\|\X_{s}^{\e,\delta,h^{\e}}\|_{\H}^{\frac{2(r-3)}{r-1}}\|\X_{s}^{\e,\delta,h^{\e}}\|_{\wi\L^{r+1}}^{\frac{2(r+1)}{r-1}}\chi_{\{t\leq\tau^{\e}_R\}}\d s\d t\right)\right]^{1/2}\nonumber\\&\quad\times\left[\E\left(\int_{\Delta}^T\int_{t-\Delta}^t\|\X_{s}^{\e,\delta,h^{\e}}-\X_{t-\Delta}^{\e,\delta,h^{\e}}\|_{\V}^2\chi_{\{t\leq\tau^{\e}_R\}}\d s\d t\right)\right]^{1/2}\nonumber\\&\leq 2\left[\Delta\E\left(\int_{0}^T\|\X_{t}^{\e,\delta,h^{\e}}\|_{\H}^{\frac{2(r-3)}{r-1}}\|\X_{t}^{\e,\delta,h^{\e}}\|_{\wi\L^{r+1}}^{\frac{2(r+1)}{r-1}}\chi_{\{t\leq\tau^{\e}_R\}}\d t\right)\right]^{1/2}\left[2\Delta\E\left(\int_{0}^T\|\X_{t}^{\e,\delta,h^{\e}}\|_{\V}^2\d t\right)\right]^{1/2}\nonumber\\&\leq 2\Delta T^{\frac{r-3}{2(r-1)}}\left[\E\left(\int_{0}^T\|\X_{t}^{\e,\delta,h^{\e}}\|_{\H}^{r-3}\|\X_{t}^{\e,\delta,h^{\e}}\|_{\wi\L^{r+1}}^{r+1}\chi_{\{t\leq\tau^{\e}_R\}}\d t\right)\right]^{\frac{1}{r-1}}\left[2\E\left(\int_{0}^T\|\X_{t}^{\e,\delta,h^{\e}}\|_{\V}^2\d t\right)\right]^{1/2}\nonumber\\&\leq C_{\mu,\alpha,\lambda_1,L_{\G},T,R}\Delta\left(1+\|\x\|_{\H}^2+\|\y\|_{\H}^2\right)^{\frac{r+1}{2(r-1)}}\nonumber\\&\leq C_{\mu,\alpha,\lambda_1,L_{\G},T,R}\Delta\left(1+\|\x\|_{\H}^2+\|\y\|_{\H}^2\right),
	\end{align}
since $\frac{r+1}{2(r-1)}\leq 1$, for all $r\geq 3$. 	Once again using H\"older's inequality,  Fubini's Theorem and \eqref{5.5z}, we estimate $\E\left(\int_{\Delta}^T|I_4(t)|\chi_{\{t\leq\tau^{\e}_R\}}\d t\right)$ as (see \eqref{4p31} also)
	\begin{align}
&	\E\left(\int_{\Delta}^T|I_4(t)|\chi_{\{t\leq\tau^{\e}_R\}}\d t\right)\nonumber\\&\leq 2\beta\E\left(\int_{\Delta}^T\int_{t-\Delta}^t\|\X_{s}^{\e,\delta,h^{\e}}\|_{\wi\L^{r+1}}^r\|\X_{s}^{\e,\delta,h^{\e}}-\X_{t-\Delta}^{\e,\delta,h^{\e}}\|_{\wi\L^{r+1}}\chi_{\{t\leq\tau^{\e}_R\}}\d s\d t\right)\nonumber\\&\leq 2\beta\left[\Delta\E\left(\int_{0}^T\|\X_{t}^{\e,\delta,h^{\e}}\|_{\wi\L^{r+1}}^{r+1}\d t\right)\right]^{\frac{r}{r+1}}\left[2^{r}\Delta\E\left(\int_0^T\|\X_{t}^{\e,\delta,h^{\e}}\|_{\wi\L^{r+1}}^{r+1}\d t\right)\right]^{\frac{1}{r+1}}\nonumber\\&\leq C_{\mu,\alpha,\beta,\lambda_1,L_{\G},T}\Delta(1+\|\x\|_{\H}^2+\|\y\|_{\H}^2). 
	\end{align}
	We estimate $\E\left(\int_{\Delta}^T|I_5(t)|\chi_{\{t\leq\tau^{\e}_R\}}\d t\right)$ using the Assumption \ref{ass3.6} (A1),   \eqref{5.5z} and \eqref{443} as (see \eqref{4p32} also)
	\begin{align}
	&	\E\left(\int_{\Delta}^T|I_5(t)|\chi_{\{t\leq\tau^{\e}_R\}}\d t\right)
	\nonumber\\&\leq 2\left[\E\left(\int_{\Delta}^T\int_{t-\Delta}^t\|\F(\X_{s}^{\e,\delta,h^{\e}},\Y^{\e,\delta,h^{\e}}_s)\|_{\H}^2\chi_{\{t\leq\tau^{\e}_R\}}\d s\d t\right)\right]^{1/2}\nonumber\\&\quad\times\left[\E\left(\int_{\Delta}^T\int_{t-\Delta}^t\|\X_{s}^{\e,\delta,h^{\e}}-\X_{t-\Delta}^{\e,\delta,h^{\e}}\|_{\H}^2\chi_{\{t\leq\tau^{\e}_R\}}\d s\d t\right)\right]^{1/2}\nonumber\\&\leq C\left[\Delta\E\left(\int_{0}^T(1+\|\X_{t}^{\e,\delta,h^{\e}}\|_{\H}^2+\|\Y^{\e,\delta,h^{\e}}_t\|_{\H}^2)\d t\right)\right]^{1/2}\left[2\Delta\E\left(\int_{0}^T\|\X_{t}^{\e,\delta,h^{\e}}\|_{\H}^2\d t\right)\right]^{1/2}\nonumber\\&\leq C_{\mu,\alpha,\lambda_1,L_{\G},T}\Delta(1+\|\x\|_{\H}^2+\|\y\|_{\H}^2). 
	\end{align}
	The term $\E\left(\int_{\Delta}^T|I_6(t)|\chi_{\{t\leq\tau^{\e}_R\}}\d t\right)$ can be estimated using the Assumption \ref{ass3.6} (A1) and \eqref{5.5z} as (see \eqref{433} also)
	\begin{align}
	&\E\left(\int_{\Delta}^T|I_6(t)|\chi_{\{t\leq\tau^{\e}_R\}}\d t\right)\nonumber\\&\leq2\left[\E\left(\int_{\Delta}^T\int_{t-\Delta}^t\|\sigma_1(\X_{s}^{\e,\delta,h^{\e}})\Q_1^{1/2}\|_{\mathcal{L}_2}^2\|h^{\e}_s\|_{\H}^2\chi_{\{t\leq\tau^{\e}_R\}}\d s\d t\right)\right]^{1/2}\nonumber\\&\quad\times\left[\E\left(\int_{\Delta}^T\int_{t-\Delta}^t\|\X_{s}^{\e,\delta,h^{\e}}-\X_{t-\Delta}^{\e,\delta,h^{\e}}\|_{\H}^2\chi_{\{t\leq\tau^{\e}_R\}}\d s\d t\right)\right]^{1/2}\nonumber\\&\leq C\left[\Delta\E\left(\int_{0}^T(1+\|\X_{t}^{\e,\delta,h^{\e}}\|_{\H}^2)\|h^{\e}_t\|_{\H}^2\chi_{\{t\leq\tau^{\e}_R\}}\d t\right)\right]^{1/2}\left[2\Delta\E\left(\int_{0}^T\|\X_{t}^{\e,\delta,h^{\e}}\|_{\H}^2\d t\right)\right]^{1/2}\nonumber\\&\leq C_{\mu,\alpha,\lambda_1,L_{\G},M,T,R}\Delta(1+\|\x\|_{\H}^2+\|\y\|_{\H}^2). 
	\end{align}
	Once again using the Assumption \ref{ass3.6} (A1) and   \eqref{5.5z}, we estimate $\E\left(\int_{\Delta}^T|I_7(t)|\chi_{\{t\leq\tau^{\e}_R\}}\d t\right)$ as 
	\begin{align}
	&\E\left(\int_{\Delta}^T|I_7(t)|\chi_{\{t\leq\tau^{\e}_R\}}\d t\right)\nonumber\\&\leq C\e\E\left(\int_{\Delta}^T\int_{t-\Delta}^t(1+\|\X_{s}^{\e,\delta,h^{\e}}\|_{\H}^2)\chi_{\{t\leq\tau^{\e}_R\}}\d s\d t\right)\leq C\e T\Delta \E\left[1+\sup_{t\in[0,T]}\|\X_{t}^{\e,\delta,h^{\e}}\|_{\H}^2\right]\nonumber\\&\leq C_{\mu,\alpha,\lambda_1,L_{\G},T}\Delta(1+\|\x\|_{\H}^2+\|\y\|_{\H}^2),
	\end{align}
since $\e\in(0,1)$.	Finally, using the Burkholder-Davis-Gundy inequality, the Assumption \ref{ass3.6} (A1), Fubini's theorem and   \eqref{5.5z},  we estimate $\E\left(\int_{\Delta}^T|I_8(t)|\chi_{\{t\leq\tau^{\e}_R\}}\d t\right)$ as 
	\begin{align}\label{333}
&	\E\left(\int_{\Delta}^T|I_8(t)|\chi_{\{t\leq\tau^{\e}_R\}}\d t\right)\nonumber\\&\leq C\int_{\Delta}^T\E\left[\left(\int_{t-\Delta}^t\|\sigma_1(\X_{s}^{\e,\delta,h^{\e}})\Q_1^{1/2}\|_{\mathcal{L}_2}^2\|\X_{s}^{\e,\delta,h^{\e}}-\X_{t-\Delta}^{\e,\delta,h^{\e}}\|_{\H}^2\chi_{\{t\leq\tau^{\e}_R\}}\d s\right)^{1/2}\right]\d t\nonumber\\&\leq CT^{1/2}\left[\E\left(\int_{\Delta}^T\int_{t-\Delta}^{t}\left(1+\|\X_{s}^{\e,\delta,h^{\e}}\|_{\H}^2\right)\|\X_{s}^{\e,\delta,h^{\e}}-\X_{t-\Delta}^{\e,\delta,h^{\e}}\|_{\H}^2\chi_{\{t\leq\tau^{\e}_R\}}\d s\d t\right)\right]^{1/2}\nonumber\\&\leq C_RT\Delta^{1/2}\left[\E\left(1+\sup_{t\in[0,T]}\|\X_t\|_{\H}^2\right)\right]^{1/2}\nonumber\\&\leq C_{\mu,\alpha,\lambda_1,L_{\G},T,R}\Delta^{1/2}(1+\|\x\|_{\H}^2+\|\y\|_{\H}^2). 
	\end{align}
	Combining \eqref{327}-\eqref{333}, we deduce that 
	\begin{align}\label{334}
	\E\left[\int_{\Delta}^T\|\X_{t}^{\e,\delta,h^{\e}}-\X^{\e,\delta,h^{\e}}_{t-\Delta}\|_{\H}^2\chi_{\{t\leq\tau^{\e}_R\}}\d t\right]\leq  C_{\mu,\alpha,\beta,\lambda_1,L_{\G},M,T,R}\Delta^{1/2}(1+\|\x\|_{\H}^{2}+\|\y\|_{\H}^{2}).
	\end{align}
	 A similar argument leads to 
	\begin{align}\label{335}
	\E\left[\int_{\Delta}^T\|\X_{t(\Delta)}^{\e,\delta,h^{\e}}-\X^{\e,\delta,h^{\e}}_{t-\Delta}\|_{\H}^2\chi_{\{t\leq\tau^{\e}_R\}}\d t\right]\leq   C_{\mu,\alpha,\beta,\lambda_1,L_{\G},M,T,R}\Delta^{1/2}(1+\|\x\|_{\H}^{2}+\|\y\|_{\H}^{2}).
	\end{align}
	Combining \eqref{3.25}, \eqref{334} and \eqref{335}, we obtain the required result \eqref{3.24}. 
\end{proof}

\subsection{Estimates of auxiliary process $\widehat{\Y}^{\e,\delta}_t$} We use the method proposed by Khasminskii, in \cite{RZK} to obtain the estimates for an auxiliary process. We introduce the auxiliary process $\widehat{\Y}^{\e,\delta}_t\in\H$ (see \eqref{3.37} below) and divide the interval $[0,T]$ into subintervals of  size $\Delta$, where $\Delta$ is a fixed positive number, which depends on $\delta$ and it will be chosen later. Let us construct the process $\widehat{\Y}^{\e,\delta}_t$ with the initial value $\widehat{\Y}^{\e,\delta}_0=\Y^{\e,\delta}_0=\y$, and for any $k\in\mathbb{N}$ and $t\in[k\Delta,\min\{(k+1)\Delta,T\}]$ as 
\begin{align}\label{3.37}
\widehat{\Y}^{\e,\delta}_t&=\widehat{\Y}^{\e,\delta}_{k\Delta}-\frac{\mu}{\delta}\int_{k\Delta}^t\A\widehat{\Y}^{\e,\delta}_s\d s-\frac{\alpha}{\delta}\int_{k\Delta}^t\widehat{\Y}^{\e,\delta}_s\d s-\frac{\beta}{\delta}\int_{k\Delta}^t\mathcal{C}(\widehat{\Y}^{\e,\delta}_s)\d s+\frac{1}{\delta}\int_{k\Delta}^t\G(\X^{\e,\delta,h^{\e}}_{k\Delta},\widehat{\Y}^{\e,\delta}_s)\d s\nonumber\\&\quad+\frac{1}{\sqrt{\delta}}\int_{k\Delta}^t\sigma_2(\X^{\e,\delta,h^{\e}}_{k\Delta},\widehat{\Y}^{\e,\delta}_s)\Q_2^{1/2}\d\W_s, \ \mathbb{P}\text{-a.s.}, 
\end{align}
which is equivalent to 
\begin{equation}\label{338}
\left\{
\begin{aligned}
\d\widehat{\Y}^{\e,\delta}_t&=-\frac{1}{\delta}\left[\mu\A\widehat{\Y}^{\e,\delta}_t+\alpha\widehat{\Y}^{\e,\delta}_t+\beta\mathcal{C}(\widehat{\Y}^{\e,\delta}_t)-\G(\X^{\e,\delta,h^{\e}}_{t(\Delta)},\widehat{\Y}^{\e,\delta}_t)\right]\d t\\&\quad+\frac{1}{\sqrt{\delta}}\sigma_2(\X^{\e,\delta,h^{\e}}_{t(\Delta)},\widehat{\Y}^{\e,\delta}_t)\Q_2^{1/2}\d\W_t,\\
\widehat{\Y}^{\e,\delta}_0&=\y. 
\end{aligned}\right. 
\end{equation}
The following energy estimate satisfied by $\widehat{\Y}^{\e,\delta}_t$ can be proved in a similar way as in Theorem \ref{thm5.10} (see Lemma 4.2, \cite{MTM11} also). 
\begin{lemma}\label{lem3.9}
	For any $\x,\y\in\H$, $T>0$ and $\e\in(0,1)$, there exists a constant $C_{\mu,\alpha,\lambda_1,L_{\G},M,T}>0$ such that the strong solution $\Y^{\e,\delta}_{t}$ to the system \eqref{338} satisfies: 
	\begin{align}\label{341}
	\sup_{t\in[0,T]}	\E\left[\|\widehat\Y^{\e,\delta}_t\|_{\H}^{2}\right]\leq C_{\mu,\alpha,\lambda_1,L_{\G},M,T}\left(1+\|\x\|_{\H}^{2}+\|\y\|_{\H}^{2}\right). 
	\end{align}
\end{lemma}
Our next aim is to establish an estimate on the difference between the processes $\Y^{\e,\delta,h^{\e}}_t$ and $\widehat\Y^{\e,\delta}_t$. 
\begin{lemma}\label{lem3.10}
	For any $\x,\y\in\H$, $T>0$ and $\e,\delta\in(0,1)$, there exists a constant $C_{\mu,\alpha,\beta,\lambda_1,L_{\G},T}>0$ such that 
	\begin{align}\label{3.42}
	\mathbb{E}\left(\int_0^{T}\|\Y^{\e,\delta,h^{\e}}_t-\widehat\Y^{\e,\delta}_t\|_{\H}^2\d t\right)\leq C_{\mu,\alpha,\beta,\lambda_1,L_{\G},M,T}(1+\|\x\|_{\H}^2+\|\y\|_{\H}^2)\left[\left(\frac{\delta}{\e}\right)+\Delta^{1/2}\right].
	\end{align}
\end{lemma}
\begin{proof}
	Let us define $\mathbf{U}^{\e}_t:=\Y^{\e,\delta,h^{\e}}_t-\widehat\Y^{\e,\delta}_t$. Then $\mathbf{U}^{\e}_{t}$ satisfies the following It\^o stochastic differential: 
	\begin{equation}
	\left\{
	\begin{aligned}
	\d\U^{\e}_t&=-\frac{1}{\delta}\left[\mu\A\U^{\e}_t+\alpha\U^{\e}_t+\beta(\mathcal{C}(\Y^{\e,\delta,h^{\e}}_t)-\mathcal{C}(\widehat\Y^{\e,\delta}_t))\right]\d t\\&\quad+\frac{1}{\delta}\left[(\G(\X_{t}^{\e,\delta,h^{\e}},\Y^{\e,\delta,h^{\e}}_t)-\G(\X^{\e,\delta,h^{\e}}_{t(\Delta)},\widehat{\Y}^{\e,\delta}_t))\right]\d t\\&\quad+\frac{1}{\sqrt{\delta\e}}\sigma_2(\X^{\e,\delta,h^{\e}}_t,\Y^{\e,\delta,h^{\e}}_t)\Q_2^{1/2}h_t^{\e}\d t \\&\quad+\frac{1}{\sqrt{\delta}}\left[\sigma_2(\X_{t}^{\e,\delta,h^{\e}},{\Y}^{\e,\delta,h^{\e}}_t)-\sigma_2(\X^{\e,\delta,h^{\e}}_{t(\Delta)},\widehat{\Y}^{\e,\delta}_t)\right]\Q_2^{1/2}\d\W_t,\\
	\U^{\e}_0&=\mathbf{0}. 
	\end{aligned}
	\right. 
	\end{equation}
	An application of the infinite dimensional It\^o formula to the process $\|\U^{\e}_{t}\|_{\H}^2$ yields 
	\begin{align}\label{344}
	\|\U^{\e}_t\|_{\H}^2&=-\frac{2\mu}{\delta}\int_0^t\|\U^{\e}_s\|_{\V}^2\d s-\frac{2\alpha}{\delta}\int_0^t\|\U^{\e}_s\|_{\H}^2\d s-\frac{2\beta}{\delta}\int_0^t\langle\mathcal{C}(\Y_{s}^{\e,\delta,h^{\e}})-\mathcal{C}(\widehat\Y_{s}^{\e,\delta}),\U^{\e}_s\rangle\d s\nonumber\\&\quad+\frac{2}{\delta}\int_0^t(\G(\X_{s}^{\e,\delta,h^{\e}},{\Y}^{\e,\delta,h^{\e}}_s)-\G(\X^{\e,\delta,h^{\e}}_{s(\Delta)},\widehat{\Y}^{\e,\delta}_s),\U^{\e}_s)\d s\nonumber\\&\quad+\frac{2}{\sqrt{\delta\e}}\int_0^t(\sigma_2(\X^{\e,\delta,h^{\e}}_s,\Y^{\e,\delta,h^{\e}}_s)\Q_2^{1/2}h_s^{\e},\U^{\e}_s)\d s\nonumber\\&\quad+\frac{1}{\sqrt{\delta}}\int_0^t([\sigma_2(\X_{s}^{\e,\delta,h^{\e}},{\Y}^{\e,\delta,h^{\e}}_s)-\sigma_2(\X^{\e,\delta,h^{\e}}_{s(\Delta)},\widehat{\Y}^{\e,\delta}_s)]\Q_2^{1/2}\d\W_s,\U^{\e}_s)\nonumber\\&\quad+\frac{1}{\delta }\int_0^t\|[\sigma_2(\X_{s}^{\e,\delta,h^{\e}},{\Y}^{\e}_s)-\sigma_2(\X^{\e,\delta,h^{\e}}_{s(\Delta)},\widehat{\Y}^{\e,\delta}_s)]\Q_2^{1/2}\|_{\mathcal{L}_2}^2\d s, \ \mathbb{P}\text{-a.s.},
	\end{align}
	for all $t\in[0,T]$. Taking expectation in \eqref{344}, we obtain 
	\begin{align}
	\E\left[	\|\U^{\e}_t\|_{\H}^2\right]&=-\frac{2\mu}{\delta}\int_0^t\E\left[\|\U^{\e}_s\|_{\V}^2\right]\d s-\frac{2\alpha}{\delta}\int_0^t\E\left[\|\U^{\e}_s\|_{\H}^2\right]\d s\nonumber\\&\quad-\frac{2\beta}{\delta}\E\left[\int_0^t\langle\mathcal{C}(\Y_{s}^{\e,\delta,h^{\e}})-\mathcal{C}(\widehat\Y_{s}^{\e,\delta}),\U^{\e}_s\rangle\d s\right]\nonumber\\&\quad+\frac{2}{\delta}\int_0^t\E\left[(\G(\X_{s}^{\e,\delta,h^{\e}},{\Y}^{\e,\delta,h^{\e}}_s)-\G(\X^{\e,\delta,h^{\e}}_{s(\Delta)},\widehat{\Y}^{\e,\delta}_s),\U^{\e}_s)\right]\d s\nonumber\\&\quad+\frac{2}{\sqrt{\delta\e}}\int_0^t\E\left[(\sigma_2(\X^{\e,\delta,h^{\e}}_s,\Y^{\e,\delta,h^{\e}}_s)\Q_2^{1/2}h_s^{\e},\U^{\e}_s)\right]\d s\nonumber\\&\quad+\frac{1}{\delta}\int_0^t\E\left[\|[\sigma_2(\X_{s}^{\e,\delta,h^{\e}},{\Y}^{\e,\delta,h^{\e}}_s)-\sigma_2(\X^{\e,\delta,h^{\e}}_{s(\Delta)},\widehat{\Y}^{\e,\delta}_s)]\Q_2^{1/2}\|_{\mathcal{L}_2}^2\right]\d s. 
	\end{align}
	Thus, it is immediate that 
	\begin{align}\label{3.46}
	\frac{\d}{\d t}	\E\left[	\|\U^{\e}_t\|_{\H}^2\right]&=-\frac{2\mu}{\delta}\E\left[\|\U^{\e}_t\|_{\V}^2\right]-\frac{2\alpha}{\delta}\E\left[\|\U^{\e}_t\|_{\H}^2\right]-\frac{2\beta}{\delta}\E\left[\langle\mathcal{C}(\Y^{\e,\delta,h^{\e}}_t)-\mathcal{C}(\widehat\Y^{\e,\delta}_t),\U^{\e}_t\rangle\right]\nonumber\\&\quad+\frac{2}{\delta}\E\left[(\G(\X_{t}^{\e,\delta,h^{\e}},{\Y}^{\e,\delta,h^{\e}}_t)-\G(\X^{\e,\delta,h^{\e}}_{t(\Delta)},\widehat{\Y}^{\e,\delta}_t),\U^{\e}_t)\right]\nonumber\\&\quad+\frac{2}{\sqrt{\delta\e}}\E\left[(\sigma_2(\X^{\e,\delta,h^{\e}}_t,\Y^{\e,\delta,h^{\e}}_t)\Q_2^{1/2}h_t^{\e},\U^{\e}_t)\right]\nonumber\\&\quad+\frac{1}{\delta}\E\left[\|[\sigma_2(\X_{t}^{\e,\delta,h^{\e}},{\Y}^{\e,\delta,h^{\e}}_t)-\sigma_2(\X^{\e,\delta,h^{\e}}_{t(\Delta)},\widehat{\Y}^{\e,\delta}_t)]\Q_2^{1/2}\|_{\mathcal{L}_2}^2\right], 
	\end{align}
	for a.e. $t\in[0,T]$. Applying \eqref{214}, we find 
	\begin{align}\label{3p47}
&-\frac{2\beta}{\e}\langle\mathcal{C}(\Y^{\e,\delta,h^{\e}}_t)-\mathcal{C}(\widehat\Y^{\e,\delta}_t),\U^{\e}_t\rangle\leq-\frac{\beta}{2^{r-2}\e}\|\U^{\e}_t\|_{\wi\L^{r+1}}^{r+1}.
	\end{align}
	Using the Assumption \ref{ass3.6} (A1), we get 
	\begin{align}\label{3.47}
&\frac{2}{\delta}(\G(\X_{t}^{\e,\delta,h^{\e}},\Y^{\e,\delta,h^{\e}}_t)-\G(\X^{\e,\delta,h^{\e}}_{t(\Delta)},\widehat{\Y}^{\e,\delta}_t),\U^{\e}_t)\nonumber\\&\leq \frac{2}{\delta}\|\G(\X_{t}^{\e,\delta,h^{\e}},\Y^{\e,\delta,h^{\e}}_t)-\G(\X^{\e,\delta,h^{\e}}_{t(\Delta)},\widehat{\Y}^{\e,\delta}_t)\|_{\H}\|\U^{\e}_t\|_{\H}\nonumber\\&\leq \frac{C}{\delta}\|\X_{t}^{\e,\delta,h^{\e}}-\X^{\e,\delta,h^{\e}}_{t(\Delta)}\|_{\H}\|\U^{\e}_t\|_{\H}+\frac{2L_{\G}}{\delta}\|\U^{\e}_t\|_{\H}^2\nonumber\\&\leq\left(\frac{\mu\lambda_1}{2\delta}+\frac{2L_{\G}}{\delta}\right)\|\U^{\e}_t\|_{\H}^2+\frac{C}{\mu\lambda_1\delta}\|\X_{t}^{\e,\delta,h^{\e}}-\X^{\e,\delta,h^{\e}}_{t(\Delta)}\|_{\H}^2. 
	\end{align} 
	Making use of the Assumption \ref{ass3.6} (A2), H\"older's and Young's inequalities, we have 
	\begin{align}
&\frac{2}{\sqrt{\delta\e}}(\sigma_2(\X^{\e,\delta,h^{\e}}_t,\Y^{\e,\delta,h^{\e}}_t)\Q_2^{1/2}h_t^{\e},\U^{\e}_t)\nonumber\\&\leq\frac{2}{\sqrt{\delta\e}}\|\sigma_2(\X^{\e,\delta,h^{\e}}_t,\Y^{\e,\delta,h^{\e}}_t)\Q_2^{1/2}\|_{\mathcal{L}_2}\|h_t^{\e}\|_{\H}\|\U^{\e}_t\|_{\H}\nonumber\\&\leq \frac{\mu\lambda_1}{2\delta}\|\U^{\e}_t\|_{\H}^2+\frac{C}{\mu\lambda_1\delta}\|\sigma_2(\X^{\e,\delta,h^{\e}}_t,\Y^{\e,\delta,h^{\e}}_t)\Q_2^{1/2}\|_{\mathcal{L}_2}^2\|h_t^{\e}\|_{\H}^2\nonumber\\&\leq \frac{\mu\lambda_1}{2\delta}\|\U^{\e}_t\|_{\H}^2+\frac{C}{\mu\lambda_1\e}(1+\|\X^{\e,\delta,h^{\e}}_t\|_{\H}^2)\|h_t^{\e}\|_{\H}^2.
	\end{align}
	Similarly, using the Assumption \ref{ass3.6} (A1), we obtain
	\begin{align}\label{3.48}
\frac{1}{\delta}	\|[\sigma_2(\X_{t}^{\e,\delta,h^{\e}},{\Y}^{\e,\delta,h^{\e}}_t)-\sigma_2(\X^{\e,\delta,h^{\e}}_{t(\Delta)},\widehat{\Y}^{\e,\delta}_t)]\Q_2^{1/2}\|_{\mathcal{L}_2}^2&\leq \frac{C}{\delta}\|\X_{t}^{\e,\delta,h^{\e}}-\X^{\e,\delta,h^{\e}}_{t(\Delta)}\|_{\H}^2+\frac{2L_{\sigma_2}^2}{\delta}\|\U^{\e}_t\|_{\H}^2. 
	\end{align}
	Combining \eqref{3p47}-\eqref{3.48} and substituting it in \eqref{3.46}, we deduce that 
	\begin{align}\label{3.49}
		\frac{\d}{\d t}	\E\left[	\|\U^{\e}_t\|_{\H}^2\right]&\leq-\frac{1}{\delta}(\mu\lambda_1+2\alpha-2L_{\G}-2L_{\sigma_2}^2)\E\left[\|\U^{\e}_t\|_{\H}^2\right]+\frac{C}{\mu\lambda_1\e}\E\left[(1+\|\X^{\e,\delta,h^{\e}}_t\|_{\H}^2)\|h_t^{\e}\|_{\H}^2\right]\nonumber\\&\quad+\frac{C}{\delta}\left(1+\frac{1}{\mu\lambda_1}\right)\E\left[\|\X_{t}^{\e,\delta,h^{\e}}-\X^{\e,\delta,h^{\e}}_{t(\Delta)}\|_{\H}^2\right]. 
	\end{align}
	Using the Assumption \ref{ass3.6} (A3) and variation of constants  formula, from \eqref{3.49}, we infer that 
	\begin{align}
	\E\left[\|\U^{\e}_t\|_{\H}^2\right]&\leq \frac{C}{\mu\lambda_1\e}\int_0^te^{-\frac{\kappa}{\delta}(t-s)}\E\left[(1+\|\X^{\e,\delta,h^{\e}}_s\|_{\H}^2)\|h_s^{\e}\|_{\H}^2\right]\d s \nonumber\\&\quad+\frac{C}{\delta}\left(1+\frac{1}{\mu\lambda_1}\right)\int_0^te^{-\frac{\kappa}{\delta}(t-s)}\E\left[\|\X_{s}^{\e,\delta,h^{\e}}-\X^{\e,\delta,h^{\e}}_{s(\Delta)}\|_{\H}^2\right]\d s,
	\end{align}
	where $\kappa=\mu\lambda_1+2\alpha-2L_{\G}-2L_{\sigma_2}^2>0$. Applying Fubini's theorem, for any $T>0$, we have 
	\begin{align}
	\E\left[\int_0^{T}\|\U^{\e}_t\|_{\H}^2\d t\right]&\leq  \frac{C}{\mu\lambda_1\e}\int_0^{T}\int_0^te^{-\frac{\kappa}{\delta}(t-s)}\E\left[(1+\|\X^{\e,\delta,h^{\e}}_s\|_{\H}^2)\|h_s^{\e}\|_{\H}^2\right]\d s\d t\nonumber\\&\quad+ \frac{C}{\delta}\left(1+\frac{1}{\mu\lambda_1}\right)\int_0^{T}\int_0^te^{-\frac{\kappa}{\delta}(t-s)}\E\left[\|\X_{s}^{\e,\delta,h^{\e}}-\X^{\e,\delta,h^{\e}}_{s(\Delta)}\|_{\H}^2\right]\d s\d t\nonumber\\&=\frac{C}{\mu\lambda_1\e}\E\left[\int_0^{T}(1+\|\X^{\e,\delta,h^{\e}}_s\|_{\H}^2)\|h_s^{\e}\|_{\H}^2\int_s^Te^{-\frac{\kappa}{\delta}(t-s)}\d t\d s\right]\nonumber\\&\quad+\frac{C}{\delta}\left(1+\frac{1}{\mu\lambda_1}\right)\E\left[\int_0^T\|\X_{s}^{\e,\delta,h^{\e}}-\X^{\e,\delta,h^{\e}}_{s(\Delta)}\|_{\H}^2\left(\int_s^Te^{-\frac{\kappa}{\delta}(t-s)}\d t\right)\d s\right]\nonumber\\&\leq\frac{C}{\mu\lambda_1\kappa}\left(\frac{\delta}{\e}\right)\E\left[\sup_{t\in[0,T]}(1+\|\X^{\e,\delta,h^{\e}}_t\|_{\H}^2)\int_0^{T}\|h_t^{\e}\|_{\H}^2\d t\right]\nonumber\\&\quad+\frac{C}{\kappa}\left(1+\frac{1}{\mu\lambda_1}\right)\E\left[\int_0^T\|\X_{t}^{\e,\delta,h^{\e}}-\X^{\e,\delta,h^{\e}}_{t(\Delta)}\|_{\H}^2\d t\right]\nonumber\\&\leq C_{\mu,\alpha,\beta,\lambda_1,L_{\G},M,T}(1+\|\x\|_{\H}^2+\|\y\|_{\H}^2)\left[\left(\frac{\delta}{\e}\right)+\Delta^{1/2}\right],
	\end{align}
	using \eqref{5.5z} and Lemma \ref{lem3.8} (see \eqref{3.24}), which completes the proof. 
\end{proof}

\begin{remark}
	It can be easily seen from the stochastic differential equations corresponding to  $\Y^{\e,\delta,h^{\e}}_t$ and $\widehat\Y^{\e,\delta}_t$ (see \eqref{442} and \eqref{338}) that the control term involving $h^{\e}$ in $\Y^{\e,\delta,h^{\e}}_t$  vanishes in $\widehat\Y^{\e,\delta}_t$. From Lemma \ref{lem3.10}, we infer that the additional control term takes no effect as $\e\to 0$ due to the Assumption \ref{ass3.6} (A4). 
\end{remark}

Our next aim is to establish an estimate for the process ${\X}^{\e,\delta,h^{\e}}_{t}-\bar{\X}_{t}^h$. For $n=2$ and $r\in[1,3]$, we need to construct an another stopping time to obtain an estimate for $\|\X_{t}^{\e,\delta,h^{\e}}-\bar\X_t^h\|_{\H}^2$. For fixed $\e\in(0,1)$ and $R>0$, we define 
\begin{align}\label{stop1}
\wi\tau^{\e}_R:=\inf_{t\geq 0}\left\{t:\|\X_{s}^{\e,\delta,h^{\e}}\|_{\H}+\int_0^t\|\X_{s}^{\e,\delta,h^{\e}}\|_{\V}^2\d s> R\right\}. 
\end{align}
It is clear that $\wi\tau_R^{\e}(\omega)\leq\tau_R^{\e}(\omega)$, for all $\omega\in\Omega$.  We see in the next Lemma that such a stopping time is not needed for $r\in(3,\infty)$ for $n=2$ and $r\in[3,\infty)$ for $n=3$ ($2\beta\mu> 1,$ for $r=n=3$). 

\begin{lemma}\label{lem3.14}
	For any $\x,\y\in\H$, $T>0$ and $\e\in(0,1)$, the following estimate holds: 
	\begin{align}\label{389}
		&	\E\left[	\sup_{t\in[0,T\wedge\wi\tau_R^{\e}]}	\|\Z^{\e}_{t}\|_{\H}^2+\mu\int_0^{T\wedge\wi\tau_R^{\e}}\|\Z^{\e}_t\|_{\V}^2\d t+\frac{\beta}{2^{r-2}}\int_0^{T\wedge\wi\tau_R^{\e}}\|\Z^{\e}_t\|_{\wi\L^{r+1}}^{r+1}\d t\right]\nonumber\\&\leq C_{\mu,\alpha,\beta,\lambda_1,L_{\G},L_{\sigma_2},M,T,R}(1+\|\x\|_{\H}^3+\|\y\|_{\H}^3)\left[\e^2+\left(\frac{\delta}{\e}\right)+\delta^{1/8}\right]\nonumber\\&\quad+\E\left[\int_0^{T}\|\sigma_1(\bar\X^{h}_t)\Q_1^{1/2}(h^{\e}_t-h_t)\|_{\H}^2\d t\right], 
	\end{align}	
	for $n=2$ and $r\in[1,3]$, and 
	\begin{align}\label{390}
	&	\E\left[\sup_{t\in[0,T\wedge\tau_R^{\e}]}	\|\Z^{\e}_{t}\|_{\H}^2+\mu\int_0^{T\wedge\tau_R^{\e}}\|\Z^{\e}_t\|_{\V}^2\d t+\frac{\beta}{2^{r-2}}\int_0^{T\wedge\tau_R^{\e}}\|\Z^{\e}_t\|_{\wi\L^{r+1}}^{r+1}\d t\right]\nonumber\\&\leq C_{\mu,\alpha,\beta,\lambda_1,L_{\G},L_{\sigma_2},M,T,R}(1+\|\x\|_{\H}^2+\|\y\|_{\H}^2)\left[\e^2+\left(\frac{\delta}{\e}\right)+\delta^{1/8}\right]\nonumber\\&\quad+\E\left[\int_0^{T}\|\sigma_1(\bar\X^{h}_t)\Q_1^{1/2}(h^{\e}_t-h_t)\|_{\H}^2\d t\right], 
	\end{align}	
	for $n=2$, $r\in(3,\infty)$ and $n=3$, $r\in[3,\infty)$ ($2\beta\mu> 1,$ for $r=3$). Here, $\tau_R^{\e}$ and $\wi\tau_R^{\e}$ are stopping times defined in \eqref{stop} and \eqref{stop1}, respectively. 
\end{lemma}
\begin{proof}
	Let us denote $\Z^{\e}_{t}:=\X_{t}^{\e,\delta,h^{\e}}-\bar{\X}_{t}^h$, where $(\X_{t}^{\e,\delta,h^{\e}},\Y_{t}^{\e,\delta,h^{\e}})$ is the unique strong solution of the system \eqref{5.4z} and $\bar{\X}_{t}^h$ is the unique weak solution of the system \eqref{5.4y}. Then $\Z^{\e}_{t}$ satisfies the following  It\^o stochastic differential: 
		\begin{equation}\label{3.91}
	\left\{
	\begin{aligned}
	\d\Z^{\e}_{t}&=-[\mu\A\Z^{\e}_{t}+(\B(\X_{t}^{\e,\delta,h^{\e}})-\B(\bar\X^{h}_t))+\beta(\mathcal{C}(\X_{t}^{\e,\delta,h^{\e}})-\mathcal{C}(\bar\X^{h}_t))]\d t\\&\quad+[\F(\X_{t}^{\e,\delta,h^{\e}},\Y_{t}^{\e,\delta,h^{\e}})-\bar \F(\bar\X^{h}_{t})]\d t+[\sigma_1(\X^{\e,\delta,h^{\e}}_t)\Q_1^{1/2}h^{\e}_t-\sigma_1(\bar\X^{h}_t)\Q_1^{1/2}h_t]\d t\\&\quad+\sqrt{\e}\sigma_1(\X_{t}^{\e,\delta,h^{\e}})\Q_1^{1/2}\d\W_t,\\
	\Z^{\e}_{0}&=\mathbf{0}. 
	\end{aligned}\right. 
	\end{equation}
	\vskip 2 mm
\noindent	\textbf{Case 1: $n=2$ and $r\in[1,3]$.} We first consider the case $n=2$ and $r\in[1,3]$. An application of the infinite dimensional It\^o formula to the process $\|\Z^{\e}_{t}\|_{\H}^2$ yields 
	\begin{align}\label{3.92}
	\|\Z^{\e}_{t}\|_{\H}^2&=-2 \mu\int_0^t\|\Z^{\e}_s\|_{\V}^2\d s-2\alpha\int_0^t\|\Z^{\e}_s\|_{\H}^2\d s-2\int_0^t\langle (\B(\X_{s}^{\e,\delta,h^{\e}})-\B(\bar{\X}^{h}_s)),\Z^{\e}_s\rangle\d s\nonumber\\&\quad-2\beta\int_0^t\langle\mathcal{C}(\X_{s}^{\e,\delta,h^{\e}})-\mathcal{C}(\bar{\X}^{h}_s),\Z^{\e}_s\rangle\d s+2\int_0^t(\F(\X_{s}^{\e,\delta,h^{\e}},\Y_{s}^{\e,\delta,h^{\e}})-\bar\F(\bar{\X}^{h}_{s}),\Z^{\e}_s)\d s\nonumber\\&\quad+2\int_0^t([\sigma_1(\X^{\e,\delta,h^{\e}}_s)\Q_1^{1/2}h^{\e}_s-\sigma_1(\bar\X^{h}_s)\Q_1^{1/2}h_s],\Z^{\e}_s)\d s+\e\int_0^t\|\sigma_1(\X_{s}^{\e,\delta,h^{\e}})\Q_1^{1/2}\|_{\mathcal{L}_2}^2\d s\nonumber\\&\quad+2\sqrt{\e}\int_0^t(\sigma_1(\X_{s}^{\e,\delta,h^{\e}})\Q_1^{1/2}\d\W_s,\Z^{\e}_s)\nonumber\\&=- 2 \mu\int_0^t\|\Z^{\e}_s\|_{\V}^2\d s- 2 \alpha\int_0^t\|\Z^{\e}_s\|_{\H}^2\d s-2\int_0^t\langle (\B(\X_{s}^{\e,\delta,h^{\e}})-\B(\bar{\X}^{h}_s)),\Z^{\e}_s\rangle\d s\nonumber\\&\quad-2\beta\int_0^t\langle\mathcal{C}(\X_{s}^{\e,\delta,h^{\e}})-\mathcal{C}(\bar{\X}^{h}_s),\Z^{\e}_s\rangle\d s+2\int_0^t(\bar \F(\X_{s}^{\e,\delta,h^{\e}})-\bar \F(\bar{\X}^{h}_{s}),\Z^{\e}_s)\d s\nonumber\\&\quad+2\int_0^t(\F(\X_{s}^{\e,\delta,h^{\e}},\Y_{s}^{\e,\delta,h^{\e}})-\bar \F(\X_{s}^{\e,\delta,h^{\e}})-\F(\X_{s(\Delta)}^{\e,\delta,h^{\e}},\widehat\Y_{s}^{\e,\delta})+\bar \F(\X_{s(\Delta)}^{\e,\delta,h^{\e}}),\Z^{\e}_s)\d s\nonumber\\&\quad+2\int_0^t(\F(\X_{s(\Delta)}^{\e,\delta,h^{\e}},\widehat\Y_{s}^{\e,\delta})-\bar \F(\X_{s(\Delta)}^{\e,\delta,h^{\e}}),\Z^{\e}_s-\Z^{\e}_{s(\Delta)})\d s\nonumber\\&\quad+2\int_0^t(\F(\X_{s(\Delta)}^{\e,\delta,h^{\e}},\widehat\Y_{s}^{\e,\delta})-\bar \F(\X_{s(\Delta)}^{\e,\delta,h^{\e}}),\Z^{\e}_{s(\Delta)})\d s \nonumber\\&\quad+2\int_0^t([\sigma_1(\X^{\e,\delta,h^{\e}}_s)\Q_1^{1/2}h^{\e}_s-\sigma_1(\bar\X^{h}_s)\Q_1^{1/2}h_s],\Z^{\e}_s)\d s\nonumber\\&\quad+\e\int_0^t\|\sigma_1(\X_{s}^{\e,\delta,h^{\e}})\Q_1^{1/2}\|_{\mathcal{L}_2}^2\d s+2\sqrt{\e}\int_0^t(\sigma_1(\X_{s}^{\e,\delta,h^{\e}})\Q_1^{1/2}\d\W_s,\Z^{\e}_s),\  \mathbb{P}\text{-a.s.},
	\end{align}
	for all $t\in[0,T]$. Using H\"older's, Ladyzhenskaya's and Young's inequalities, we estimate $2|\langle(\B(\X_{s}^{\e,\delta,h^{\e}})-\B(\bar{\X}^{h}_s)),\Z^{\e}_s\rangle|$ as 
	\begin{align}\label{3.93}
	2|	\langle(\B(\X_{s}^{\e,\delta,h^{\e}})-\B(\bar{\X}^{h}_s)),\Z^{\e}_s\rangle|&= 2|\langle\B(\Z^{\e}_{s},\X_{s}^{\e,\delta,h^{\e}}),\Z^{\e}_{s}\rangle |\leq 2\|\X_{s}^{\e,\delta,h^{\e}}\|_{\V}\|\Z^{\e}_{s}\|_{\wi\L^4}^2\nonumber\\&\leq 2\sqrt{2}\|\X_{s}^{\e,\delta,h^{\e}}\|_{\V}\|\Z^{\e}_{s}\|_{\H}\|\Z^{\e}_{s}\|_{\V}\nonumber\\&\leq\mu\|\Z^{\e}_{s}\|_{\V}^2+\frac{2}{\mu}\|\X_{s}^{\e,\delta,h^{\e}}\|_{\V}^2\|\Z^{\e}_s\|_{\H}^2. 
	\end{align}
	Using \eqref{2.23} and \eqref{a215}, we know that 
	\begin{align}
	-2\beta\langle\mathcal{C}(\X_{s}^{\e,\delta,h^{\e}})-\mathcal{C}(\bar{\X}^{h}_s),\Z^{\e}_s\rangle\leq-\frac{\beta}{2^{r-2}}\|\Z^{\e}_s\|_{\wi\L^{r+1}}^{r+1},
	\end{align}
	for $r\in[1,\infty)$. Using \eqref{3p93}, H\"older's and Young's inequalities, we estimate $2(\bar \F(\X_{s}^{\e,\delta,h^{\e}})-\bar \F(\bar{\X}^{h}_{s}),\Z^{\e}_s)$ as 
	\begin{align}
	2(\bar \F(\X_{s}^{\e,\delta,h^{\e}})-\bar \F(\bar{\X}^{h}_{s}),\Z^{\e}_s)&\leq 2\|\bar \F(\X_{s}^{\e,\delta,h^{\e}})-\bar \F(\bar{\X}^{h}_{s})\|_{\H}\|\Z^{\e}_s\|_{\H}\leq C_{\mu,\alpha,\lambda_1,L_{\G},L_{\sigma_2}}\|\Z^{\e}_{s}\|_{\H}^2. 
	\end{align}
	Similarly, we estimate $2(\F(\X_{s}^{\e,\delta,h^{\e}},\Y_{s}^{\e,\delta,h^{\e}})-\bar \F(\X_{s}^{\e,\delta,h^{\e}})-\F(\X_{s(\Delta)}^{\e,\delta,h^{\e}},\widehat\Y_{s}^{\e,\delta})+\bar \F(\X_{s(\Delta)}^{\e,\delta,h^{\e}}),\Z^{\e}_s)$ as 
	\begin{align}
	&2(\F(\X_{s}^{\e,\delta,h^{\e}},\Y_{s}^{\e,\delta,h^{\e}})-\bar \F(\X_{s}^{\e,\delta,h^{\e}})-\F(\X_{s(\Delta)}^{\e,\delta,h^{\e}},\widehat\Y_{s}^{\e,\delta})+\bar \F(\X_{s(\Delta)}^{\e,\delta,h^{\e}}),\Z^{\e}_s)\nonumber\\&\leq 2\left(\|\F(\X_{s}^{\e,\delta,h^{\e}},\Y_{s}^{\e,\delta,h^{\e}})-\F(\X_{s(\Delta)}^{\e,\delta,h^{\e}},\widehat\Y_{s}^{\e,\delta})\|_{\H}+\|\bar \F(\X_{s(\Delta)}^{\e,\delta,h^{\e}})-\bar \F(\X_{s}^{\e,\delta,h^{\e}})\|_{\H}\right)\|\Z^{\e}_s\|_{\H}\nonumber\\&\leq\|\Z^{\e}_s\|_{\H}^2+C_{\mu,\alpha,\lambda_1,L_{\G},L_{\sigma_2}}(\|\X_{s}^{\e,\delta,h^{\e}}-\X_{s(\Delta)}^{\e,\delta,h^{\e}}\|_{\H}^2+\|\Y_{s}^{\e,\delta,h^{\e}}-\widehat\Y_{s}^{\e,\delta}\|_{\H}^2).
	\end{align}
	Using the Assumption \ref{ass3.6} (A1), H\"older's and Young's inequalities, we estimate the term  $2\int_0^t(\F(\X_{s(\Delta)}^{\e,\delta,h^{\e}},\widehat\Y_{s}^{\e,\delta})-\bar \F(\X_{s(\Delta)}^{\e,\delta,h^{\e}}),\Z^{\e}_s-\Z^{\e}_{s(\Delta)})\d s$ as 
	\begin{align}\label{3.97}
	&2\int_0^t(\F(\X_{s(\Delta)}^{\e,\delta,h^{\e}},\widehat\Y_{s}^{\e,\delta})-\bar \F(\X_{s(\Delta)}^{\e,\delta,h^{\e}}),\Z^{\e}_s-\Z^{\e}_{s(\Delta)})\d s\nonumber\\&\leq 2\int_0^t \|\F(\X_{s(\Delta)}^{\e,\delta,h^{\e}},\widehat\Y_{s}^{\e,\delta})-\bar \F(\X_{s(\Delta)}^{\e,\delta,h^{\e}})\|_{\H}\|\Z^{\e}_s-\Z^{\e}_{s(\Delta)}\|_{\H}\d s\nonumber\\&\leq 4\left(\int_0^t(\|\F(\X_{s(\Delta)}^{\e,\delta,h^{\e}},\widehat\Y_{s}^{\e,\delta})\|_{\H}^2+\|\bar \F(\X_{s(\Delta)}^{\e,\delta,h^{\e}})\|_{\H}^2)\d s\right)^{1/2}\nonumber\\&\quad\times\left(\int_0^t(\|\X^{\e,\delta,h^{\e}}_{s}-\X^{\e,\delta,h^{\e}}_{s(\Delta)}\|_{\H}^2+\|\bar{\X}^{h}_s-\bar{\X}^{h}_{s(\Delta)}\|_{\H}^2)\d s\right)^{1/2}\nonumber\\&\leq C\left(\int_0^t(1+\|\X_{s(\Delta)}^{\e,\delta,h^{\e}}\|_{\H}^2+\|\widehat\Y_{s}^{\e,\delta}\|_{\H}^2)\d s\right)^{1/2}\left(\int_0^t(\|\X^{\e,\delta,h^{\e}}_{s}-\X^{\e,\delta,h^{\e}}_{s(\Delta)}\|_{\H}^2+\|\bar{\X}^{h}_s-\bar{\X}^{h}_{s(\Delta)}\|_{\H}^2)\d s\right)^{1/2}.
	\end{align}
	We estimate the term $2\int_0^t([\sigma_1(\X^{\e,\delta,h^{\e}}_s)\Q_1^{1/2}h^{\e}_s-\sigma_1(\bar\X^{h}_s)\Q_1^{1/2}h_s],\Z^{\e}_s)\d s$  from the right hand side of the inequality \eqref{3.92} using the Cauchy-Schwarz inequality  and Young's inequality, and the Assumption \ref{ass3.6} (A1) as 
	\begin{align}\label{456}
	&2\int_0^t([\sigma_1(\X^{\e,\delta,h^{\e}}_s)\Q_1^{1/2}h^{\e}_s-\sigma_1(\bar\X^{h}_s)\Q_1^{1/2}h_s],\Z^{\e}_s)\d s \nonumber\\& \leq 2\int_0^{t}|([\sigma_1(\X^{\e,\delta,h^{\e}}_s)-\sigma_1(\bar\X^{h}_s)]\Q_1^{1/2}h^{\e}_s,\Z^{\e}_s)|\d s   +2\int_0^{t}|(\sigma_1(\bar\X^{h}_s)\Q_1^{1/2}(h^{\e}_s-h_s),\Z^{\e}_s)|\d s \nonumber\\&\leq 2\int_0^{t}\|[\sigma_1(\X^{\e,\delta,h^{\e}}_s)-\sigma_1(\bar\X^{h}_s)]\Q_1^{1/2}\|_{\mathcal{L}_2}\|h^{\e}_s\|_{\H}\|\Z^{\e}_s\|_{\H}\d s \nonumber\\&\quad+2 \int_0^{t}\|\sigma_1(\bar\X^{h}_s)\Q_1^{1/2}(h^{\e}_s-h_s)\|_{\H}\|\Z^{\e}_s\|_{\H}\d s \nonumber\\&\leq C\int_0^{t}\left(1+\|h^{\e}_s\|_{\H}^2\right)\|\Z^{\e}_s\|_{\H}^2\d s +\int_0^{t}\|\Z^{\e}_s\|_{\H}^2\d s+\int_0^{t}\|\sigma_1(\bar\X^{h}_s)\Q_1^{1/2}(h^{\e}_s-h_s)\|_{\H}^2\d s.
	\end{align}
	Making use of the Assumption \ref{ass3.6} (A1),  it can be easily seen that 
	\begin{align}\label{457}
	\int_0^{t}\|\sigma_1(\X_{s}^{\e,\delta,h^{\e}})\Q_1^{1/2}\|_{\mathcal{L}_2}^2\d s&\leq 2\int_0^{t} \|[\sigma_1(\X_{s}^{\e,\delta,h^{\e}})-\sigma_1(\bar\X_{s}^h)]\Q_1^{1/2}\|_{\mathcal{L}_2}^2\d s+2 \int_0^{t} \|\sigma_1(\bar\X_{s}^h)\Q_1^{1/2}\|_{\mathcal{L}_2}^2\d s \nonumber\\&\leq \frac{C}{\e}\int_0^{t}\|\Z^{\e}_s\|_{\H}^2\d s+C\e\int_0^{t}\left(1+\|\bar\X_{s}^h\|_{\H}^2\right)\d s.
	\end{align}
	Combining \eqref{3.93}-\eqref{457}, and then substituting it in \eqref{3.92}, we obtain 
	\begin{align}\label{3.99} 
	&	\|\Z^{\e}_{t}\|_{\H}^2+\mu\int_0^t\|\Z^{\e}_s\|_{\V}^2\d s+\frac{\beta}{2^{r-2}}\int_0^t\|\Z^{\e}_s\|_{\wi\L^{r+1}}^{r+1}\d s\nonumber\\&\leq \frac{2}{\mu}\int_0^t\|\X_{s}^{\e,\delta,h^{\e}}\|_{\V}^2\|\Z^{\e}_s\|_{\H}^2\d s+C_{\mu,\alpha,\lambda_1,L_{\G},L_{\sigma_2}}\int_0^t\|\Z^{\e}_s\|_{\H}^2\d s\nonumber\\&\quad+C_{\mu,\alpha,\lambda_1,L_{\G},L_{\sigma_2}}\int_0^t\|\X_{s}^{\e,\delta,h^{\e}}-\X_{s(\Delta)}^{\e,\delta,h^{\e}}\|_{\H}^2\d s+C_{\mu,\alpha,\lambda_1,L_{\G},L_{\sigma_2}}\int_0^t\|\Y_{s}^{\e,\delta,h^{\e}}-\widehat\Y_{s}^{\e,\delta}\|_{\H}^2\d s\nonumber\\&\quad+C\int_0^{t}\|h^{\e}_s\|_{\H}^2\|\Z^{\e}_s\|_{\H}^2\d s +\int_0^{t}\|\sigma_1(\bar\X^{h}_s)\Q_1^{1/2}(h^{\e}_s-h_s)\|_{\H}^2\d s+C\e^2\int_0^{t}\left(1+\|\bar\X_{s}^h\|_{\H}^2\right)\d s\nonumber\\&\quad+C\left(\int_0^t(1+\|\X_{s(\Delta)}^{\e,\delta,h^{\e}}\|_{\H}^2+\|\widehat\Y_{s}^{\e,\delta})\|_{\H}^2)\d s\right)^{1/2}\left(\int_0^t(\|\X^{\e,\delta,h^{\e}}_{s}-\X^{\e,\delta,h^{\e}}_{s(\Delta)}\|_{\H}^2+\|\bar{\X}^{h}_s-\bar{\X}^{h}_{s(\Delta)}\|_{\H}^2)\d s\right)^{1/2} \nonumber\\&\quad +2\int_0^t(\F(\X_{s(\Delta)}^{\e,\delta,h^{\e}},\widehat\Y_{s}^{\e,\delta})-\bar \F(\X_{s(\Delta)}^{\e,\delta,h^{\e}}),\Z^{\e}_{s(\Delta)})\d s\nonumber\\&\quad+2\sqrt{\e}\int_0^t([\sigma_1(\X_{s}^{\e,\delta,h^{\e}})-\sigma_1(\bar{\X}^{h}_s)]\Q_1^{1/2}\d\W_s,\Z^{\e}_s), \ \mathbb{P}\text{-a.s.,}
	\end{align}
	for all $t\in[0,T]$. An application of the Gronwall inequality in \eqref{3.99} gives 
	\begin{align}\label{3100} 
&	\sup_{t\in[0,T\wedge\wi\tau_R^{\e}]}	\|\Z^{\e}_{t}\|_{\H}^2+\mu\int_0^{T\wedge\wi\tau_R^{\e}}\|\Z^{\e}_t\|_{\V}^2\d t+2\alpha\int_0^{T\wedge\wi\tau_R^{\e}}\|\Z^{\e}_t\|_{\H}^2\d t+\frac{\beta}{2^{r-2}}\int_0^{T\wedge\wi\tau_R^{\e}}\|\Z^{\e}_t\|_{\wi\L^{r+1}}^{r+1}\d t\nonumber\\&\leq C_{\mu,\alpha,\lambda_1,L_{\G},L_{\sigma_2},T}\Bigg\{\int_0^{T\wedge\wi\tau_R^{\e}}\|\X_{s}^{\e,\delta,h^{\e}}-\X_{s(\Delta)}^{\e,\delta,h^{\e}}\|_{\H}^2\d s+\int_0^{T\wedge\wi\tau_R^{\e}}\|\Y_{s}^{\e,\delta,h^{\e}}-\widehat\Y_{s}^{\e,\delta}\|_{\H}^2\d s\nonumber\\&\quad +\int_0^{T\wedge\wi\tau_R^{\e}}\|\sigma_1(\bar\X^{h}_t)\Q_1^{1/2}(h^{\e}_t-h_t)\|_{\H}^2\d t+\e^2\int_0^{T\wedge\wi\tau_R^{\e}}\left(1+\|\bar\X_{t}^h\|_{\H}^2\right)\d t\nonumber\\&\quad+\left(\int_0^{T\wedge\wi\tau_R^{\e}}(1+\|\X_{s(\Delta)}^{\e,\delta,h^{\e}}\|_{\H}^2+\|\widehat\Y_{s}^{\e,\delta}\|_{\H}^2)\d s\right)^{1/2}\nonumber\\&\qquad\times\left(\int_0^{T\wedge\wi\tau_R^{\e}}(\|\X^{\e,\delta,h^{\e}}_{s}-\X^{\e,\delta,h^{\e}}_{s(\Delta)}\|_{\H}^2+\|\bar{\X}^{h}_s-\bar{\X}^{h}_{s(\Delta)}\|_{\H}^2)\d s\right)^{1/2}\nonumber\\&\quad+	\sup_{t\in[0,T\wedge\wi\tau_R^{\e}]}\left|\int_0^t(\F(\X_{s(\Delta)}^{\e,\delta,h^{\e}},\widehat\Y_{s}^{\e,\delta})-\bar \F(\X_{s(\Delta)}^{\e,\delta,h^{\e}}),\Z^{\e}_{s(\Delta)})\d s\right|\nonumber\\&\quad+\sqrt{\e}	\sup_{t\in[0,T\wedge\wi\tau_R^{\e}]}\left|\int_0^t([\sigma_1(\X_{s}^{\e,\delta,h^{\e}})-\sigma_1(\bar{\X}^{h}_s)]\Q_1^{1/2}\d\W_s,\Z^{\e}_s)\right|\Bigg\}\nonumber\\&\quad\times \exp\left(C\int_0^{T\wedge\wi\tau_R^{\e}}\|h^{\e}_t\|_{\H}^2\d t\right)\exp\left(\frac{2}{\mu}\int_0^{T\wedge\wi\tau_R^{\e}}\|\X_{t}^{\e,\delta,h^{\e}}\|_{\V}^2\d t\right)\nonumber\\&\leq C_{\mu,\alpha,\lambda_1,L_{\G},L_{\sigma_2},M,T,R}\bigg\{\int_0^{T\wedge\wi\tau_R^{\e}}\|\X_{s}^{\e,\delta,h^{\e}}-\X_{s(\Delta)}^{\e,\delta,h^{\e}}\|_{\H}^2\d s+\int_0^{T}\|\Y_{s}^{\e,\delta,h^{\e}}-\widehat\Y_{s}^{\e,\delta}\|_{\H}^2\d s\nonumber\\&\quad +\int_0^{T}\|\sigma_1(\bar\X^{h}_t)\Q_1^{1/2}(h^{\e}_t-h_t)\|_{\H}^2\d t+\e^2\int_0^{T}\left(1+\|\bar\X_{t}^h\|_{\H}^2\right)\d t\nonumber\\&\quad+\left(\int_0^{T}(1+\|\X_{s(\Delta)}^{\e,\delta,h^{\e}}\|_{\H}^2+\|\widehat\Y_{s}^{\e,\delta}\|_{\H}^2)\d s\right)^{1/2}\nonumber\\&\qquad\times\left(\int_0^{T\wedge\wi\tau_R^{\e}}\|\X^{\e,\delta,h^{\e}}_{s}-\X^{\e,\delta,h^{\e}}_{s(\Delta)}\|_{\H}^2\d s+\int_0^T\|\bar{\X}^{h}_s-\bar{\X}^{h}_{s(\Delta)}\|_{\H}^2\d s\right)^{1/2}\nonumber\\&\quad+	\sup_{t\in[0,T\wedge\wi\tau_R^{\e}]}\left|\int_0^t(\F(\X_{s(\Delta)}^{\e,\delta,h^{\e}},\widehat\Y_{s}^{\e,\delta})-\bar \F(\X_{s(\Delta)}^{\e,\delta,h^{\e}}),\Z^{\e}_{s(\Delta)})\d s\right|\nonumber\\&\quad+\sqrt{\e}	\sup_{t\in[0,T\wedge\wi\tau_R^{\e}]}\left|\int_0^t([\sigma_1(\X_{s}^{\e,\delta,h^{\e}})-\sigma_1(\bar{\X}^{h}_s)]\Q_1^{1/2}\d\W_s,\Z^{\e}_s)\right|\bigg\}, \ \mathbb{P}\text{-a.s.},
	\end{align}
since $h^{\e}\in\mathcal{A}_M$, 	where we used the definition of stopping time given in \eqref{stop} also. Taking expectation on both sides of \eqref{3100} and then using Theorem \ref{thm5.9}, Lemmas \ref{lem3.13}-\ref{lem3.8} (see \eqref{5.5y}, \eqref{3.87}, \eqref{3.24} and \eqref{3.42}), we obtain 
	\begin{align}\label{3101}
&	\E\left[	\sup_{t\in[0,T\wedge\wi\tau_R^{\e}]}	\|\Z^{\e}_{t}\|_{\H}^2+\mu\int_0^{T\wedge\wi\tau_R^{\e}}\|\Z^{\e}_t\|_{\V}^2\d t+\frac{\beta}{2^{r-2}}\int_0^{T\wedge\wi\tau_R^{\e}}\|\Z^{\e}_t\|_{\wi\L^{r+1}}^{r+1}\d t\right]\nonumber\\&\leq C_{\mu,\alpha,\beta,\lambda_1,L_{\G},L_{\sigma_2},M,T,R}\Bigg\{\left(1+\|\x\|_{\H}^3+\|\y\|_{\H}^3\right)\left[\left(\frac{\delta}{\e}\right)+\Delta^{1/4}\right]\nonumber\\&\quad+\E\left[\int_0^{T}\|\sigma_1(\bar\X^{h}_t)\Q_1^{1/2}(h^{\e}_t-h_t)\|_{\H}^2\d t\right]+\e^2T\left(1+\sup_{t\in[0,T]}\|\bar\X_{t}^h\|_{\H}^2\right)\nonumber\\&\quad+\E\left[\sup_{t\in[0,T\wedge\wi\tau_R^{\e}]}\left|\int_0^t(\F(\X_{s(\Delta)}^{\e,\delta,h^{\e}},\widehat\Y_{s}^{\e,\delta})-\bar \F(\X_{s(\Delta)}^{\e,\delta,h^{\e}}),\Z^{\e}_{s(\Delta)})\d s\right|\right]\nonumber\\&\quad+\E\left[\sup_{t\in[0,T\wedge\wi\tau_R^{\e}]}\left|\int_0^t([\sigma_1(\X_{s}^{\e,\delta,h^{\e}})-\sigma_1(\bar{\X}^{h}_s)]\Q_1^{1/2}\d\W_s,\Z^{\e}_s)\right|\right]\Bigg\}. 
	\end{align}
	Using the Burkholder-Davis-Gundy inequality and Assumption \ref{ass3.6} (A1), we estimate the final term from the right hand side of the inequality \eqref{3101} as 
	\begin{align}\label{3102}
&C_{\mu,\alpha,\beta,\lambda_1,L_{\G},L_{\sigma_2},M,T,R}\E\left[\sup_{t\in[0,T\wedge\wi\tau_R^{\e}]}\left|\int_0^t([\sigma_1(\X_{s}^{\e,\delta,h^{\e}})-\sigma_1(\bar{\X}^{h}_s)]\Q_1^{1/2}\d\W_s,\Z^{\e}_s)\right|\right]\nonumber\\&\leq C_{\mu,\alpha,\beta,\lambda_1,L_{\G},L_{\sigma_2},M,T,R} \E\left[\int_0^{T\wedge\wi\tau_R^{\e}}\|[\sigma_1(\X_{s}^{\e,\delta,h^{\e}})-\sigma_1(\bar{\X}^{h}_s)]\Q_1^{1/2}\|_{\mathcal{L}_2}^2\|\Z^{\e}_s\|_{\H}^2\d s\right]^{1/2}\nonumber\\&\leq C_{\mu,\alpha,\beta,\lambda_1,L_{\G},L_{\sigma_2},M,T,R}\E\left[\sup_{s\in[0,T\wedge\wi\tau_R^{\e}]}\|\Z^{\e}_s\|_{\H}\left(\int_0^{T\wedge\wi\tau_R^{\e}}\|[\sigma_1(\X_{s}^{\e,\delta,h^{\e}})-\sigma_1(\bar{\X}^{h}_s)]\Q_1^{1/2}\|_{\mathcal{L}_2}^2\d s\right)^{1/2}\right]\nonumber\\&\leq\frac{1}{2}\E\left[\sup_{t\in[0,T\wedge\wi\tau_R^{\e}]}\|\Z^{\e}_t\|_{\H}^2\right]+C_{\mu,\alpha,\beta,\lambda_1,L_{\G},L_{\sigma_2},M,T,R}\E\left[\int_0^{T\wedge\wi\tau_R^{\e}}\|\Z^{\e}_t\|_{\H}^2\d t\right]. 
	\end{align}
	Using \eqref{3102} in \eqref{3101}, we get 
	\begin{align}\label{3103} 
	&	\E\left[	\sup_{t\in[0,T\wedge\wi\tau_R^{\e}]}	\|\Z^{\e}_{t}\|_{\H}^2+\mu\int_0^{T\wedge\wi\tau_R^{\e}}\|\Z^{\e}_t\|_{\V}^2\d t+\frac{\beta}{2^{r-2}}\int_0^{T\wedge\wi\tau_R^{\e}}\|\Z^{\e}_t\|_{\wi\L^{r+1}}^{r+1}\d t\right]\nonumber\\&\leq C_{\mu,\alpha,\beta,\lambda_1,L_{\G},L_{\sigma_2},M,T,R}\left(1+\|\x\|_{\H}^3+\|\y\|_{\H}^3\right)\left[\left(\frac{\delta}{\e}\right)+\Delta^{1/4}\right]\nonumber\\&\quad+C_{\mu,\alpha,\beta,\lambda_1,L_{\G},L_{\sigma_2},M,T,R}\E\left[\int_0^{T\wedge\wi\tau_R^{\e}}\|\Z^{\e}_s\|_{\H}^2\d s\right]\nonumber\\&\quad+\E\left[\int_0^{T}\|\sigma_1(\bar\X^{h}_t)\Q_1^{1/2}(h^{\e}_t-h_t)\|_{\H}^2\d t\right]+\e^2T\left(1+\sup_{t\in[0,T]}\|\bar\X_{t}^h\|_{\H}^2\right)\nonumber\\&\quad+C_{\mu,\alpha,\beta,\lambda_1,L_{\G},L_{\sigma_2},M,T,R}\E\left[\sup_{t\in[0,T\wedge\wi\tau_R^{\e}]}\left|\int_0^t(\F(\X_{s(\Delta)}^{\e,\delta,h^{\e}},\widehat\Y_{s}^{\e,\delta})-\bar \F(\X_{s(\Delta)}^{\e,\delta,h^{\e}}),\Z^{\e}_{s(\Delta)})\d s\right|\right].
	\end{align}
	An application of Gronwall's inequality in \eqref{3103} yields 
	\begin{align}\label{3104}
	&	\E\left[	\sup_{t\in[0,T\wedge\wi\tau_R^{\e}]}	\|\Z^{\e}_{t}\|_{\H}^2+\mu\int_0^{T\wedge\wi\tau_R^{\e}}\|\Z^{\e}_t\|_{\V}^2\d t+\frac{\beta}{2^{r-2}}\int_0^{T\wedge\wi\tau_R^{\e}}\|\Z^{\e}_t\|_{\wi\L^{r+1}}^{r+1}\d t\right]\nonumber\\&\leq C_{\mu,\alpha,\beta,\lambda_1,L_{\G},L_{\sigma_2},M,T,R}\left(1+\|\x\|_{\H}^3+\|\y\|_{\H}^3\right)\left[\left(\frac{\delta}{\e}\right)+\Delta^{1/4}\right]\nonumber\\&\quad+\E\left[\int_0^{T}\|\sigma_1(\bar\X^{h}_t)\Q_1^{1/2}(h^{\e}_t-h_t)\|_{\H}^2\d t\right]+\e^2T\left(1+\sup_{t\in[0,T]}\|\bar\X_{t}^h\|_{\H}^2\right)\nonumber\\&\quad+C_{\mu,\alpha,\beta,\lambda_1,L_{\G},L_{\sigma_2},M,T,R}\E\left[\sup_{t\in[0,T]}|I(t)|\right], 
	\end{align}
	where $I(t)=\int_0^t(\F(\X_{s(\Delta)}^{\e,\delta,h^{\e}},\widehat\Y_{s}^{\e,\delta})-\bar \F(\X_{s(\Delta)}^{\e,\delta,h^{\e}}),\Z^{\e}_{s(\Delta)})\d s$. 
	Let us now estimate the final term from the right hand side of the inequality \eqref{3104}.  We follow similar arguments given in Lemma 3.8, \cite{SLXS} and Lemma 4.6, \cite{MTM11} to obtain the required result. For the sake of completeness, we provide a proof here. Note that 
	\begin{align}\label{3105}
	|I(t)|&= \left|\sum_{k=0}^{[t/\Delta]-1}\int_{k\Delta}^{(k+1)\Delta}(\F(\X_{k\Delta}^{\e,\delta,h^{\e}},\widehat\Y_{s}^{\e,\delta})-\bar \F(\X_{k\Delta}^{\e,\delta,h^{\e}}),\X^{\e,\delta,h^{\e}}_{k\Delta}-\bar{\X}^{h}_{s(\Delta)})\d s\right.\nonumber\\&\quad+\left.\int_{t(\Delta)}^t(\F(\X_{s(\Delta)}^{\e,\delta,h^{\e}},\widehat\Y_{s}^{\e,\delta})-\bar \F(\X_{s(\Delta)}^{\e,\delta,h^{\e}}),\X^{\e,\delta,h^{\e}}_{s(\Delta)}-\bar{\X}^{h}_{s(\Delta)})\d s\right|=:|I_1(t)+I_2(t)|. 
	\end{align}
	Using the Assumption \ref{ass3.6} (A1), we estimate $\E\left[\sup\limits_{t\in[0,T]}|I_2(t)|\right]$ as 
	\begin{align}\label{4129}
&	\E\left[\sup\limits_{t\in[0,T]}|I_2(t)|\right]\nonumber\\&\leq\E\left[\sup\limits_{t\in[0,T]}\int_{t(\Delta)}^t\|\F(\X_{s(\Delta)}^{\e,\delta,h^{\e}},\widehat\Y_{s}^{\e,\delta})-\bar \F(\X_{s(\Delta)}^{\e,\delta,h^{\e}})\|_{\H}\|\X^{\e,\delta,h^{\e}}_{s(\Delta)}-\bar{\X}^{h}_{s(\Delta)}\|_{\H}\d s\right]\nonumber\\&\leq C\left[\E\left(\sup_{t\in[0,T]}\|\X_{t}^{\e,\delta,h^{\e}}-\bar{\X}_t^{h}\|_{\H}^2\right)\right]^{1/2}\left[\E\left(\sup_{t\in[0,T]}\left|\int_{t(\Delta)}^t\left(1+\|\X^{\e,\delta,h^{\e}}_{s(\Delta)}\|_{\H}+\|\widehat\Y_{s}^{\e,\delta}\|_{\H}\right)\d s\right|^2\right)\right]^{1/2}\nonumber\\&\leq C\Delta^{1/2} \left[\E\left(\sup_{t\in[0,T]}(\|\X_{t}^{\e,\delta,h^{\e}}\|_{\H}^2+\|\bar{\X}_t^{h}\|_{\H}^2)\right)\right]^{1/2}\left[\E\left(\int_0^T\left(1+\|\X^{\e,\delta,h^{\e}}_{s(\Delta)}\|_{\H}^2+\|\widehat\Y_{s}^{\e,\delta}\|_{\H}^2\right)\d s\right)\right]^{1/2}\nonumber\\&\leq C_{\mu,\alpha,\lambda_1,L_{\G},L_{\sigma_2},M,T}(1+\|\x\|_{\H}^2+\|\y\|_{\H}^2)\Delta^{1/2}, 
	\end{align} 
where we used \eqref{5.5y} and \eqref{5.5z}. 	Next, we estimate the term $\E\left[\sup\limits_{t\in[0,T]}|I_1(t)|\right]$ as 
	\begin{align}\label{3107}
&	\E\left[\sup\limits_{t\in[0,T]}|I_1(t)|\right]\nonumber\\&\leq \E\left[\sum_{k=0}^{[T/\Delta]-1}\left|\int_{k\Delta}^{(k+1)\Delta}(\F(\X_{k\Delta}^{\e,\delta,h^{\e}},\widehat\Y_{s}^{\e,\delta})-\bar \F(\X_{k\Delta}^{\e,\delta,h^{\e}}),\X^{\e,\delta,h^{\e}}_{k\Delta}-\bar{\X}^{h}_{k\Delta})\d s\right|\right]\nonumber\\&\leq\left[\frac{T}{\Delta}\right]\max_{0\leq k\leq [T/\Delta]-1}\E\left[\left|\int_{k\Delta}^{(k+1)\Delta}(\F(\X_{k\Delta}^{\e,\delta,h^{\e}},\widehat\Y_{s}^{\e,\delta})-\bar \F(\X_{k\Delta}^{\e,\delta,h^{\e}}),\X^{\e,\delta,h^{\e}}_{k\Delta}-\bar{\X}^{h}_{k\Delta})\d s\right|\right]\nonumber\\&\leq\frac{C_T}{\Delta}\max_{0\leq k\leq [T/\Delta]-1}\left[\E\left(\|\X^{\e,\delta,h^{\e}}_{k\Delta}-\bar{\X}^{h}_{k\Delta}\|_{\H}^2\right)\right]^{1/2}\left[\E\left\|\int_{k\Delta}^{(k+1)\Delta}\F(\X_{k\Delta}^{\e,\delta,h^{\e}},\widehat\Y_{s}^{\e,\delta})-\bar \F(\X_{k\Delta}^{\e,\delta,h^{\e}})\d s\right\|_{\H}^2\right]^{1/2}\nonumber\\&\leq \frac{C_T\delta}{\Delta}\max_{0\leq k\leq [T/\Delta]-1}\left[\E\left(\|\X^{\e,\delta,h^{\e}}_{k\Delta}-\bar{\X}^{h}_{k\Delta}\|_{\H}^2\right)\right]^{1/2}\left[\E\left\|\int_{0}^{\frac{\Delta}{\delta}}\F(\X_{k\Delta}^{\e,\delta,h^{\e}},\widehat\Y_{s\delta+k\Delta}^{\e,\delta})-\bar \F(\X_{k\Delta}^{\e,\delta,h^{\e}})\d s\right\|_{\H}^2\right]^{1/2}\nonumber\\&\leq\frac{C_{\mu,\alpha,\lambda_1,L_{\G},M,T}(1+\|\x\|_{\H}+\|\y\|_{\H})\delta}{\Delta}\max_{0\leq k\leq [T/\Delta]-1}\left[\int_0^{\frac{\Delta}{\delta}}\int_r^{\frac{\Delta}{\delta}}\Phi_k(s,r)\d s\d r\right]^{1/2},
	\end{align}
	where for any $0\leq r\leq s\leq \frac{\Delta}{\delta}$, 
	\begin{align}\label{3108}
	\Phi_k(s,r):=\E\left[(\F(\X_{k\Delta}^{\e,\delta,h^{\e}},\widehat\Y_{s\delta+k\Delta}^{\e,\delta})-\bar \F(\X_{k\Delta}^{\e,\delta,h^{\e}}),\F(\X_{k\Delta}^{\e,\delta,h^{\e}},\widehat\Y_{r\delta+k\Delta}^{\e,\delta})-\bar \F(\X_{k\Delta}^{\e,\delta,h^{\e}}))\right].
	\end{align}
	Now, for any $\e>0$, $s>0$ and $\mathscr{F}_s$-measurable $\H$-valued random variables $\X$ and $\Y$, let $\{\Y_t^{\e,s,\X,\Y}\}_{t\geq s}$ be the unique strong solution of the following It\^o stochastic differential equation: 
	\begin{equation}
	\left\{
	\begin{aligned}
	\d\wi\Y_t&=-\frac{1}{\delta}[\mu\A\wi\Y_t+\alpha\wi\Y_t+\beta\mathcal{C}(\wi\Y_t)-\G(\X,\wi\Y_t)]\d t+\frac{1}{\sqrt{\delta}}\sigma_2(\X,\wi\Y_t)\Q_2^{1/2}\d\W_t,\\
	\wi\Y_s&=\Y. 
	\end{aligned}
	\right.
	\end{equation}
	Then, from the construction of the  process $\widehat\Y_{t}^{\e,\delta}$ (see \eqref{3.37}), for any $t\in[k\Delta,(k+1)\Delta]$ with $k\in\mathbb{N}$, we have 
	\begin{align}
\widehat\Y_{t}^{\e,\delta}=\widetilde\Y_t^{\e,k\Delta,\X_{k\Delta}^{\e,\delta,h^{\e}},\widehat\Y_{k\Delta}^{\e,\delta}}, \ \mathbb{P}\text{-a.s.}
	\end{align}
	Using this fact in \eqref{3108}, we infer that 
	\begin{align}
	\Phi_k(s,r)&=\E\left[(\F(\X_{k\Delta}^{\e,\delta,h^{\e}},\widetilde\Y_{s\delta+k\Delta}^{\e,k\Delta,\X_{k\Delta}^{\e,\delta,h^{\e}},\widehat\Y_{k\Delta}^{\e,\delta}})-\bar \F(\X_{k\Delta}^{\e,\delta,h^{\e}}),\F(\X_{k\Delta}^{\e,\delta,h^{\e}},\widetilde\Y_{r\delta+k\Delta}^{\e,k\Delta,\X_{k\Delta}^{\e,\delta,h^{\e}},\widehat\Y_{k\Delta}^{\e,\delta}})-\bar \F(\X_{k\Delta}^{\e,\delta,h^{\e}}))\right]\nonumber\\&=\int_{\Omega}\E\left[(\F(\X_{k\Delta}^{\e,\delta,h^{\e}},\widetilde\Y_{s\delta+k\Delta}^{\e,k\Delta,\X_{k\Delta}^{\e,\delta,h^{\e}},\widehat\Y_{k\Delta}^{\e,\delta}})-\bar \F(\X_{k\Delta}^{\e,\delta,h^{\e}}),\right.\nonumber\\&\qquad\qquad\left.\F(\X_{k\Delta}^{\e,\delta,h^{\e}},\widetilde\Y_{r\delta+k\Delta}^{\e,k\Delta,\X_{k\Delta}^{\e,\delta,h^{\e}},\widehat\Y_{k\Delta}^{\e,\delta}})-\bar \F(\X_{k\Delta}^{\e,\delta,h^{\e}}))\Big|{\mathscr{F}_{k\Delta}}\right]\d\mathbb{P}(\omega)\nonumber\\&=\int_{\Omega}\E\left[(\F(\X_{k\Delta}^{\e,\delta,h^{\e}}(\omega),\widetilde\Y_{s\delta+k\Delta}^{\e,k\Delta,\X_{k\Delta}^{\e,\delta,h^{\e}}(\omega),\widehat\Y_{k\Delta}^{\e,\delta}(\omega)})-\bar \F(\X_{k\Delta}^{\e,\delta,h^{\e}}(\omega)),\right.\nonumber\\&\qquad\qquad\left.\F(\X_{k\Delta}^{\e,\delta,h^{\e}}(\omega),\widetilde\Y_{r\delta+k\Delta}^{\e,k\Delta,\X_{k\Delta}^{\e,\delta,h^{\e}}(\omega),\widehat\Y_{k\Delta}^{\e,\delta}(\omega)})-\bar \F(\X_{k\Delta}^{\e,\delta,h^{\e}}(\omega)))\right]\mathbb{P}(\d\omega),
	\end{align}
	where in the final step we used the fact the processes $\X_{k\Delta}^{\e,\delta,h^{\e}}$ and $\widehat{\Y}_{k\Delta}^{\e,\delta}$ are $\mathscr{F}_{k\Delta}$-measurable, whereas the process $\{\widetilde\Y^{\e,k\Delta,\x,\y}_{s\delta+k\Delta}\}_{s\geq 0}$ is independent of $\mathscr{F}_{k\Delta}$, for any fixed $(\x,\y)\in\H\times\H$. Moreover, from the definition of the process  $\widetilde\Y^{\e,k\Delta,\x,\y}_{s\delta+k\Delta}$, we obtain 
	\begin{align}\label{3112}
	\widetilde\Y^{\e,k\Delta,\x,\y}_{s\delta+k\Delta}&= \y-\frac{\mu}{\delta}\int_{k\Delta}^{s\delta+k\Delta}\A\widetilde{\Y}^{\e,k\Delta,\x,\y}_r\d r-\frac{\alpha}{\delta}\int_{k\Delta}^{s\delta+k\Delta}\widetilde{\Y}^{\e,k\Delta,\x,\y}_r\d r-\frac{\beta}{\delta}\int_{k\Delta}^{s\delta+k\Delta}\mathcal{C}(\widetilde{\Y}^{\e,k\Delta,\x,\y}_r)\d r\nonumber\\&\quad+\frac{1}{\delta}\int_{k\Delta}^{s\delta+k\Delta}\G(\x,\widetilde{\Y}^{\e,k\Delta,\x,\y}_r)\d r+\frac{1}{\sqrt{\delta}}\int_{k\Delta}^{s\delta+k\Delta}\sigma_2(\x,\widetilde{\Y}^{\e,k\Delta,\x,\y}_r)\Q_2^{1/2}\d\W_r\nonumber\\&=\y-\frac{\mu}{\delta}\int_{0}^{s\delta}\A\widetilde{\Y}^{\e,k\Delta,\x,\y}_{r+k\Delta}\d r-\frac{\alpha}{\delta}\int_{0}^{s\delta}\widetilde{\Y}^{\e,k\Delta,\x,\y}_{r+k\Delta}\d r-\frac{\beta}{\delta}\int_{0}^{s\delta}\mathcal{C}(\widetilde{\Y}^{\e,k\Delta,\x,\y}_{r+k\Delta})\d r\nonumber\\&\quad+\frac{1}{\delta}\int_{0}^{s\delta}\G(\x,\widetilde{\Y}^{\e,k\Delta,\x,\y}_{r+k\Delta})\d r+\frac{1}{\sqrt{\delta}}\int_{0}^{s\delta}\sigma_2(\x,\widetilde{\Y}^{\e,k\Delta,\x,\y}_{r+k\Delta})\Q_2^{1/2}\d\W^{k\Delta}_r\nonumber\\&=\y-\mu\int_{0}^{s}\A\widetilde{\Y}^{\e,k\Delta,\x,\y}_{r\delta+k\Delta}\d r-\alpha\int_{0}^{s}\widetilde{\Y}^{\e,k\Delta,\x,\y}_{r\delta+k\Delta}\d r-\beta\int_{0}^{s}\mathcal{C}(\widetilde{\Y}^{\e,k\Delta,\x,\y}_{r\delta+k\Delta})\d r\nonumber\\&\quad+\int_{0}^{s}\G(\x,\widetilde{\Y}^{\e,k\Delta,\x,\y}_{r\delta+k\Delta})\d r+\int_{0}^{s}\sigma_2(\x,\widetilde{\Y}^{\e,k\Delta,\x,\y}_{r\delta+k\Delta})\Q_2^{1/2}\d\widehat\W^{k\Delta}_r,
	\end{align}
	where $\{\W_r^{k\Delta}:=\W_{r+k\Delta}-\W_{k\Delta}\}_{r\geq 0}$ is the shift version of $\W_r$ and $\{\widehat\W^{k\Delta}_r:=\frac{1}{\sqrt{\delta}}\W_{r\delta}^{k\Delta}\}_{r\geq 0}$. Using the uniqueness of the strong solutions of \eqref{3.74} and \eqref{3112}, we infer that 
	\begin{align}
\mathscr{L}\left(\{\widetilde\Y^{\e,k\Delta,\x,\y}_{s\delta+k\Delta}\}_{0\leq s\leq\frac{\Delta}{\delta}}\right)=\mathscr{L}\left(\{\Y^{\x,\y}_{s}\}_{0\leq s\leq\frac{\Delta}{\delta}}\right),
	\end{align}
	where $\mathscr{L}(\cdot)$ denotes the law of the distribution. Using Markov's property, Proposition \ref{prop3.12}, the estimates \eqref{5.5z} and \eqref{341}, we estimate $\Phi_{k}(\cdot,\cdot)$ as 
		\begin{align}\label{3114}
	\Phi_k(s,r)&=\int_{\Omega}\wi\E\left[(\F(\X_{k\Delta}^{\e,\delta,h^{\e}}(\omega),\Y_{s}^{\X_{k\Delta}^{\e,\delta,h^{\e}}(\omega),\widehat\Y_{k\Delta}^{\e,\delta}(\omega)})-\bar \F(\X_{k\Delta}^{\e,\delta,h^{\e}}(\omega)),\right.\nonumber\\&\qquad\left.\F(\X_{k\Delta}^{\e,\delta,h^{\e}}(\omega),\Y_{r}^{\X_{k\Delta}^{\e,\delta,h^{\e}}(\omega),\widehat\Y_{k\Delta}^{\e,\delta}(\omega)})-\bar \F(\X_{k\Delta}^{\e,\delta,h^{\e}}(\omega)))\right]\mathbb{P}(\d\omega)\nonumber\\&=\int_{\Omega}\int_{\wi{\Omega}}\Big(\wi\E\Big[\F(\X_{k\Delta}^{\e,\delta,h^{\e}}(\omega),\Y_{s-r}^{\X_{k\Delta}^{\e,\delta,h^{\e}}(\omega),\z})-\bar \F(\X_{k\Delta}^{\e,\delta,h^{\e}}(\omega))\Big],\nonumber\\&\qquad \F(\X_{k\Delta}^{\e,\delta,h^{\e}}(\omega),\z)-\bar \F(\X_{k\Delta}^{\e,\delta,h^{\e}}(\omega))\Big)\Big|_{\z=\Y_{r}^{\X_{k\Delta}^{\e,\delta,h^{\e}}(\omega),\widehat\Y_{k\Delta}^{\e,\delta}(\omega)}(\wi\omega)}\wi{\mathbb{P}}(\d\wi{\omega})\mathbb{P}(\d\omega)\nonumber\\&\leq \int_{\Omega}\int_{\wi{\Omega}}\left\|\wi\E\Big[\F(\X_{k\Delta}^{\e,\delta,h^{\e}}(\omega),\Y_{s-r}^{\X_{k\Delta}^{\e,\delta,h^{\e}}(\omega),\z})-\bar \F(\X_{k\Delta}^{\e,\delta,h^{\e}}(\omega))\Big]\right\|_{\H}\nonumber\\&\qquad\times \|\F(\X_{k\Delta}^{\e,\delta,h^{\e}}(\omega),\z)-\bar \F(\X_{k\Delta}^{\e,\delta,h^{\e}}(\omega))\|_{\H}\Big|_{\z=\Y_{r}^{\X_{k\Delta}^{\e,\delta,h^{\e}}(\omega),\widehat\Y_{k\Delta}^{\e,\delta}(\omega)}(\wi\omega)}\wi{\mathbb{P}}(\d\wi{\omega})\mathbb{P}(\d\omega)\nonumber\\&\leq C_{\mu,\alpha,\lambda_1,L_{\G}}\int_{\Omega}\int_{\wi{\Omega}}\left[1+\|\X_{k\Delta}^{\e,\delta,h^{\e}}(\omega)\|_{\H}+\|\Y_{r}^{\X_{k\Delta}^{\e,\delta,h^{\e}}(\omega),\widehat\Y_{k\Delta}^{\e,\delta}(\omega)}(\wi\omega)\|_{\H}\right]e^{-\frac{(s-r)\zeta}{2}}\nonumber\\&\qquad \times\left[1+\|\X_{k\Delta}^{\e,\delta,h^{\e}}(\omega)\|_{\H}+\|\Y_{r}^{\X_{k\Delta}^{\e,\delta,h^{\e}}(\omega),\widehat\Y_{k\Delta}^{\e,\delta}(\omega)}(\wi\omega)\|_{\H}\right]\wi{\mathbb{P}}(\d\wi{\omega})\mathbb{P}(\d\omega)\nonumber\\&\leq C_{\mu,\alpha,\lambda_1,L_{\G}}\int_{\Omega}\left(1+\|\X_{k\Delta}^{\e,\delta,h^{\e}}(\omega)\|_{\H}^2+\|\widehat\Y_{k\Delta}^{\e,\delta}(\omega)\|_{\H}^2\right)\mathbb{P}(\d\omega)e^{-\frac{(s-r)\zeta}{2}}\nonumber\\&\leq C_{\mu,\alpha,\lambda_1,L_{\G},M,T}(1+\|\x\|_{\H}^2+\|\y\|_{\H}^2)e^{-\frac{(s-r)\zeta}{2}}. 
	\end{align}
	Using \eqref{3114} in \eqref{3107}, we obtain 
	\begin{align}\label{3115}
		\E\left[\sup\limits_{t\in[0,T]}|I_1(t)|\right]&\leq\frac{C_{\mu,\alpha,\lambda_1,L_{\G},M,T}(1+\|\x\|_{\H}^2+\|\y\|_{\H}^2)\delta}{\Delta}\left[\int_0^{\frac{\Delta}{\delta}}\int_r^{\frac{\Delta}{\delta}}e^{-\frac{(s-r)\zeta}{2}}\d s\d r\right]^{1/2}\nonumber\\&={C_{\mu,\alpha,\lambda_1,L_{\G},M,T}(1+\|\x\|_{\H}^2+\|\y\|_{\H}^2)}\frac{\delta}{\Delta}\left(\frac{2}{\zeta}\right)^{1/2}\left[\frac{\Delta}{\delta}+\frac{2}{\zeta}\left(e^{-\frac{\Delta\zeta}{2\delta}}-1\right)\right]^{1/2}\nonumber\\&\leq C_{\mu,\alpha,\lambda_1,L_{\G},L_{\sigma_2},M,T}(1+\|\x\|_{\H}^2+\|\y\|_{\H}^2)\left[\left(\frac{\delta}{\Delta}\right)^{1/2}+\frac{\delta}{\Delta}\right]. 
	\end{align}
	Combining the estimates \eqref{4129} and \eqref{3115}, we get 
		\begin{align}\label{3116}
	\E\left[\sup\limits_{t\in[0,T\wedge\wi\tau_R^{\e}]}|I(t)|\right]&\leq C_{\mu,\alpha,\lambda_1,L_{\G},L_{\sigma_2},M,T}(1+\|\x\|_{\H}^2+\|\y\|_{\H}^2)\left[\left(\frac{\delta}{\Delta}\right)^{1/2}+\frac{\delta}{\Delta}+\Delta^{1/2}\right]. 
	\end{align}
	Using \eqref{3116} in \eqref{3104}, we finally find 
	\begin{align}\label{3z25}
	&	\E\left[	\sup_{t\in[0,T\wedge\wi\tau_R^{\e}]}	\|\Z^{\e}_{t}\|_{\H}^2+\mu\int_0^{T\wedge\wi\tau_R^{\e}}\|\Z^{\e}_t\|_{\V}^2\d t+\frac{\beta}{2^{r-2}}\int_0^{T\wedge\wi\tau_R^{\e}}\|\Z^{\e}_t\|_{\wi\L^{r+1}}^{r+1}\d t\right]\nonumber\\&\leq C_{\mu,\alpha,\beta,\lambda_1,L_{\G},L_{\sigma_2},M,T,R}(1+\|\x\|_{\H}^3+\|\y\|_{\H}^3)\left[\left(\frac{\delta}{\e}\right)+\Delta^{1/4}\right]\nonumber\\&\quad+\E\left[\int_0^{T}\|\sigma_1(\bar\X^{h}_t)\Q_1^{1/2}(h^{\e}_t-h_t)\|_{\H}^2\d t\right]+\e^2T\left(1+\sup_{t\in[0,T]}\|\bar\X_{t}^h\|_{\H}^2\right)\nonumber\\&\quad+C_{\mu,\alpha,\beta,\lambda_1,L_{\G},L_{\sigma_2},T}(1+\|\x\|_{\H}^2+\|\y\|_{\H}^2)\left[\left(\frac{\delta}{\Delta}\right)^{1/2}+\frac{\delta}{\Delta}+\Delta^{1/2}\right]\nonumber\\&\leq C_{R,\mu,\alpha,\beta,\lambda_1,L_{\G},L_{\sigma_2},T}(1+\|\x\|_{\H}^3+\|\y\|_{\H}^3)\left[\e^2+\left(\frac{\delta}{\e}\right)+\left(\frac{\delta}{\Delta}\right)^{1/2}+\Delta^{1/4}\right]\nonumber\\&\quad+\E\left[\int_0^{T}\|\sigma_1(\bar\X^{h}_t)\Q_1^{1/2}(h^{\e}_t-h_t)\|_{\H}^2\d t\right], 
	\end{align}
	and by choosing $\Delta=\delta^{1/2}$, we obtain   \eqref{389}.

		\vskip 2 mm
	\noindent	\textbf{Case 2: $n=2,3$ and $r\in(3,\infty)$.} 	Let us now discuss the case $n=3$ and $r\in(3,\infty)$. We just need to estimate the term $2|\langle(\B(\X^{\e,\delta,h^{\e}})-\B(\bar{\X}^{h})),\Z^{\e}\rangle|$ only to get the required result. 	From the estimate \eqref{2.23}, we easily have 
	\begin{align}\label{2z27}
	-2	\beta	\langle\mathcal{C}(\X^{\e,\delta,h^{\e}})-\mathcal{C}(\bar{\X}^{h}),\Z^{\e}\rangle \leq- \beta\||\X^{\e,\delta,h^{\e}}|^{\frac{r-1}{2}}\Z^{\e}\|_{\H}^2- \beta\||\bar{\X}^{h}|^{\frac{r-1}{2}}\Z^{\e}\|_{\H}^2. 
	\end{align}
	Using H\"older's and Young's inequalities, we estimate the term  $2|\langle(\B(\X^{\e,\delta,h^{\e}})-\B(\bar{\X}^{h})),\Z^{\e}\rangle|=2|\langle\B(\Z^{\e},\X^{\e,\delta,h^{\e}}),\Z^{\e}\rangle |$ as  
	\begin{align}\label{2z28}
	2|\langle\B(\Z^{\e},\X^{\e,\delta,h^{\e}}),\Z^{\e}\rangle |&\leq 2\|\Z^{\e}\|_{\V}\|\X^{\e,\delta,h^{\e}}\Z^{\e}\|_{\H}\leq\mu\|\Z^{\e}\|_{\V}^2+\frac{1}{\mu }\|\X^{\e,\delta,h^{\e}}\Z^{\e}\|_{\H}^2.
	\end{align}
	We take the term $\|\X^{\e,\delta,h^{\e}}\Z^{\e}\|_{\H}^2$ from \eqref{2z28} and perform a similar calculation in \eqref{6.30} to deduce
	\begin{align}\label{2z29}
	\|\X^{\e,\delta,h^{\e}}\Z^{\e}\|_{\H}^2\leq\frac{\beta\mu}{2}\left(\int_{\mathcal{O}}|\X^{\e,\delta,h^{\e}}(x)|^{r-1}|\Z^{\e}(x)|^2\d x\right)+\frac{r-3}{r-1}\left(\frac{4}{\beta\mu (r-1)}\right)^{\frac{2}{r-3}}\left(\int_{\mathcal{O}}|\Z^{\e}(x)|^2\d x\right),
	\end{align}
	for $r>3$. Combining \eqref{2z27}, \eqref{2z28}  and \eqref{2z29}, we obtain 
	\begin{align}\label{2z30}
	&-2	\beta	\langle\mathcal{C}(\X^{\e,\delta,h^{\e}})-\mathcal{C}(\bar{\X}^{h}),\Z^{\e}\rangle-2\langle\B(\Z^{\e},\X^{\e,\delta,h^{\e}}),\Z^{\e}\rangle\nonumber\\&\leq -\frac{\beta}{2}\||\X^{\e,\delta,h^{\e}}|^{\frac{r-1}{2}}\Z^{\e}\|_{\H}^2- \frac{\beta}{2}\||\bar{\X}^{h}|^{\frac{r-1}{2}}\Z^{\e}\|_{\H}^2+ \frac{r-3}{\mu(r-1)}\left(\frac{4}{\beta\mu (r-1)}\right)^{\frac{2}{r-3}}\|\Z^{\e}\|_{\H}^2.
	\end{align}
	Thus a calculation similar to the estimate \eqref{3.99} yields 
	\begin{align}
&	\|\Z^{\e}_{t}\|_{\H}^2+\mu\int_0^t\|\Z^{\e}_s\|_{\V}^2\d s+\frac{\beta}{2^{r-1}}\int_0^t\|\Z^{\e}_s\|_{\wi\L^{r+1}}^{r+1}\d s\nonumber\\&\leq \left[\frac{r-3}{\mu(r-1)}\left(\frac{4}{\beta\mu (r-1)}\right)^{\frac{4}{r-3}}+C_{\mu,\alpha,\lambda_1,L_{\G},L_{\sigma_2}}\right]\int_0^t\|\Z^{\e}_s\|_{\H}^2\d s\nonumber\\&\quad+C_{\mu,\alpha,\lambda_1,L_{\G},L_{\sigma_2}}\int_0^t\|\X_{s}^{\e,\delta,h^{\e}}-\X_{s(\Delta)}^{\e,\delta,h^{\e}}\|_{\H}^2\d s+C_{\mu,\alpha,\lambda_1,L_{\G},L_{\sigma_2}}\int_0^t\|\Y_{s}^{\e,\delta,h^{\e}}-\widehat\Y_{s}^{\e,\delta}\|_{\H}^2\d s\nonumber\\&\quad+C\int_0^{t}\|h^{\e}_s\|_{\H}^2\|\Z^{\e}_s\|_{\H}^2\d s +\int_0^{t}\|\sigma_1(\bar\X^{h}_s)\Q_1^{1/2}(h^{\e}_s-h_s)\|_{\H}^2\d s+C\e^2\int_0^{t}\left(1+\|\bar\X_{s}^h\|_{\H}^2\right)\d s\nonumber\\&\quad+C\left(\int_0^t(1+\|\X_{s(\Delta)}^{\e,\delta,h^{\e}}\|_{\H}^2+\|\widehat\Y_{s}^{\e,\delta})\|_{\H}^2)\d s\right)^{1/2}\left(\int_0^t(\|\X^{\e,\delta,h^{\e}}_{s}-\X^{\e,\delta,h^{\e}}_{s(\Delta)}\|_{\H}^2+\|\bar{\X}^{h}_s-\bar{\X}^{h}_{s(\Delta)}\|_{\H}^2)\d s\right)^{1/2} \nonumber\\&\quad +2\int_0^t(\F(\X_{s(\Delta)}^{\e,\delta,h^{\e}},\widehat\Y_{s}^{\e,\delta})-\bar \F(\X_{s(\Delta)}^{\e,\delta,h^{\e}}),\Z^{\e}_{s(\Delta)})\d s\nonumber\\&\quad+2\sqrt{\e}\int_0^t([\sigma_1(\X_{s}^{\e,\delta,h^{\e}})-\sigma_1(\bar{\X}^{h}_s)]\Q_1^{1/2}\d\W_s,\Z^{\e}_s), 
	\end{align}
for all $t\in[0,T]$, $\mathbb{P}$-a.s. Applying Gronwall's inequality and then using a calculation similar to \eqref{3z25} yields	 the estimate \eqref{390} for the case $n=2,3$ and $r\in(3,\infty)$. 
	
	\vskip 2 mm
\noindent	\textbf{Case 3: $n=r=3$ and $2\beta\mu>1$.} 	
	For $n=r=3$, 	from \eqref{2.23}, we have 
	\begin{align}\label{2z31}
	-2	\beta	\langle\mathcal{C}(\X^{\e,\delta,h^{\e}})-\mathcal{C}(\bar{\X}^{h}),\Z^{\e}\rangle \leq- \beta\|\X^{\e,\delta,h^{\e}}\Z^{\e}\|_{\H}^2- \beta\|\bar{\X}^{h}\Z^{\e}\|_{\H}^2, 
	\end{align}
	and a calculation similar to \eqref{2z28} gives 
	\begin{align}\label{3z72}
	2|\langle\B(\Z^{\e},\X^{\e,\delta,h^{\e}}),\Z^{\e}\rangle |&\leq 2\|\Z^{\e}\|_{\V}\|\X^{\e,\delta,h^{\e}}\Z^{\e}\|_{\H}\leq \theta\mu\|\Z^{\e}\|_{\V}^2+\frac{1}{\theta\mu }\|\X^{\e,\delta,h^{\e}}\Z^{\e}\|_{\H}^2.
	\end{align}
	Combining \eqref{2z31} and \eqref{3z72}, we obtain 
	\begin{align}\label{3125}
	&-2\mu\langle\A\Z^{\e},\Z^{\e}\rangle-2	\beta	\langle\mathcal{C}(\X^{\e,\delta,h^{\e}})-\mathcal{C}(\bar{\X}^{h}),\Z^{\e}\rangle-2\langle\B(\Z^{\e},\X^{\e,\delta,h^{\e}}),\Z^{\e}\rangle\nonumber\\&\leq (2-\theta)\mu\|\Z^{\e}\|_{\V}^2 -\left(\beta-\frac{1}{\theta\mu }\right)\|\X^{\e,\delta,h^{\e}}\Z^{\e}\|_{\H}^2- \beta\|\bar{\X}^{h}\Z^{\e}\|_{\H}^2,
	\end{align}
for $\frac{1}{\beta\mu}<\theta<2$	and the hence estimate \eqref{390} follows for  $2\beta\mu> 1$. 

Note that by making use of the estimates \eqref{2z30} and \eqref{3125}, we don't need the stopping time defined in \eqref{stop1} to get the required estimate \eqref{390}. 
\end{proof}
Let us now establish the weak convergence result. The well known Skorokhod's representation theorem (see \cite{AVS}) states that if $\mu_n, n=1,2,\ldots,$ and $\mu_0$ are  probability measures on complete separable metric space (Polish space)  such that  $\mu_n\xrightarrow{w}\mu$, as $ n \to \infty$, then there exist a probability space $(\widetilde{\Omega},\widetilde{\mathscr{F}},\widetilde{\mathbb{P}})$ and a sequence of measurable random elements $\mathrm{X}_n$ such that $\mathrm{X}_n\to\mathrm{X},$ $\widetilde{\mathbb{P}}$-a.s., and $\mathrm{X}_n$ has the distribution function $\mu_n$, $n = 0, 1, 2, \ldots $  ($\mathrm{X}_n\sim\mu_n$), that is, the law of $\mathrm{X}_n$ is $\mu_n$. We use Skorokhod's representation theorem in the next theorem. 
 \begin{theorem}[Weak convergence]\label{weak}
	Let $\big\{h^{\e} : \e > 0\big\}\subset \mathcal{A}_M$ converges in distribution to $h$ with respect to the weak topology on $\mathrm{L}^2(0,T;\H)$. Then $ \mathcal{G}^{\e}\left(\W_{\cdot} +\frac{1}{\sqrt{\e}}\int_0^{\cdot}h^{\e}_s\d s\right)$ converges in distribution to $\mathcal{G}^0\left(\int_0^{\cdot}h_s\d s\right)$ in $\mathscr{E}$, as $\e\to0$.
\end{theorem}
\begin{proof}
	Let $\{h_{\e}\}$ converges to $h$ in distribution as random elements taking values in $\mathcal{S}_M,$ where $\mathcal{S}_M$ is equipped with the weak topology. 
	Since $\mathcal{A}_M$ is  Polish (see section \ref{sec4.2} and \cite{BD2}) and $\big\{h^{\e} : \e > 0\big\}\subset \mathcal{A}_M$ converges in distribution to $h$ with respect to the weak topology on $\mathrm{L}^2(0,T;\H)$, the Skorokhod representation theorem can be used  to construct a probability space $(\widetilde{\Omega},\widetilde{\mathscr{F}},(\widetilde{\mathscr{F}}_t)_{0\leq t\leq T},\widetilde{\mathbb{P}})$  and processes $(\wi h^{\e},\wi h, \wi\W^{\e})$ such that the distribution of $(\wi h^{\e}, \wi\W^{\e})$
	is same as that of $(h^{\e}, \W^{\e})$, and $\wi h^{\e}\to \wi h$,  $\widetilde{\mathbb{P}}$-a.s., in the weak topology of $\mathcal{S}_M$. Thus
	$\int_0^{t} \wi h^{\e}(s)\d s\to\int_0^{t} \wi h(s)\d s$ weakly in $\H$, $\widetilde{\mathbb{P}}$-a.s., for all $t\in[0,T]$. In the following sequel, without loss of 
	generality, we write $(\Omega,\mathscr{F},\mathbb{P})$ as the probability space and $(h^{\e},h,\W)$ as processes, though strictly speaking, one should write $(\widetilde{\Omega},\widetilde{\mathscr{F}},\widetilde{\mathbb{P}})$ and $(\wi h^{\e}, \wi h, \widetilde{\W}^{\e})$, respectively for probability space and processes.	
	
	Let us define $\Z^{\e}_t:=\X_{t}^{\e,\delta,h^{\e}}-\bar{\X}_{t}^h$, where $\Z^{\e}_t$ satisfies the stochastic differential given in \eqref{3.91}. We first prove the Theorem for $n=2$ and $r\in[1,3]$. Let $\wi\tau_R^{\e}$ be the stopping time defined in \eqref{stop1}. Then for any $\eta>0$, by using Markov's inequality, we have 
	\begin{align}\label{3130}
	&\mathbb{P}\left\{\left(\sup_{t\in[0,T]}\|\Z^{\e}_t\|_{\H}^2+\int_0^T\|\Z^{\e}_t\|_{\V}^2\d t+\int_0^T\|\Z^{\e}_t\|_{\wi\L^{r+1}}^{r+1}\d t\right)^{1/2}>\eta\right\}\nonumber\\&\leq\frac{1}{\eta}\E\left[\left(\sup_{t\in[0,T]}\|\Z^{\e}_t\|_{\H}^2+\int_0^T\|\Z^{\e}_t\|_{\V}^2\d t+\int_0^T\|\Z^{\e}_t\|_{\wi\L^{r+1}}^{r+1}\d t\right)^{1/2}\right]\nonumber\\&=\frac{1}{\eta} \E\left[\left(\sup_{t\in[0,T]}\|\Z^{\e}_t\|_{\H}^2+\int_0^T\|\Z^{\e}_t\|_{\V}^2\d t+\int_0^T\|\Z^{\e}_t\|_{\wi\L^{r+1}}^{r+1}\d t\right)^{1/2}\chi_{\{T\leq\wi\tau_R^{\e}\}}\right]\nonumber\\&\quad+\frac{1}{\eta}\E\left[\left(\sup_{t\in[0,T]}\|\Z^{\e}_t\|_{\H}^2+\int_0^T\|\Z^{\e}_t\|_{\V}^2\d t+\int_0^T\|\Z^{\e}_t\|_{\wi\L^{r+1}}^{r+1}\d t\right)^{1/2}\chi_{\{T>\wi\tau_R^{\e}\}}\right],
	\end{align}
where  $\chi_{t}$ is the indicator function.	From Lemma \ref{lem3.14} (see \eqref{389}),  we obtain 
	\begin{align}\label{3131}
& \E\left[\left(\sup_{t\in[0,T]}\|\Z^{\e}_t\|_{\H}^2+\int_0^T\|\Z^{\e}_t\|_{\V}^2\d t+\int_0^T\|\Z^{\e}_t\|_{\wi\L^{r+1}}^{r+1}\d t\right)^{1/2}\chi_{\{T\leq\wi\tau_R^{\e}\}}\right]\nonumber\\&\leq \E\left[\left(\sup_{t\in[0,T]}\|\Z^{\e}_t\|_{\H}^2+\int_0^T\|\Z^{\e}_t\|_{\V}^2\d t+\int_0^T\|\Z^{\e}_t\|_{\wi\L^{r+1}}^{r+1}\d t\right)\chi_{\{T\leq\wi\tau_R^{\e}\}}\right]^{1/2}\nonumber\\&\leq \bigg\{C_{\mu,\alpha,\beta,\lambda_1,L_{\G},L_{\sigma_2},M,T,R}(1+\|\x\|_{\H}^3+\|\y\|_{\H}^3)\left[\e^2+\left(\frac{\delta}{\e}\right)+\delta^{1/8}\right]\nonumber\\&\quad+\E\left[\int_0^{T}\|\sigma_1(\bar\X^{h}_t)\Q_1^{1/2}(h^{\e}_t-h_t)\|_{\H}^2\d t\right]\bigg\}^{1/2}. 
	\end{align}	
Using H\"older's and Markov’s inequalities,  \eqref{5.5y} and \eqref{5.5z}, we estimate the second term from the right hand side of the inequality \eqref{3130} as
\begin{align}\label{3132}
&\E\left[\left(\sup_{t\in[0,T]}\|\Z^{\e}_t\|_{\H}^2+\int_0^T\|\Z^{\e}_t\|_{\V}^2\d t+\int_0^T\|\Z^{\e}_t\|_{\wi\L^{r+1}}^{r+1}\d t\right)^{1/2}\chi_{\{T>\wi\tau_R^{\e}\}}\right]\nonumber\\&\leq \left[\E\left(\sup_{t\in[0,T]}\|\Z^{\e}_t\|_{\H}^2+\int_0^T\|\Z^{\e}_t\|_{\V}^2\d t+\int_0^T\|\Z^{\e}_t\|_{\wi\L^{r+1}}^{r+1}\d t\right)\right]^{1/2}[\mathbb{P}(T>\wi\tau_R^{\e})]^{1/2}\nonumber\\&\leq\frac{C_{\mu,\alpha,\lambda_1,L_{\G},M,T}(1+\|\x\|_{\H}^2+\|\y\|_{\H}^2)^{1/2}}{\sqrt{R}}\left[\E\left(\int_0^T\|\X_{s}^{\e,\delta,h^{\e}}\|_{\V}^2\d s\right)\right]^{1/2}\nonumber\\&\leq \frac{C_{\mu,\alpha,\lambda_1,L_{\G},M,T}(1+\|\x\|_{\H}^2+\|\y\|_{\H}^2)}{\sqrt{R}}. 
\end{align} 
Combining \eqref{3131}-\eqref{3132} and substitute it in \eqref{3130}, we find 
\begin{align}\label{3133}
&\mathbb{P}\left\{\left(\sup_{t\in[0,T]}\|\Z^{\e}_t\|_{\H}^2+\int_0^T\|\Z^{\e}_t\|_{\V}^2\d t+\int_0^T\|\Z^{\e}_t\|_{\wi\L^{r+1}}^{r+1}\d t\right)^{1/2}>\eta\right\}\nonumber\\&\leq \frac{1}{\eta}\bigg\{C_{\mu,\alpha,\beta,\lambda_1,L_{\G},L_{\sigma_2},M,T,R}(1+\|\x\|_{\H}^3+\|\y\|_{\H}^3)\left[\e^2+\left(\frac{\delta}{\e}\right)+\delta^{1/8}\right]\nonumber\\&\quad+\E\left[\int_0^{T}\|\sigma_1(\bar\X^{h}_t)\Q_1^{1/2}(h^{\e}_t-h_t)\|_{\H}^2\d t\right]\bigg\}^{1/2}+\frac{C_{\mu,\alpha,\lambda_1,L_{\G},M,T}(1+\|\x\|_{\H}^2+\|\y\|_{\H}^2)}{\eta\sqrt{R}}.
\end{align}
Once again, we use the fact that compact operators maps weakly convergent sequences into strongly convergent sequences. Since $\sigma_1(\cdot)$ is compact and $\big\{h^{\e} : \e > 0\big\}\subset \mathcal{A}_M$ converges in distribution to $h$ with respect to the weak topology on $\mathrm{L}^2(0,T;\H)$, we get  
\begin{align*}
\int_0^{T}\|\sigma_1(\bar\X^{h}_t)\Q_1^{1/2}(h^{\e}_t-h_t)\|_{\H}^2\d t\to 0, \ \text{ as }\ \e\to 0,\ \mathbb{P}\text{-a.s.}
\end{align*}
Furthermore, we have 
\begin{align*}
\E\left[\int_0^{T}\|\sigma_1(\bar\X^{h}_t)\Q_1^{1/2}(h^{\e}_t-h_t)\|_{\H}^2\d t\right]&\leq\E\left[\int_0^T\|\sigma_1(\bar\X^{h}_t)\Q_1^{1/2}\|_{\mathcal{L}_2}^2\|h^{\e}_t-h_t\|_{\H}^2\d t\right]\nonumber\\&\leq C\sup_{t\in[0,T]}(1+\|\bar\X^{h}_t\|_{\H}^2)\E\left[\int_0^T\|h_t^{\e}\|_{\H}^2+\int_0^T\|h_t\|_{\H}^2\right]\nonumber\\&\leq C_{\mu,\alpha,\lambda_1,L_{\G},L_{\sigma_2},M,T}\left(1+\|\x\|_{\H}^2\right), 
\end{align*}
where we used \eqref{5.5y} and the fact that $h^{\e} \in\mathcal{A}_M$ and $h\in\mathcal{S}_M$. Thus, an application of the dominated convergence theorem gives  
\begin{align}\label{461}
\E\left[\int_0^{T}\|\sigma_1(\bar\X^{h}_t)\Q_1^{1/2}(h^{\e}_t-h_t)\|_{\H}^2\d t\right]\to 0, \ \text{ as }\ \e\to 0. 
\end{align}
Letting $\e\to 0$ and then $R\to\infty$ in \eqref{3133}, we obtain
	\begin{align}\label{3134}
&\lim_{\e\to 0}\mathbb{P}\left\{\left(\sup_{t\in[0,T]}\|\Z^{\e}_t\|_{\H}^2+\int_0^T\|\Z^{\e}_t\|_{\V}^2\d t+\int_0^T\|\Z^{\e}_t\|_{\wi\L^{r+1}}^{r+1}\d t\right)^{1/2}>\eta\right\},
\end{align}
for all $\eta>0$,  which completes the proof for the case $n=2$ and $r\in[1,3]$. 

For $n=2$, $r\in(3,\infty)$ and $n=3$, $r\in[3,\infty)$ ($2\beta\mu> 1,$ for $r=3$), one has to use the stopping time $\tau^{\e}_R$ defined in \eqref{stop} for the estimate in \eqref{3130} and then use the estimate \eqref{390}  to obtain the convergence given in \eqref{3134}.  
\end{proof}

\begin{remark}\label{rem4.22}
	Note that LDP for $\{\X^{\e,\delta}\}$ proved in Theorem \ref{thm4.14} holds true for the case $n=2$ and $r\in[1,3]$ with $\alpha=\beta=0$. Therefore, the results obtained in this work is valid for 2D Navier-Stokes equations also. In this case, the state space becomes $\mathscr{E}=\C([0,T];\H)\cap\mathrm{L}^2(0,T;\V)$ (see Remark \ref{rem3.4}). 
\end{remark}

 \medskip\noindent
{\bf Acknowledgments:} M. T. Mohan would  like to thank the Department of Science and Technology (DST), India for Innovation in Science Pursuit for Inspired Research (INSPIRE) Faculty Award (IFA17-MA110).

\end{document}